\documentclass[10pt,a4paper]{article}

\usepackage{cite}

\usepackage{graphicx}

\usepackage{amssymb,amsmath,amsthm}
\usepackage{esint}
\usepackage{tikz}
\usepackage{bbm,dsfont,bm}
\usepackage{color}

\theoremstyle{definition}

\newtheorem{theorem}[equation]{Theorem}
\newtheorem{lemma}[equation]{Lemma}
\newtheorem{corollary}[equation]{Corollary}

\newtheorem{definition}[equation]{Definition}

\newtheorem{remark}[equation]{Remark}
\newtheorem{proposition}[equation]{Proposition}

\usepackage{times}

\usepackage{graphicx}

\usepackage{amsmath,amsthm,amssymb,enumerate, fancyhdr}
\usepackage[english]{babel}
\usepackage[backref=page]{hyperref}
\usepackage{thmtools}

\newtheorem*{theorem*}{Theorem}

\newcommand{\eps}{\epsilon}

\newcommand{\R}{\mathbb{R}}

\numberwithin{equation}{section}

\usepackage{graphicx}
\usepackage{tikz}

\usepackage[left=2cm,right=2cm,top=2cm,bottom=2cm]{geometry}

\usepackage{authblk}
\begin{document}

\title{Multilayered fluid-structure interactions: existence of weak solutions for time-periodic  and initial-value problems}


\author[1,3]{Claudiu M\^{i}ndril\u{a}\thanks{\href{mailto:cmindrila@bcamath.org}{cmindrila@bcamath.org}}}
\author[1,2]{Arnab Roy\thanks{\href{mailto:aroy@bcamath.org}{aroy@bcamath.org}}}

\affil[1]{BCAM -- Basque Center for Applied Mathematics,
Mazarredo 14, E48009 Bilbao, Bizkaia, Spain}

\affil[2]{IKERBASQUE, Basque Foundation for Science,
Plaza Euskadi 5, 48009 Bilbao, Bizkaia, Spain}
\affil[3]{Gheorghe Mihoc – Caius Iacob Institute of Mathematical
Statistics and Applied Mathematics of the Romanian Academy, No.13 Calea 13 Septembrie, Sector 5, 050711 Bucharest, ROMANIA, P.O. BOX 1-24}

\date{\today}

\maketitle






\begin{abstract}
We establish the existence of weak solutions for a class of fully coupled multilayered fluid–structure interaction systems in a \emph{three-dimensional} spatial setting. The model consists of an incompressible viscous fluid interacting with a thin elastic shell, which is in turn coupled to a three-dimensional elastic solid, yielding a nonstandard $3D/2D/3D$ coupling configuration. The system is driven by time-periodic boundary forcing through Bernoulli-type pressure conditions. For sufficiently small forcing in $L^2$, we prove the existence of at least one time-periodic weak solution. A central analytical difficulty stems from the strong nonlinear coupling across interfaces of different dimensionality and the absence of classical compactness mechanisms. This challenge is overcome through a carefully designed energy framework combined with  and new  $L^{2}$ compactness arguments adapted to the multilayered geometry.

A key structural assumption is the viscoelasticity of the three-dimensional solid, which yields additional diffusion estimates and ensures energy stability. In the purely elastic case, we establish the global-in-time existence of weak solutions to the corresponding initial-value problem, provided that no degeneration (self-contact) of the fluid domain occurs. These results extend existing theories for two-dimensional and reduced-dimensional configurations to a genuinely three-dimensional multilayered setting, providing new analytical insight into complex coupled PDE systems arising in fluid–structure interaction.
\end{abstract}

\section{Introduction}



In this article, we investigate the existence of solutions for  \emph{time-periodic} and \emph{initial value  (Cauchy)} problems for a class of systems modelling the interaction between a viscous incompressible fluid and a multi-layered elastic structure consisting of thin and thick layers.
Our main motivation to study these systems is mainly due to bio-medical phenomena.

The study of time-periodic solutions is motivated by blood flow in arteries, where the rhythmic contraction and relaxation of the heart generates pulsatile flow. After transient effects have decayed, the pressure, velocity, and structural displacement are expected to repeat from one cardiac cycle to the next. Time-periodic solutions therefore provide a natural mathematical framework for describing the long-term dynamics of the coupled fluid--structure interaction.

\paragraph{Moving domains.}
We assume that the initial configuration $\Omega_{F}$ is a cylinder of radius $R$ and length $L$. In cylindrical coordinates, it is given by 
\[
\Omega_{F}:=\left\{ \left(x,y,z\right)\in\mathbb{R}^{3}:x=r\cos\theta,\ y=r\sin\theta,\ \left(r,\theta,z\right)\in[0,R)\times[0,2\pi)\times(0,L)\right\} .
\]
The \emph{inflow} and \emph{outflow} parts of the boundary will be denoted by
\[
\Gamma_{in}:=\left\{ \left(r\cos\theta,r\sin\theta,0\right):r\in\left[0,R\right),\ \theta\in\left[0,2\pi\right)\right\} 
\]
and 
\[
\Gamma_{out}:=\left\{ \left(r\cos\theta,r\sin\theta,L\right):r\in\left[0,R\right),\ \theta\in\left[0,2\pi\right)\right\} 
\]
while the \emph{flexible} part of the boundary is given by
\[
\Gamma=\left\{ \left(R\cos\theta,R\sin\theta,z\right):\left(\theta,z\right)\in\left[0,2\pi\right)\times\left(0,L\right)\right\}.
\]

Here $\Gamma$ can be parametrized by a set $\omega \subset \mathbb{R}^{2}$ and a mapping $\varphi : \omega\mapsto\mathbb{R}^{3}$ given by
\[
\omega:=\left[0,2\pi\right)\times\left(0,L\right),\quad\varphi\left(\theta,z\right)=\left(R\cos\theta,R\sin\theta,z\right).
\]
Let $I:=[0,T]$ denote the time horizon interval for $T>0$. A thin elastic shell is attached to $\Gamma$, and its motion occurs in the direction of the unit outer normal vector $\mathbf{e}_{r}=\mathbf{e}_{r}\left(\theta\right)=\left(\cos\theta,\sin\theta,0\right)$ for $\theta \in \left[0,2\pi\right)$. The displacement of the shell will be described  by the mapping
$\eta:I\times\omega\mapsto\mathbb{R}$, which represents only the radial and non-zero component of the full displacement vector $\boldsymbol{\eta}:\omega\mapsto\mathbb{R}^{3}$, $\boldsymbol{\eta}=\eta \mathbf{e}_{r}$.
The fluid domains $\Omega_{F} ^{\eta} (t)$ evolve with each $t\in I$. Thus, for each $t\in I$,  we obtain the fluid domains 
\[
\Omega_{F}^{\eta}\left(t\right):=\left\{ \left(r\cos\theta,r\sin\theta,z\right)\in\mathbb{R}^{3}:r<R+\eta\left(t,\theta,z\right),\ \left(\theta,z\right)\in\omega\right\} ,
\]
and their corresponding moving boundaries 
\[
\Gamma^{\eta}\left(t\right):=\left\{ \left(r\cos\theta,r\sin\theta,z\right)\in\mathbb{R}^{3}:r=R+\eta\left(t,\theta,z\right),\ \left(\theta,z\right)\in\omega\right\} .
\]
The situation is sketched in the following figure.
\begin{figure}[!htbp]
    \centering
    \caption{A 2D section of the multilayered interaction at some $t\in I$}
    \label{Figure:fluid-structure}

\begin{tikzpicture}[scale=1.0, font=\small]

\def\L{8}    
\def\R{1.5}  
\def\H{0.55} 
\def\A{0.28} 

\colorlet{Fc}{blue!10}
\colorlet{Sc}{orange!22}
\colorlet{Gc}{blue!60!black}

\fill[Sc]
  plot[variable=\x, domain=0:\L, samples=200, smooth]
    (\x, {\R + \A*sin(deg(2*pi*\x/\L))})
  -- (\L,{\R+\H}) -- (0,{\R+\H}) -- cycle;

\fill[Sc]
  plot[variable=\x, domain=0:\L, samples=200, smooth]
    (\x, {-\R - \A*sin(deg(2*pi*\x/\L))})
  -- (\L,{-\R-\H}) -- (0,{-\R-\H}) -- cycle;

\fill[Fc]
  plot[variable=\x, domain=0:\L, samples=200, smooth]
    (\x, {\R + \A*sin(deg(2*pi*\x/\L))})
  -- plot[variable=\x, domain=\L:0, samples=200, smooth]
    (\x, {-\R - \A*sin(deg(2*pi*\x/\L))})
  -- cycle;

\draw[thick] (0,{\R+\H}) -- (\L,{\R+\H})
             (0,{-\R-\H}) -- (\L,{-\R-\H})
             (0,{\R+\H}) -- (0,{-\R-\H})
             (\L,{\R+\H}) -- (\L,{-\R-\H});

\draw[Gc, thick]
  plot[variable=\x, domain=0:\L, samples=200, smooth]
    (\x, {\R + \A*sin(deg(2*pi*\x/\L))});
\draw[Gc, thick]
  plot[variable=\x, domain=0:\L, samples=200, smooth]
    (\x, {-\R - \A*sin(deg(2*pi*\x/\L))});

\node[blue!50!black]    at ({\L/2}, 0)            {$\Omega_F^\eta(t)$};
\node[orange!55!black]  at ({\L/2},  {\R+\H*.5})  {$\Omega_S$};
\node[orange!55!black]  at ({\L/2}, {-\R-\H*.5})  {$\Omega_S$};
\node[Gc, above, font=\footnotesize]
  at ({\L*.65}, {\R + \A*sin(deg(2*pi*.65)) + .05})
  {$\Gamma^\eta(t)$};
\node[above, font=\footnotesize]
  at ({\L*.5}, {\R+\H+.03}) {$\Gamma_{\mathrm{ext}}$};
\node[left, font=\footnotesize] at (0, 0) {$\Gamma_{in}$};
\node[right,font=\footnotesize] at (\L,0) {$\Gamma_{out}$};

\draw[<->, gray!70!black, thin]
  ({\L/4},\R) -- ({\L/4},{\R+\A})
  node[midway, right, font=\footnotesize] {$\eta$};

\end{tikzpicture}
\end{figure}

A mapping from $\Omega_{F}$ to $\Omega_{F}^{\eta}(t)$ can be given by the formula
\begin{equation*}
\psi_{\eta}:\Omega\mapsto\Omega_{\eta},\quad\psi_{\eta}\left(r\cos\theta,r\sin\theta,z\right)=\left(\frac{R+\eta}{R}\left(r\cos\theta\right),\frac{R+\eta}{R}\left(r\sin\theta\right),z\right)
\end{equation*}
Depending on the regularity of $\eta$, the mapping $\psi_{\eta}$ is a homeomorphism  and even a $C^{k}$ diffeomorphism provided that $\eta \in C^{k}(\omega)$ with $k\in \mathbb{N}$.

External to the thin membrane and the fluid , we  consider an additional thick elastic solid,  of thickness $H\gg h>0$, where $h$ is the thickness of the thin solid. Typically $h>0$ but is considered very small so that the thin elastic solid is considered to be $2$-dimensional.
In order to describe the thin and thick solids we employ the usual Lagrangian coordinates. The reference configuration of the thick  solid is defined as
\[
\Omega_{S}=\left\{ \left(r\cos\theta,r\sin\theta,z\right)\in\mathbb{R}^{3}:R<r<R+H\right\} 
\]
and it evolves in time to $\Omega_{S}(t)$.
We denote its displacement vector  by \[\mathbf{d}:I\times\Omega_{S}\mapsto\mathbb{R}^{3}.\]

Concerning the fluid equations, since blood flow in large vessels is typically modeled by the \emph{incompressible Navier--Stokes equations}, we adopt the same model, as it accounts for the important nonlinear convective effects.

 Therefore, for each $t\in I$ we consider the eulerian fluid velocity field 
 \[
 \mathbf{u}:I\times\Omega_{F}^{\eta}\mapsto\mathbb{R}^{3}
 \]
 fulfilling the system
\begin{equation}\label{eqn:fluid}
\begin{cases}
\rho_{F}\left(\partial_{t}\mathbf{u}+\left(\mathbf{u}\cdot\nabla\right)\mathbf{u}\right)=\mathrm{div}\ \sigma & \text{in}\ \Omega_{F}^{\eta}\left(t\right)\\
\text{div}\ \mathbf{u}=0 & \text{in}\ \Omega_{F}^{\eta}\left(t\right)
\end{cases}
\end{equation}
 Here $\rho_{F}$ denotes the fluid density, which is constant and set to equal 1. The symbol $\sigma$ denotes the usual Cauchy stress tensor, defined by
\begin{equation}
\sigma:=2\mu\mathbb{D}\left(\mathbf{u}\right)-p\mathbb{I},\quad\mathbb{D}\left(\mathbf{u}\right)=\frac{1}{2}\left(\nabla\mathbf{u}+\left(\nabla\mathbf{u}\right)^{T}\right)
\end{equation}
where $\mu$ is the kinematic viscosity of the fluid, also set to 1. The function
$p:I\times\Omega_{F}^{\eta}\left(t\right)\mapsto\mathbb{R}$ denotes the fluid pressure and $\mathbb{D}(\mathbf{u})\in \mathbb{R}_{sym}^{3\times 3}$ is the symmetric part of the gradient.

The vertical \emph{displacement} of the thin structure from the reference configuration is denoted by \[\eta : I \times \omega \mapsto \mathbb{R}\] and we assume is a solution of  the following wave type equation:
\begin{equation}\label{eqn:wave}
\rho_{S}h\partial_{tt}\eta+K^{\prime}\left(\eta\right)=f,\quad\left(t,z\right)\in I\times\omega
\end{equation}
where $\rho_{S}$ denotes  the density of the membrane, $h$ its thickness and we assume for simplicity that $\rho_{S}h=1$.
In \eqref{eqn:wave} the term $K^{\prime}\left(\eta\right)$ denotes the $L^{2}$ gradient of the Koiter elastic energy $K(\eta)$ from \eqref{eqn:koiter}.
In this work we will assume for the simplicity of the exposition that $K^{\prime}\left(\eta\right)=\Delta^{2}\eta$ which would correspond to a linear elastic shell.
Let us give some details. 
Following  \cite[p. 162]{Ciarlet05},
let us consider the elastic \emph{Koiter energy} defined via 
\begin{equation}\label{eqn:koiter}
K\left(\eta\right):=\frac{h}{2}\int_{\omega}\mathcal{A}:\mathbb{G}\left(\eta\right)\otimes\mathbb{G}\left(\eta\right)dA+\frac{h^{3}}{6}\int_{\omega}\mathcal{A}:\mathbb{R}^{\sharp}\left(\eta\right)\otimes\mathbb{R}^{\sharp}\left(\eta\right)dA
\end{equation}
where $\mathcal{A}$ is a 4th order tensor also known in the literature as elasticity (or stiffness) tensor depending on the first fundamental form of $\Gamma$. 
Then, if $\phi : \omega \mapsto \mathbb{R}^{3}$ is a parametrization of $\Gamma$, we may parametrize the moving boundary $\Gamma^{\eta}$ by 
\[
\phi_{\eta}:\omega\mapsto\mathbb{R}^{3},\quad\phi_{\eta}\left(\theta,z\right)=\phi\left(\theta,z\right)+\eta\left(\theta,z\right)\mathbf{e}_{r}\left(\theta\right)
\] 
and define the \emph{change of metric tensor} via
\begin{equation}
\mathbb{G}_{ij}\left(\eta\right):=\partial_{i}\phi_{\eta}\cdot\partial_{j}\phi_{\eta}-\partial_{i}\phi\cdot\partial_{j}\phi
\end{equation}
and the \emph{change of curvature tensor} via 
\begin{equation}
\mathbb{R}_{ij}^{\sharp}\left(\eta\right):=\frac{\partial_{ij}\phi_{\eta}\cdot\left(\partial_{1}\phi_{\eta}\times\partial_{2}\phi_{\eta}\right)}{\left|\partial_{1}\phi\times\partial_{2}\phi\right|}-\partial_{ij}\phi\cdot\mathbf{e}_{r}.
\end{equation}
Following \cite{Ciarlet05} one can linearize the tensors $\mathbb{G}$, $\mathbb{R}^{\sharp}$ to obtain $K^{\prime}\left(\eta\right)=\Delta^{2}\eta+B\eta$ where $B$ is a linear second order differential operator. It is shown in \cite[Thm. 4.4-2]{Ciarlet05} that $K$ is $H^2$ coercive. Thus, 
\[
K\left(\eta\right)\gtrsim\left\Vert \nabla^{2}\eta\right\Vert _{L_{x}^{2}}^{2}.
\]
Since our focus is on the analytical aspects of this system, we may assume
without loss of generality  that  \[
K\left(\eta\right)=\int_{\omega}\left|\nabla^{2}\eta\right|^{2}dA,\quad K^{\prime}\left(\eta\right)=\Delta^{2}\eta.
\]
Concerning the thick elastic structure, we assume it is a homogeneous  elastic material of constant density $\rho_{S_2}$.
 Recall that its displacement is denoted by \[\mathbf{d}:I\times\Omega_{S}\mapsto\mathbb{R}^{3}\] and we assume it fulfills the  following wave-type equation 
\begin{equation}\label{eqn:thick}
\rho_{S_2} \partial_{tt}\mathbf{d}=\text{div} \  \mathbb{S}
\end{equation}
with $\mathbb{S}$ the Lam\'{e} stress tensor given by
\begin{equation}\label{eqn:lame-stress}
\mathbb{S}:=2\lambda_{1}\mathbb{D}\left(\mathbf{d}+\delta\partial_{t}\mathbf{d}\right)+\lambda_{2}\left(\text{div}\ \mathbf{d}\right)\mathbb{I}_{3}
\end{equation}

with $\lambda_{1}$ and $\lambda_{2}$ are the Lam\'{e} constants. Here $\mathbb{D}$ again denotes the symmetric gradient, so $\mathbb{D}\left(\mathbf{v}\right):=\dfrac{\nabla\mathbf{v}+\left(\nabla\mathbf{v}\right)^{T}}{2}$ 
for all $\mathbf{v}\in \mathbb{R}^{3}$.
{The coefficients $\lambda_{1}$, $\lambda_{2}$ will be set to 1 for simplicity, as well as $\rho_{S_2}$. The parameter $\delta\ge 0$ ensures additional dissipation for the thick elastic solid. It is important to note that $\delta>0$ is crucial to obtain the \emph{uniform in time estimates} that are needed to establish time periodicity. We will fix $\delta=1$ for simplicity of the exposition.
}

\paragraph{Boundary and coupling conditions.}

In this work we assume the  \emph{no-slip} coupling conditions between fluid and structures, that is:
\begin{equation}\label{eqn:coupling}
\begin{cases}
\text{tr}_{\Gamma^{\eta}\left(t\right)}\mathbf{u}=\left.\mathbf{u}\right|_{r=R+\eta}=\partial_{t}\eta\ \mathbf{e}_{r}, & \text{on }I\times\omega\\
\text{tr}_{\Gamma}\mathbf{d}:=\left.\mathbf{d}\right|_{r=R}=\eta\mathbf{e}_{r} & \text{on }I\times\omega
\end{cases}
\end{equation}

and also the dynamic coupling condition
\begin{equation}\label{eqn:dynamic-coupling}
\partial_{tt}\eta+\Delta^{2}\eta=-J_{\eta}\left(\left.\sigma\mathbf{n}^{\eta}\right|_{r=R+\eta}\right)\cdot\mathbf{e}_{r}+\left(\left.\mathbb{S}\right|_{r=R}\mathbf{e}_{r}\right)\cdot\mathbf{e}_{r}
\end{equation}
where $\mathbf{e}_{r}\left(\theta\right)= 
(\cos\theta, \sin\theta,0)$ is the outer unit normal vector to $\Gamma$ in the radial direction, $\mathbf{n}^{\eta}$ the outer unit normal at $\Gamma^{\eta}(t)$ and

\[
J_{\eta}\left(t,\theta, z\right):=\sqrt{\left(1+\left(\partial_{z}\eta\right)^{2}\right)\left(R+\eta\right)^{2}+\left(\partial_{\theta}\eta\right)^{2}},
\] represents the Jacobian of the transformation from Eulerian to Lagrangian coordinates.

At the inlet/outlet boundaries denoted $\Gamma_{in}, \ \Gamma_{out}$  of the fluid domain we assume that the fluid velocity is zero in \emph{tangential direction} and a given dynamic pressure condition. \footnote{Sometimes  called \emph{head pressure} or \emph{Bernoulli pressure} in the literature.}
Thus, we denote $\Gamma_{in/out}:=\Gamma_{in} \cup \Gamma_{out}$ and we set
\begin{equation}
   \begin{cases}
p+\frac{\rho_{f}}{2}\left|\mathbf{u}\right|^{2}=P_{in/out}\left(t\right) & \text{on}\ \Gamma_{in/out}\\
\mathbf{u}
\times \mathbf{e}_{z}=\mathbf{0}
& \text{on}\ \Gamma_{in/out},
\end{cases}
\end{equation}
where $\mathbf{e}_{z}=(0,0,1)$ is the normal direction at $\Gamma_{in/out}$. Thus $\mathbf{u}$ has only the $z$ component which could be nonzero. 

Concerning the thin elastic structure, we assume that it is $2\pi$-periodic in the $\theta$ variable and 
\emph{clamped} at the inflow and outflow boundaries, so
\begin{equation}
\left|\left.\eta\left(t,\theta,z\right)\right|_{z=0}\right|=\left|\left.\nabla\eta\left(t,\theta,z\right)\right|_{z=0}\right|=0,\quad  t\in I.
\end{equation}
Regarding the thick structure we also assume it is clamped at the inlet/outlet boundaries, so at 
\[
\Gamma_{in}^{s}:=\left\{ \left(r\cos\theta,r\sin\theta,0\right):r\in\left[R,R+H\right),\ \theta\in\left[0,2\pi\right)\right\} 
\]
and at
\[
\Gamma_{out}^{s}:=\left\{ \left(r\cos\theta,r\sin\theta,L\right):r\in\left[R,R+H\right),\ \theta\in\left[0,2\pi\right)\right\} 
\] we assume that 
\begin{equation}
\mathbf{d}\left(t,r,\theta,z\right)=\mathbf{0} \ 
\text{on} \   I\times\Gamma_{in/out}^{s}.
\end{equation}
At the external (top) boundary of the thick structure, corresponding to $r=R+H$, we assume for simplicity that is exposed to an ambient pressure $P_e$, which is set $P_e=0$. Thus, 
\begin{equation}
\mathbb{S}\mathbf{e}_{r}=-P_{e}\mathbf{e}_{r}=\boldsymbol{0}\quad\text{on}\ \Gamma_{ext}=\left\{ \left(\left(R+H\right)\cos\theta,\left(R+H\right)\sin\theta,z\right):\ \left(\theta,z\right)\in\omega\right\} .
\end{equation}

Finally, let us specify the time-periodic boundary conditions 
\begin{equation}
\begin{cases}
\mathbf{u}(0,\cdot)=\mathbf{u}(T,\cdot)
& \text{in } \Omega_F^{\eta}(0),\\[0.5ex]
\eta(0,\cdot)=\eta(T,\cdot), \qquad
\partial_t\eta(0,\cdot)=\partial_t\eta(T,\cdot)
& \text{in } \omega,\\[0.5ex]
\mathbf{d}(0,\cdot)=\mathbf{d}(T,\cdot), \qquad
\partial_t\mathbf{d}(0,\cdot)=\partial_t\mathbf{d}(T,\cdot)
& \text{in } \Omega_S.
\end{cases}
\end{equation}
Denoting the space-time domains by
\[I\times\Omega_{F}^{\eta}:=\bigcup_{t\in I}\left\{ t\right\} \times\Omega_{F}^{\eta}\left(t\right)\] and the force exerted on the thin elastic structure by 
\[
\mathbf{F}\left(t,\theta,z\right):=-J_{\eta}\left(t,\theta,z\right)\left(\left.\sigma\mathbf{n}^{\eta}\right|_{r=R+\eta}\right)\cdot\mathbf{e}_{r}+\left(\left.\mathbb{S}\right|_{r=R}\mathbf{e}_{r}\right)\cdot\mathbf{e}_{r}
\] 
we may summarize all the above in the following system:
\begin{equation}\label{eqn:FSI}
\begin{cases}
\left(\partial_{t}\mathbf{u}+\left(\mathbf{u}\cdot\nabla\right)\mathbf{u}\right)=\mathrm{div}\ \sigma & \text{in}\ I\times\Omega_{F}^{\eta},\\
\text{div}\ \mathbf{u}=0 & \text{in}\ I\times\Omega_{F}^{\eta},\\
\frac{1}{2}\left|\mathbf{u}\right|^{2}+p=P_{in/out}\left(t\right),\ \mathbf{u}\times\mathbf{e}_{z}=\mathbf{0} & \text{on}\ I\times\Gamma_{in/out},\\
\partial_{tt}\eta+\Delta^{2}\eta=\mathbf{F}\cdot\mathbf{e}_{r}\left(\theta\right) & \text{in}\ I\times\omega,\\
\partial_{tt}\mathbf{d}=\text{div} \  \mathbb{S} & \text{in}\ I\times\Omega_{S},\\
\text{tr}_{\Gamma^{\eta}}\mathbf{u}=\partial_{t}\eta\mathbf{e}_{r},\ \text{tr}_{\Gamma}\mathbf{d}=\eta\mathbf{e}_{r} & \text{on}\ I\times\omega,\\
\eta=\left|\nabla\eta\right|=0 & \text{on}\ I\times\partial\omega,\\
\mathbf{d}=\mathbf{0} & \text{on}\ I\times\Gamma_{in/out}^{s},\\
\mathbb{S}\mathbf{e}_{r}=\mathbf{0} & \text{on}\ I\times\Gamma_{ext},\\
\mathbf{u}\left(0,\cdot\right)=\mathbf{u}\left(T,\cdot\right) & \text{in}\ \Omega_{\eta}\left(0\right),\\
\eta\left(0,\cdot\right)=\eta\left(T,\cdot\right),\ \partial_{t}\eta\left(0,\cdot\right)=\partial_{t}\eta\left(T,\cdot\right) & \text{in}\ \omega,\\
\mathbf{d}\left(0,\cdot\right)=\mathbf{d}\left(T,\cdot\right),\ \partial_{t}\mathbf{d}\left(0,\cdot\right)=\partial_{t}\mathbf{d}\left(T,\cdot\right) & \text{in}\ \Omega_{S}.
\end{cases}
\end{equation}
\paragraph{Brief summary of previous results.}
Over the past two decades, fluid-structure interaction (FSI) problems involving moving interfaces have attracted significant attention. Broadly, these models fall into two categories: those where the structure moves within the fluid and those where the structure is situated at the boundary of the fluid domain. As this article focuses on FSI models with structures positioned at the fluid boundary, we will highlight relevant studies from the literature that specifically address this scenario. For a comprehensive review of analytical results and related applications concerning the FSI problems, we refer to the excellent survey \cite{vcanic2021moving} and the references therein. In literature, a lot of focus has been given in studying the Cauchy problem, which involves determining solutions given initial data. The existence of weak solutions for incompressible fluid-elastic plates interaction has been initiated in the works \cite{chambolle2005existence, grandmont2008existence} and has been extended to incompressible fluid-linear elastic Koiter shell interaction in \cite{LR14, lengeler2014weak}. Subsequent advancements were achieved in \cite{MS22}, where the existence of solutions for the nonlinear Koiter shell model was demonstrated. In \cite{MC13}, the authors provided constructive existence results for weak solutions to incompressible fluid and linear elastic shell interaction problems, utilizing Arbitrary-Lagrangian-Eulerian (ALE) methods. The analysis was extended in \cite{MC15} to account for nonlinear cylindrical Koiter shells, allowing for displacements that are not restricted to radial symmetry. Recently, \cite{mindrilua2026existence} established the existence of weak solutions up to contact for a three-dimensional fluid-structure interaction model with slip boundary conditions, under the assumption that the structure undergoes only transverse deformations.

The pioneering work on the local-in-time existence and uniqueness of strong solutions for fluid-viscoelastic beam interactions with small initial data was first presented in \cite{beirao2004existence}. A significant breakthrough came in \cite{grandmont2016existence}, where the authors demonstrated the global-in-time existence and uniqueness of solutions for arbitrary initial data in a 2D/1D framework, focusing on viscoelastic beam dynamics. Nonlinear elastic Koiter shell models were rigorously analyzed in \cite{cheng2010interaction, maity2020maximal}, pushing the boundaries of structural interaction theory. Recent strides have also been made in exploring complex interactions involving compressible and heat-conducting fluids coupled with beams and shells, as highlighted in \cite{BS18, BS21, maity2021existence, macha2022existence}, marking important progress in this rapidly evolving field.


The study of \textbf{time-periodic} solutions holds significant importance, particularly in biomedical contexts, where many critical fluid flows, such as blood flow, are driven by the heart's periodic pumping action. This periodic nature mirrors the cyclical behavior observed in various physiological processes, making the analysis of time-periodic models highly relevant for understanding cardiovascular dynamics and related phenomena. Despite its importance, the body of literature addressing time-periodic fluid-structure interaction problems remains relatively limited. This scarcity is largely attributed to the mathematical challenges involved, with the primary difficulty arising from the need to establish \emph{uniform-in-time} energy estimates. Achieving such estimates is crucial for proving the long-term stability and boundedness of solutions over successive cycles, which can be particularly demanding in the context of nonlinear and coupled systems.

These challenges have been acknowledged and addressed in several notable works. In the context of incompressible fluid interactions with elastic plates and shells, key contributions have been made in \cite{Casanova} (focusing on strong solutions), as well as in \cite{Claudiu22} and \cite{Claudiu23}, where the problem of time-periodic interactions was rigorously studied. For compressible fluids, progress has been made in \cite{kreml}, extending the analysis to account for the complexities introduced by variable fluid density. Further insights into time-periodic behavior in fluid-structure interaction problems can be found in \cite{juodagalvyte2020time} and \cite{panasenko2024multiscale}, where multiscale methods and advanced mathematical frameworks  are employed to explore periodic dynamics across different structural and fluid models.

\paragraph{Novelty and significance.}
In this work, we investigate the interaction between incompressible viscous fluids and multilayered elastic structures. The problem is set within a \textbf{3D/2D/3D} framework, representing a 3D incompressible fluid interacting with a 2D thin elastic structure, which in turn is coupled to a 3D thick elastic solid. The dynamics of the system are driven by prescribed inflow and outflow boundary conditions involving the Bernoulli pressure, expressed as 
\begin{equation}
\frac{1}{2}\left|\mathbf{u}\right|^{2}+p=:P_{in/out}\left(t\right)
\end{equation}
where 
$p$ denotes the fluid pressure and 
$\mathbf{u}$ the velocity field. 

{
On the inflow and outflow sections $\Gamma_{in/out}$, where the velocity is purely
axial, $\mathbf{u} = u_z\,\mathbf{e}_z$, we prescribe the total
(dynamic) pressure rather than the velocity profile itself,
\begin{equation}
  \label{eq:bernoulli-bc}
  \tfrac{1}{2} u_z^2 + p = P(t)
  \qquad \text{on } \Gamma_{in/out} \times (0,T),
\end{equation}
with $P(t)$ a prescribed driving pressure. Condition
\eqref{eq:bernoulli-bc} is a \emph{defective} boundary condition in the
sense of~\cite{HeywoodRannacherTurek1996}: it couples $u_z$ and $p$
through a single scalar relation and therefore arises naturally as the
Neumann-type boundary term of the variational formulation, rather than
as an essential (Dirichlet) constraint on the trial space. Testing the
momentum equation with $\mathbf{u}$ and integrating by parts, the
convective and pressure contributions on $\Gamma_{in/out}$ combine
into
\begin{equation}
  \label{eq:bernoulli-energy}
\int_{\Gamma_{in/out}}\left(\tfrac{1}{2}|\mathbf{u}|^{2}+p\right)\left(\mathbf{u}\cdot\mathbf{n}\right)dA=P(t)\int_{\Gamma_{in/out}}\left(\mathbf{u}\cdot\mathbf{n}\right)dA.
\end{equation}

Conditions of this type were first introduced by
Pironneau~\cite{Pironneau1986,Pironneau1987} and given a rigorous
well-posedness theory by Conca, Murat and
Pironneau~\cite{ConcaMuratPironneau1994}, with the flux/pressure variant
closest to~\eqref{eq:bernoulli-bc} developed for physiological flows by
Conca, Par\'es, Pironneau and Thiriet~\cite{ConcaParesPironneauThiriet1995};
see also~\cite{HeywoodRannacherTurek1996} for the ``do-nothing''
terminology and~\cite{FormaggiaGerbeauNobileQuarteroni2001} for its use
in fluid--structure interaction with compliant vessels. More recently it was also justified with variational arguments by Bathory and Stefanelli in \cite{BaSt22}.
}

We establish that if the prescribed boundary pressure $P_{in/out}$ is time-periodic with sufficiently small $L^2$
 -norm, at least one time-periodic weak solution exists. In contrast to the results presented in \cite{Boris}, our analysis necessitates the additional assumption that the thick elastic structure exhibits viscoelastic properties. This assumption is not only practically reasonable but also crucial for obtaining \emph{diffusion estimates}, which play a key role in deriving energy bounds and ensuring the stability of solutions. However, if the thick solid is assumed to be purely elastic, our methodology still yields the existence of weak solutions, albeit for the initial-value problem rather than the time-periodic case. This result follows as a direct corollary of our main approach. To the best of our knowledge, this outcome is novel, as the work in \cite{Boris} primarily addressed fluid-structure interaction in a 2D/1D/2D multilayered setting, and did not extend to the 3D/2D/3D case considered here. Our findings contribute to the broader understanding of multilayered fluid-structure interaction models, bridging the gap between theoretical analysis and practical scenarios involving complex elastic structures. A study of multilayer fluid-structure interaction systems containing poroelastic structures under linear coupling assumptions was presented in \cite{bociu2021multilayered}.
A recent contribution in the framework of strong solutions is our work \cite{brandt2026strong}, jointly with F. Brandt, where we studied a fully three-dimensional time-periodic fluid–structure interaction system coupling the Navier–Stokes equations for an incompressible viscous fluid with a multilayer elastic structure formed by a damped thin linear plate and a thick viscoelastic layer. Time-periodic fluid-rigid body interaction problems have been investigated in \cite{bonheure2024global}.

In terms of methodology, we adopt an approach that relies on decoupling the fluid-structure interaction problem and constructing solutions through the Galerkin method. This strategy allows us to handle the complexities arising from the multilayered nature of the system by addressing each component separately before integrating them into a cohesive solution framework. A key element of our approach is the extension operator introduced in Proposition \ref{prop:estimates-extension-operator}. This operator proves to be a powerful and versatile tool, facilitating the derivation of essential \emph{a priori estimates}. These estimates play a crucial role in controlling the behavior of the coupled system, ensuring stability and boundedness of the solutions. Additionally, the extension operator aids in establishing the  $L^2$-compactness of the fluid velocity field, even within the challenging context of multilayered fluid-structure interactions.  We would like to the emphasize that the uniform in time estimates on the energy are due to Poincar\'{e} estimates, available upon provided estimates on the trace of $\mathbf{u}$. These estimates transfer from the visco-elasticity of the \emph{thick external structure}, a feature which was not available in simpler fluid-shell models as in \cite{Claudiu22, Claudiu23}.

The robustness of this extension operator is particularly valuable when dealing with the interaction between thin and thick elastic layers, as it provides the necessary analytical framework to manage the discontinuities and variations across different structural components. By leveraging this tool, we can effectively overcome the difficulties posed by the multilayered environment, enhancing the overall stability and convergence of the Galerkin approximation process. This methodological approach not only streamlines the solution construction but also strengthens the analytical foundations required to address the existence of solutions, both for the initial-value problem and for the time-periodic setting under small periodic boundary data conditions. 
\begin{remark}
   We consider the visco-elastic in \eqref{eqn:lame-stress}   to obtain the diffusion estimate \eqref{eqn:diffusion-estimate} which is crucial to the existence of \emph{time-periodic} problems. Hence $\delta$ can be small but positive. For simplicity we set $\delta=1$. For the case of an initial-value problem we may relax the assumptions to allow also the purely elastic case that is  $\delta \ge 0$. Alternatively, the additional term $\delta_{1} (\partial_{t} \mathbf{d}) \mathbb{I}_{3}$ with $\delta_1>0$ in \eqref{eqn:lame-stress} would suffice to regularize the displacement of the structure and to establish the estimate~\eqref{eqn:diffusion-estimate}.  The regularizing effect of visco-elasticity (and of the viscous effects in general)  was highlighted in many works and it is also related to the phenomenon of \emph{resonance} see \cite{galdi2014hyperbolic} and also the more recent works \cite{benson2024resonance} and \cite{mosny2024time}.
   \end{remark}
\begin{remark}
    It seems possible that the inflow/outflow boundary condition
    that we use (see  also \cite{CMP94}) might be replaced by a Dirichlet condition in normal direction or a net-flux condition. We refer to \cite{BaSt22} for further details.
\end{remark}
\begin{remark}
    Since the problem that we study can be easily related to the flow of blood through an artery, it is usual to model the blood as a non-Newtonian fluid of Carreau type; this means updating the Cauchy stress tensor to $\tilde{\sigma}:=\nu_{\infty}\mathbb{D}\mathbf{u}+\nu\left(\mu^{2}+\left|\mathbb{D}\mathbf{u}\right|^{\frac{p-2}{2}}\right)-p\mathbb{I}$ for $\nu_{\infty}, \ \mu \ge 0$, $\nu>0$ and $p>2$. Including this model seems to be  possible in the light of \cite{Lengeler16}.
\end{remark}
\paragraph{Plan of the paper.}
The plan of the paper is as follows. Section \ref{section:weak-soln} begins by defining the concept of weak solutions for the fluid-structure interaction system described by \eqref{eqn:FSI}. This section also presents the primary results of the paper, outlining the key theorem and corollary that form the foundation of our analysis. In Section \ref{sec:formal-a-priori}, we introduce the construction of a divergence-free extension operator, a critical tool for addressing the coupling between fluid and solid dynamics. Following this, we derive the system’s energy balance, providing the necessary a priori estimates that are instrumental for the subsequent analysis.
Section \ref{sec:compactness} focuses on establishing the $L^2-$ compactness of fluid velocities, a fundamental step in demonstrating the convergence of approximate solutions and ensuring the stability of the multilayered interaction model. The heart of the paper lies in Section \ref{sec:main-construction}, where we present the detailed proofs of the main results. This section synthesizes the preparatory work from earlier sections into a cohesive argument, ultimately confirming the existence of weak solutions to the system. For completeness and to enhance the readability of the paper, we have included Section \ref{sec:appendix}, which lists auxiliary results and technical tools used throughout the analysis.
\section{Weak formulation and main results}\label{section:weak-soln}
Motivated by the energy bound \eqref{eqn:formal-en-bound}, we introduce the function spaces for the unknowns as follows:
\begin{equation}\label{eqn:space-energy}
\begin{aligned}H_{fl}^{\eta}\left(t\right) & :=\left\{ \mathbf{u}\in H^{1}\left(\Omega_{F}^{\eta}\left(t\right);\mathbb{R}^{3}\right):\text{div}\ \mathbf{u}=0,\ \mathbf{u}\cdot\mathbf{e}_{\theta}=\mathbf{u}\cdot\mathbf{e}_{z}=\mathbf{0}\ \text{on}\ \Gamma_{\eta}\left(t\right),\right.\\
 & \left.\ \qquad\mathbf{u}\times\mathbf{e}_{z}=\mathbf{0}
\ \text{on}\ \Gamma_{in/out}\right\} ,\\
H_{0}^{2}\left(\omega\right) & :=\left\{ \eta\in H^{2}\left(\omega\right):\ \eta\ \text{is }2\pi\text{-periodic in }\theta,\ \left.\eta\right|_{z=0,L}=\left.\partial_{z}\eta\right|_{z=0,L}=0\right\} \\
H_{solid} & :=\left\{ \mathbf{d}\in H^{1}\left(\Omega_{S};\mathbb{R}^{3}\right):\mathbf{d}\left(R,\cdot\right)\cdot\mathbf{e}_{\theta}=\mathbf{d}\left(R,\cdot\right)\cdot\mathbf{e}_{z}=\mathbf{0},\ \mathbf{d}=\mathbf{0}\ \text{on}\ \Gamma_{in/out}^{s}\right\} ,\\
V_{fl}^{\eta} & :=L_{per}^{\infty}\left(I;L^{2}\left(\Omega_{F}^{\eta}\left(t\right)\right)\right)\cap L_{per}^{2}\left(I;H_{fl}^{\eta}\left(t\right)\right),\\
V_{w} & :=W_{per}^{1,\infty}\left(I;L^{2}\left(\omega\right)\right)\cap L_{per}^{2}\left(I;H_{0}^{2}\left(\omega\right)\right),\\
V_{s} & :=W_{per}^{1,\infty}\left(I;L^{2}\left(\Omega_{S}\right)\right)\cap W_{per}^{1,2}\left(I;W^{1,2}\left(\Omega_{S}\right)\right)\cap L_{per}^{2}\left(I;H_{solid}\right),
\end{aligned}
\end{equation}

where for any Banach space $X$ we define the spaces $L^{p}_{per}\left(I;X\right)$ and $W^{k,p}_{per}(I;X)$ as the closure in the corresponding norms of the smooth periodic functions
\[\left\{ \varphi:\varphi\in C^{\infty}\left(\mathbb{R};X\right):\varphi\left(t+T\right)=\varphi\left(t\right),\ \forall\  t\in\mathbb{R}\right\}. \] 

Then we may define the function spaces corresponding to  the solutions and  the test functions respectively as

\begin{equation}\label{eqn:spaces-solns-testfunctions}
    \begin{aligned}V_{soln}^{\eta}:= & \left\{ \left(\mathbf{u},\eta,\mathbf{d}\right)\in V_{fl}^{\eta}\times V_{w}\times V_{s}:\text{tr}_{\Gamma^{\eta}\left(t\right)}\mathbf{u}=\partial_{t}\eta\left(t,z\right)\mathbf{e}_{r},\right.
\left.\mathbf{d}\left(t,R,\theta,z\right)=\eta\left(t,z\right)\mathbf{e}_{r}\right\}, \\
V_{test}^{\eta}:= & \left\{ \left(\mathbf{q},\xi,\boldsymbol{\xi}\right)\in C_{\text{per}}^{1}\left(I;H_{fl}\times H^{2}_{0}\left(\omega \right)\times H_{solid}\right):\right.
\left.\text{tr}_{\Gamma^{\eta}\left(t\right)}\mathbf{q}=\xi\left(t,z\right)\mathbf{e}_{r}=\boldsymbol{\xi}\left(t,R,\theta,z\right)\right\}. 
\end{aligned}
\end{equation}

Due to the clamped boundary condition we will denote by  $C_{0}^{\infty}\left(\omega\right)$ 
the space 
\[
C_{0}^{\infty}\left(\omega\right):=\left\{ \eta\in C^{\infty}\left(\omega\right):\ \eta\ \text{is }2\pi\text{-periodic in }\theta,\ \left.\eta\right|_{z=0,L}=\left.\partial_{z}\eta\right|_{z=0,L}=0\right\}. 
\]

\subsection{Weak formulation}\label{ssec:deriving-weak-form}
Let us assume for simplicity and without loss of generality that all the involved quantities in the sequel are smooth.
Please note that the fluid domains $\Omega_{F}^{\eta}(t)$ for $t\in I$ are not degenerate provided that $R+\eta\left(t,\theta,z\right)>0$ for all $(\theta,z)\in \omega$ and all $t\in I$. 
We will ensure this condition by constructing solutions that fulfill the condition $\left\Vert \eta\right\Vert _{L_{t,x}^{\infty}}<R$. 

 To this end, since $H^{2}\left(\omega\right)\hookrightarrow L^{\infty}\left(\omega\right)$ we may estimate using that $\eta \in H_{0}^{2}\left(\omega \right)$, as follows:\footnote{Recall that $\left.\eta\right|_{z=0}=\left.\partial_{z}\eta\right|_{z=0}=0$ and $\int_{0}^{2\pi}\partial_{\theta}\eta d\theta=0$.}
\[
\begin{aligned}\left\Vert \eta\left(t\right)\right\Vert _{L_{x}^{\infty}}\lesssim & \sup_{t\in I}\left\Vert \eta\right\Vert _{L_{x}^{2}}+\left\Vert \nabla\eta\right\Vert _{L_{x}^{2}}+\left\Vert \nabla^{2}\eta\right\Vert _{L_{x}^{2}}\\
\lesssim & \sup_{t\in I}K\left(\eta\right)\le\sup_{t\in I}E\left(t\right)
\end{aligned}
\]

Thus, throughout $\left\Vert \eta\right\Vert _{L_{t,x}^{\infty}}\lesssim\sup_{t\in I}E\left(t\right)\le M$ we can guarantee the existence of non-degenerate fluid-domains $\Omega_{F}^{\eta}\left(t\right)$ provided that we choose a sufficiently small $M$, that ensures $\left\Vert \eta\right\Vert _{L_{t,x}^{\infty}}<R$.  

Let us formally derive the weak or variational formulation.
We multiply the fluid equation \eqref{eqn:fluid} by $\mathbf{q}$, the elastic equation \eqref{eqn:wave} by $\xi$ and the solid equation~\eqref{eqn:thick} by $\boldsymbol{\xi}$ which are chosen so that $\left(\mathbf{q},\xi,\boldsymbol{\xi}\right)\in V_{test}^{\eta}$.
Performing integration by parts, using the divergence free condition and the coupling condition we arrive at the following
\begin{definition}\label{def:weak-solution}
    We say a couple $\left(\mathbf{u},\eta,\mathbf{d}\right)\in V_{soln}^{\eta}$  with $\left\Vert \eta\right\Vert _{L_{t,x}^{\infty}} <R$ is a \emph{time-periodic weak solution} to the FSI problem \eqref{eqn:FSI}, if for all $\left(\mathbf{q},\xi,\boldsymbol{\xi}\right)\in V_{test}^{\eta}$ the following weak formulation holds: 
    \begin{equation}\label{eqn:weak-form}
        \begin{aligned}\int_{I}\int_{\Omega_{F}^{\eta}\left(t\right)}-\mathbf{u}\cdot\partial_{t}\mathbf{q}+\nabla\mathbf{u}:\nabla\mathbf{q}\ dx \ dt+\int_{I}b\left(t,\mathbf{u},\mathbf{u},\mathbf{q}\right)dt +
\int_{I}\int_{\omega}-\frac{1}{2}\left(\partial_{t}\eta\right)^{2}\left(R+\eta\right)\xi -\partial_{t}\eta\partial_{t}\xi\ dA\ dt+ & \\\int_{I}K\left(\eta,\xi\right)\ dt  +
\int_{I}\int_{\Omega_{S}}-\partial_{t}\mathbf{d}\cdot\partial_{t}\boldsymbol{\xi}\ dx\ dt+\int_{I}a_{S}\left(\mathbf{d},\boldsymbol{\xi}\right)\ dt =
\int_{I}\left\langle F\left(t\right),\mathbf{q}\right\rangle _{\Gamma_{in/out}}\ dt .
\end{aligned}
    \end{equation}
\end{definition}
Here, we denote
\begin{equation}\label{eqn:functionals}
\begin{aligned}b\left(t,\mathbf{u},\mathbf{v},\mathbf{w}\right):= & \frac{1}{2}\int_{\Omega_{F}^{\eta}\left(t\right)}\left(\mathbf{u}\cdot\nabla\right)\mathbf{v}\cdot\mathbf{w}\ dx-\frac{1}{2}\left(\mathbf{u}\cdot\nabla\right)\mathbf{w}\cdot\mathbf{v}\ dx,\\
K\left(\eta,\xi\right):= & \int_{\omega}\nabla^{2}\eta:\nabla^{2}\xi\ dA,\\
a_{S}\left(\mathbf{d},\boldsymbol{\xi}\right):= & \int_{\Omega_{S}}\nabla\mathbf{d}:\nabla\boldsymbol{\xi}+\nabla\partial_{t}\mathbf{d}:\nabla\boldsymbol{\xi}+\left(\text{div}\mathbf{d}\right)\left(\text{div}\boldsymbol{\xi}\right)\ dx,\\
\left\langle F\left(t\right),\mathbf{v}\right\rangle _{\Gamma_{in/out}}:= & P_{\text{in}}\left(t\right)\int_{\Gamma_{in}}\mathbf{v}\cdot\mathbf{n}\ dA-P_{out}\left(t\right)\int_{\Gamma_{out}}\mathbf{v}\cdot\mathbf{n}\ dA .
\end{aligned}
\end{equation}
Detailed computations can be found for example in \cite{Boris}.
\subsection{Main results}
We are now in a position to formulate the main results of this work.
\begin{theorem}\label{thm:main}
There exists a constant $C=C(\texttt{data)}$ with the following property: if
$P_{in/out}\in L^{2}_{per}\left(I\right)$ is a time-periodic prescribed inflow/outflow with $\left\Vert P_{in/out}\right\Vert _{L_{t}^{2}}\le C$, then there exists at least one time-periodic weak solution $\left(\mathbf{u},\eta,\mathbf{d}\right)\in V_{soln}^{\eta}$ to \eqref{eqn:FSI} in the sense of Definition~\ref{def:weak-solution}.

Furthermore, recalling the energy $E(t)$ from \eqref{eqn:energy-generic} and the diffusion $D(t)$ from \eqref{eqn:def-diff}
our solutions enjoy, furthermore, the following \emph{uniform in time} estimate:
\begin{equation}
    \sup_{t\in I}E\left(t\right)+\int_{I}D\left(t\right)dt\lesssim \left\Vert P_{in/out}\right\Vert _{L^{2}(I)}^{2}.
\end{equation}
up to a constant depending on \texttt{data}.
\end{theorem}

Concerning the initial-value problem our methodology allows to include the existence in a 3D/2D/3D setting which seems to have not been covered in \cite{Boris}. This extension highlights the adaptability and robustness of our approach. In the simpler yet widely studied 2D/1D/2D configuration, where the fluid and solid domains are reduced by one dimension, our results align with and recover the findings presented in \cite{Boris}. This consistency not only validates our methodology but also demonstrates its capacity to generalize known results to more complex and realistic settings. 

The corresponding initial-value (Cauchy) problem is
\begin{equation}\label{eqn:FSIini}
\begin{cases}
\left(\partial_{t}\mathbf{u}+\left(\mathbf{u}\cdot\nabla\right)\mathbf{u}\right)=\mathrm{div}\ \sigma & \text{in}\ I\times\Omega_{F}^{\eta},\\
\text{div}\ \mathbf{u}=0 & \text{in}\ I\times\Omega_{F}^{\eta},\\
\frac{1}{2}\left|\mathbf{u}\right|^{2}+p=P_{in/out}\left(t\right),\ \mathbf{u}\times\mathbf{e}_{z}=\mathbf{0} & \text{on}\ I\times\Gamma_{in/out},\\
\partial_{tt}\eta+\Delta^{2}\eta=\mathbf{F}\cdot\mathbf{e}_{r}\left(\theta\right) & \text{in}\ I\times\omega,\\
\partial_{tt}\mathbf{d}=\text{div} \  \mathbb{S} & \text{in}\ I\times\Omega_{S},\\
\text{tr}_{\Gamma^{\eta}}\mathbf{u}=\partial_{t}\eta\mathbf{e}_{r},\ \text{tr}_{\Gamma}\mathbf{d}=\eta\mathbf{e}_{r} & \text{on}\ I\times\omega,\\
\eta=\left|\nabla\eta\right|=0 & \text{on}\ I\times\partial\omega,\\
\mathbf{d}=\mathbf{0} & \text{on}\ I\times\Gamma_{in/out}^{s},\\
\mathbb{S}\mathbf{e}_{r}=\mathbf{0} & \text{on}\ I\times\Gamma_{ext},\\
\mathbf{u}\left(0,\cdot\right)=\mathbf{u}_0 & \text{in}\ \Omega_{\eta}\left(0\right),\\
\eta\left(0,\cdot\right)=\eta_0,\ \partial_{t}\eta\left(0,\cdot\right)=\eta_1 & \text{in}\ \omega,\\
\mathbf{d}\left(0,\cdot\right)=\mathbf{d}_0,\ \partial_{t}\mathbf{d}\left(0,\cdot\right)=\mathbf{d}_1 & \text{in}\ \Omega_{S}.
\end{cases}
\end{equation}
{
The counterpart of Theorem~\ref{thm:main} corresponding to the initial condition \eqref{eqn:FSIini} is stated below. In this setting, we also allow $\delta=0$ in the definition of the elastic tensor $\mathbb{S}$ from \eqref{eqn:lame-stress}. 

\begin{corollary}\label{thm-corollary}
Given an initial data $\left(\mathbf{u}_{0},\eta_{0},\eta_{1},\mathbf{d}_{0},\mathbf{d}_{1}\right)$ with finite energy $E\left(0\right)<\infty$ and such that $\inf\limits_{x\in\omega}\left(R+\eta_{0}\left(x\right)\right)=\min\limits_{x\in\omega}\left(R+\eta_{0}\left(x\right)\right)>0$, there exists a maximal time $T_{\max}=T_{\max}\left(\mathbf{u}_{0},\eta_{0},\eta_{1},\mathbf{d}_{0},\mathbf{d}_{1}\right)$ with
$0<T_{\max}\le \infty$ for which the FSI-problem \eqref{eqn:FSIini} has at least one weak-solution in $(0,T_{\max})$. 
In case of no degeneracy (no self-contact) of the fluid domains $\Omega_{F}^{\eta}$ occur, we can take $T_{\max}=\infty$.
The result is valid any $\delta \ge 0$ in the definition of the tensor $\mathbb{S}$ from \eqref{eqn:lame-stress}.
\end{corollary}
}
\begin{remark}
     The  main results--Theorem~\ref{thm:main} and Corollary~\ref{thm-corollary}  depend on  given quantities such as fluid density $\rho_F$, viscosity $\mu$, thickness $H$ of the thick solid, Lam\'{e} coefficients of elasticity, the volume of the domains $\Omega$ and $\Omega_{S}$, the time $T$ and so on.
     These parameters are fixed from the very beginning and we will generically refer to them as $\texttt{data}$. This is why the main results are formulated up to a constant $C=C(\texttt{data})$.
\end{remark}
\section{A-priori estimates}\label{sec:formal-a-priori}
In order to obtain \emph{uniform in time} a-priori estimates, we first need to construct a suitable extension operator. Its construction resembles \cite{LR14} and \cite{Claudiu23}.
\subsection{A divergence-free extension operator}

We construct test functions adapted to the fluid equations on the moving domains.
Since the trace of a fluid test function on the moving interface must agree with
the corresponding test function for the thin structure, we start with a scalar
function $\xi$ defined on the fixed parameter domain $\omega$ and construct a
divergence-free extension having trace $\xi\mathbf e_r$ on the moving interface.

Let $\delta\in C_{\mathrm{per}}^{\infty}
(I;C^{\infty}(\omega;\mathbb R))$ describe an admissible motion and let
$\Omega_F^\delta(t)$ be the associated fluid domain. We assume that there is a
fixed constant $M<R$ such that $\|\delta\|_{L^\infty(I\times\omega)}\leq M.$
In particular, $R+\delta(t,\theta,z)\geq R-M>0$ on $I\times\omega,$
so the moving fluid domain does not degenerate. In the rest of the paper, we
use $M=R/2$ and thus, we have
\begin{equation}\label{eq:delta-half-R-extension}
    \|\delta\|_{L^\infty(I\times\omega)}\leq \frac R2,
    \qquad R+\delta\geq\frac R2.
\end{equation}
{
For $0\leq a<b$, let
\[
    C_{a,b}
    :=\left\{
    (r\cos\theta,r\sin\theta,z):
    a<r<b,\ (\theta,z)\in\omega
    \right\}.
\]
As in the geometric setting of the problem, we assume that
$\Omega_F^\delta(t)\subset C_{0,R+H}$ for every $t\in I$, and we write
$C:=C_{0,R/2}$. We use cylindrical coordinates
$(x,y,z)=(r\cos\theta,r\sin\theta,z)$, 
    $\mathbf e_r(\theta)=(\cos\theta,\sin\theta,0)$, $ 
    \mathbf e_z=(0,0,1).$
Let $\xi\in C^\infty(I;C_0^\infty(\omega))$ and set
\begin{equation}\label{eq:def-f-delta-xi}
    f_{\delta,\xi}(t,\theta,z)
    :=\bigl(R+\delta(t,\theta,z)\bigr)\xi(t,\theta,z),
    \qquad
    I_{\delta,\xi}(t)
    :=\int_\omega f_{\delta,\xi}(t,\theta,z)\,d\theta\,dz.
\end{equation}
We choose the functions
\begin{equation}\label{eq:cutoffs-extension}
\begin{aligned}
    &\chi\in C^\infty([0,R+H]),
    \qquad 0\leq\chi\leq1,
    \qquad
    \chi=0\ \text{on }\left[0,\frac R4\right],
    \qquad
    \chi=1\ \text{on }\left[\frac R2,R+H\right],
    \\
    &\beta\in C_0^\infty\left(0,\frac R2\right),
    \qquad
    2\pi\int_0^{R/2}\beta(r)r\,dr=1,
    \\
    &\gamma\in C^\infty([0,L]),
    \qquad
    \gamma=1\ \text{in a neighbourhood of }z=0,
    \qquad
    \gamma=0\ \text{in a neighbourhood of }z=L.
\end{aligned}
\end{equation}
We extend $\beta$ by zero to $[0,R+H]$ and define on the fixed cylinder
$C_{0,R+H}$
\begin{equation}\label{eq:def-F-tilde-cutoff}
    \widetilde{\mathcal F}_\delta(\xi)(t,r,\theta,z)
    :=
    \chi(r)\frac{f_{\delta,\xi}(t,\theta,z)}{r}\mathbf e_r(\theta)
    +\beta(r)I_{\delta,\xi}(t)\gamma(z)\mathbf e_z.
\end{equation}
The first term in \eqref{eq:def-F-tilde-cutoff} is understood to be zero near
$r=0$. This gives a smooth definition there because $\chi$ vanishes identically
on $[0,R/4]$. The trace condition on the moving interface is already satisfied. Indeed,
by \eqref{eq:delta-half-R-extension}, $R+\delta\geq R/2$, and therefore
$\chi(R+\delta)=1$ and $\beta(R+\delta)=0$. It follows that
\begin{equation}\label{eq:trace-F-tilde}
\begin{aligned}
    \widetilde{\mathcal F}_\delta(\xi)
    \bigl(t,R+\delta(t,\theta,z),\theta,z\bigr)
    =\frac{f_{\delta,\xi}(t,\theta,z)}
    {R+\delta(t,\theta,z)}\mathbf e_r(\theta) =\xi(t,\theta,z)\mathbf e_r(\theta).
\end{aligned}
\end{equation}
We next correct the divergence of \eqref{eq:def-F-tilde-cutoff}. Since the
radial, angular and axial components of $\widetilde{\mathcal F}_\delta(\xi)$ are
\[
    \widetilde{\mathcal F}_{\delta,r}
    =\chi(r)\frac{f_{\delta,\xi}}r,
    \qquad
    \widetilde{\mathcal F}_{\delta,\theta}=0,
    \qquad
    \widetilde{\mathcal F}_{\delta,z}
    =\beta(r)I_{\delta,\xi}\gamma(z),
\]
the cylindrical-coordinate formula for the divergence gives
\begin{equation}\label{eq:div-F-tilde-cutoff}
\begin{aligned}
    g_{\delta,\xi}
    :=\operatorname{div}\widetilde{\mathcal F}_\delta(\xi)
    =\frac1r\partial_r
      \bigl(r\widetilde{\mathcal F}_{\delta,r}\bigr)
      +\frac1r\partial_\theta
      \widetilde{\mathcal F}_{\delta,\theta}
      +\partial_z\widetilde{\mathcal F}_{\delta,z}
    =\frac{\chi'(r)}r f_{\delta,\xi}(t,\theta,z)
      +\beta(r)I_{\delta,\xi}(t)\gamma'(z).
\end{aligned}
\end{equation}
Both terms on the right-hand side are supported in the interior cylinder $C$.
Furthermore, $g_{\delta,\xi}$ has zero mean. Indeed, using
$dx=r\,dr\,d\theta\,dz$, \eqref{eq:def-f-delta-xi} and
\eqref{eq:cutoffs-extension}, we obtain
\begin{equation}\label{eq:mean-zero-g-extension}
\begin{aligned}
    \int_C g_{\delta,\xi}\,dx
    &=\int_\omega f_{\delta,\xi}\,d\theta\,dz
      \int_0^{R/2}\chi'(r)\,dr
      +2\pi I_{\delta,\xi}
      \int_0^{R/2}\beta(r)r\,dr
      \int_0^L\gamma'(z)\,dz
    \\
    &=I_{\delta,\xi}
      +I_{\delta,\xi}\bigl(\gamma(L)-\gamma(0)\bigr)
     =I_{\delta,\xi}-I_{\delta,\xi}=0.
\end{aligned}
\end{equation}
We may therefore apply the Bogovskii operator on the fixed cylinder $C$ and set
\begin{equation}\label{eq:def-F-delta-corrected}
    \mathcal B_{\delta,\xi}
    :=\operatorname{Bog}_C(g_{\delta,\xi}),
    \qquad
    \mathcal F_\delta(\xi)
    :=\widetilde{\mathcal F}_\delta(\xi)-\mathcal B_{\delta,\xi},
\end{equation}
where $\mathcal B_{\delta,\xi}$ is extended by zero from $C$ to $C_{0,R+H}$.
Since $\mathcal B_{\delta,\xi}\in W_0^{1,q}(C;\mathbb R^3)$ for every
$1<q<\infty$, its zero extension belongs to
$W^{1,q}(C_{0,R+H};\mathbb R^3)$. In particular, no distribution supported on
$\partial C$ is produced by this zero extension. It follows from
\eqref{eq:div-F-tilde-cutoff} and \eqref{eq:def-F-delta-corrected} that
\[
    \operatorname{div}\mathcal F_\delta(\xi)=0
    \quad\text{in }\quad C_{0,R+H}.
\]

We summarize the properties of the extension operator in the following
proposition.

\begin{proposition}\label{prop:estimates-extension-operator}
Let $\delta\in C_{\mathrm{per}}^\infty
(I;C^\infty(\omega;\mathbb R))$ satisfy
\eqref{eq:delta-half-R-extension}. For fixed $\delta$, the map
$\xi\mapsto\mathcal F_\delta(\xi)$ defined by
\eqref{eq:def-f-delta-xi}--\eqref{eq:def-F-delta-corrected} is linear and has
the following properties:
\begin{equation}\label{eq:properties-F-delta}
\begin{aligned}
    &\operatorname{div}\mathcal F_\delta(\xi)=0
      &&\text{in }\Omega_F^\delta(t),
    \\
    &\operatorname{tr}_{\Gamma^\delta(t)}
      \mathcal F_\delta(\xi)=\xi\mathbf e_r
      &&\text{on }\Gamma^\delta(t),
    \\
    &\mathcal F_\delta(\xi)\times\mathbf e_z=\mathbf0
      &&\text{on }\Gamma_{\mathrm{in/out}}.
\end{aligned}
\end{equation}
Moreover, for every $p\in[1,\infty]$, $q\in(1,\infty)$ and
$k\in\{0,1\}$,
\begin{equation}\label{eqn:cons-ext-1}
    \|\mathcal F_\delta(\xi)\|_{L^p(I;W^{k,q}(C_{0,R+H}))}
    \leq C
    \| (R+\delta)\xi\|_{L^p(I;W^{k,q}(\omega))}.
\end{equation}
 We also have
\begin{equation}\label{eqn:cons-ext-2}
    \|\partial_t\mathcal F_\delta(\xi)\|_{L^p(I;L^q(C_{0,R+H}))}
    \leq C
    \|\partial_t\bigl((R+\delta)\xi\bigr)\|_{L^p(I;L^q(\omega))}.
\end{equation}
Finally, if $\delta_1$ and $\delta_2$ satisfy the same assumptions, then
\begin{equation}\label{eqn:cons-ext-3}
\begin{aligned}
    &\|\mathcal F_{\delta_1}(\xi)
       -\mathcal F_{\delta_2}(\xi)\|_{L^\infty(I;L^q(C_{0,R+H}))}
\leq C
    \|(\delta_1-\delta_2)\xi\|_{L^\infty(I;L^q(\omega))}.
\end{aligned}
\end{equation}
For fixed $q$, the constant $C$ depends only on $q$, the fixed geometry and the
fixed functions $\chi$, $\beta$ and $\gamma$; in particular, it is independent
of $\delta$ and $\xi$.
\end{proposition}

\begin{proof}
The divergence-free property in \eqref{eq:properties-F-delta} follows directly
from \eqref{eq:div-F-tilde-cutoff}, \eqref{eq:mean-zero-g-extension} and the
defining property of the Bogovskii operator. Since
$\mathcal B_{\delta,\xi}$ has zero trace on $\partial C$ and vanishes outside
$C$, it does not change the trace on the moving interface. The coupling
condition therefore follows from \eqref{eq:trace-F-tilde}.

Since $\xi\in C_0^\infty(\omega)$, we have
$\xi|_{z=0,L}=0$, and hence $f_{\delta,\xi}|_{z=0,L}=0$.  Consequently, the radial
component of $\widetilde {\mathcal F}_\delta(\xi)$ has zero trace on the inlet and outlet. The
remaining term is parallel to $\mathbf e_z$, while
$\mathcal B_{\delta,\xi}$ has zero trace on the corresponding parts of
$\partial C$. This proves the last condition in
\eqref{eq:properties-F-delta}.

It remains to prove the estimates. Put $f=f_{\delta,\xi}$.  Since $\chi=0$ on $[0,R/4]$, the vector field
\[
 (r,\theta,z)\longmapsto \chi(r)\frac{f(\theta,z)}{r}e_r(\theta)
\]
vanishes near the axis.  On its support, $r\geq R/4$,  its Cartesian first derivatives are
therefore bounded by $C_R\bigl(|f|+|\partial_\theta f|+|\partial_z f|\bigr).$ Consequently, for
$k\in\{0,1\}$ and $1<q<\infty$,
\begin{equation}\label{eq:estimate-radial-cutoff-part}
    \left\|
    \chi(r)\frac f r\mathbf e_r
    \right\|_{W^{k,q}(C_{0,R+H})}
    \leq C\|f\|_{W^{k,q}(\omega)}.
\end{equation}
By H\"older's inequality,
\begin{equation}\label{eq:estimate-I-delta-xi}
    |I_{\delta,\xi}|
    \leq |\omega|^{1-1/q}\|f\|_{L^q(\omega)}.
\end{equation}
Since $\beta$ and $\gamma$ are fixed smooth functions,
\eqref{eq:estimate-I-delta-xi} gives
\begin{equation}\label{eq:estimate-axial-correction}
    \|\beta I_{\delta,\xi}\gamma\mathbf e_z
      \|_{W^{k,q}(C_{0,R+H})}
    \leq C\|f\|_{L^q(\omega)},
    \qquad k\in\{0,1\}.
\end{equation}
Combining \eqref{eq:estimate-radial-cutoff-part} and
\eqref{eq:estimate-axial-correction}, we obtain
\begin{equation}\label{eq:estimate-F-tilde-final}
    \|\widetilde{\mathcal F}_\delta(\xi)
      \|_{W^{k,q}(C_{0,R+H})}
    \leq C\|f\|_{W^{k,q}(\omega)},
    \qquad k\in\{0,1\}.
\end{equation}

The explicit formula \eqref{eq:div-F-tilde-cutoff} and
\eqref{eq:estimate-I-delta-xi} imply
\begin{equation}\label{eq:estimate-g-delta-xi}
    \|g_{\delta,\xi}\|_{L^q(C)}
    \leq C\|f\|_{L^q(\omega)}.
\end{equation}
The continuity of the Bogovskii operator, followed by
\eqref{eq:estimate-g-delta-xi}, yields
\begin{equation}\label{eq:estimate-B-delta-xi}
\begin{aligned}
    \|\mathcal B_{\delta,\xi}\|_{W^{1,q}(C)}
    &\leq C\|g_{\delta,\xi}\|_{L^q(C)}
    \leq C\|f\|_{L^q(\omega)}.
\end{aligned}
\end{equation}
It follows from \eqref{eq:estimate-F-tilde-final},
\eqref{eq:estimate-B-delta-xi} and
\eqref{eq:def-F-delta-corrected} that, pointwise in time,
\[
    \|\mathcal F_\delta(\xi)\|_{W^{k,q}(C_{0,R+H})}
    \leq C\|f\|_{W^{k,q}(\omega)},
    \qquad k\in\{0,1\}.
\]
Taking the $L^p(I)$-norm proves \eqref{eqn:cons-ext-1}. Applying the preceding
$L^q$ estimates with $\partial_t f_{\delta,\xi}$ in place of
$f_{\delta,\xi}$ gives
\[
    \|\partial_t\mathcal F_\delta(\xi)(t)\|_{L^q(C_{0,R+H})}
    \leq C
    \|\partial_t f_{\delta,\xi}(t)\|_{L^q(\omega)}.
\]
Taking the $L^p(I)$-norm proves \eqref{eqn:cons-ext-2}. Finally, the condition \eqref{eqn:cons-ext-3} follows again as above, by direct computations and the properties of the Bogovskii operator.
\end{proof}
}
\begin{remark}\label{rmk:test-functions-extension}
It can be easily seen that for all smooth functions $\xi\in C^{\infty}\left(I\times\omega\right)$ and $\boldsymbol{\xi}\in C^{\infty}\left(I\times\Omega_{S}\right)$ such that $\boldsymbol{\xi}\left(t,R,\theta,z\right)=\mathbf{e}_{r}\left(\theta\right)\xi\left(\theta,z\right),\ \left(\theta,z\right)\in\omega$, if we restrict $\mathcal{F}_{\delta}:\Omega_{F}^{\delta}\mapsto\mathbb{R}^{3}$, then  $\left(\mathcal{F}_{\delta}\left(\xi\right),\xi,\boldsymbol{\xi}\right)$ is a valid test function and 
\begin{equation}
\left(\mathcal{F}_{\delta}\left(\xi\right),\xi,\boldsymbol{\xi}\right)\in V_{test}^{\delta}.
\end{equation}
\end{remark}



\subsection{The energy balance}\label{ssec:energy-balance}
Let us multiply the fluid equation \eqref{eqn:fluid} by $\mathbf{u}$ to obtain
\begin{equation}\label{eqn:test-u-fluid-formal}
\int_{\Omega_{F}^{\eta}\left(t\right)}\partial_{t}\mathbf{u}\cdot\mathbf{u}\ dx+\int_{\Omega_{F}^{\eta}\left(t\right)}\left(\mathbf{u}\cdot\nabla\right)\mathbf{u}\cdot\mathbf{u}\ dx=\int_{\Omega_{F}^{\eta}\left(t\right)}\text{div}\ \sigma \cdot\mathbf{u}\ dx
\end{equation}
By using Reynolds' transport theorem, we have 
\begin{equation}\label{Rey}
\begin{aligned}\int_{\Omega_{F}^{\eta}\left(t\right)}\partial_{t}\mathbf{u}\cdot\mathbf{u}\ dx & =\int_{\Omega_{F}^{\eta}\left(t\right)}\partial_{t}\frac{\left|\mathbf{u}\right|^{2}}{2}\ dx=\frac{d}{dt}\int_{\Omega_{F}^{\eta}\left(t\right)}\frac{\left|\mathbf{u}\right|^{2}}{2}dx-\int_{\Gamma\left(t\right)}\frac{\left|\mathbf{u}\right|^{2}}{2}\mathbf{u}\cdot\mathbf{n}\ dA.
\end{aligned}
\end{equation}
We use integration by parts and the divergence-free condition to rewrite the convective term in the following way:
\begin{equation}\label{conv}
\begin{aligned}\int_{\Omega_{F}^{\eta}\left(t\right)}\left(\mathbf{u}\cdot\nabla\right)\mathbf{u}\cdot\mathbf{u}\ dx =\int_{\partial\Omega_{F}^{\eta}\left(t\right)}\frac{\left|\mathbf{u}\right|^{2}}{2}\mathbf{u}\cdot\mathbf{n}\ dA =\int_{\Gamma\left(t\right)}\frac{\left|\mathbf{u}\right|^{2}}{2}\mathbf{u}\cdot\mathbf{n}\ dA+\int_{\Gamma_{in/out}}\frac{\left|\mathbf{u}\right|^{2}}{2}\mathbf{u}\cdot\mathbf{n}\ dA.
\end{aligned}
\end{equation}

Now, concerning the right hand side of \eqref{eqn:test-u-fluid-formal}, we have
\begin{multline}\label{RHS}
\int_{\Omega_{F}^{\eta}\left(t\right)}\text{div}\ \sigma \cdot\mathbf{u}\ dx =\int_{\Omega_{F}^{\eta}\left(t\right)}\text{div}\left(\sigma\mathbf{u}\right)\ dx-\int_{\Omega_{F}^{\eta}\left(t\right)}\sigma:\nabla\mathbf{u}\ dx\\
=\int_{\partial\Omega_{F}^{\eta}\left(t\right)}\sigma\mathbf{u}\cdot\mathbf{n}\ dA-\int_{\Omega_{F}^{\eta}\left(t\right)}\left|\nabla\mathbf{u}\right|^{2}\ dx =\int_{\Gamma\left(t\right)}\sigma\mathbf{u}\cdot\mathbf{n} \ dA+\int_{\Gamma_{in/out}}\sigma\mathbf{u}\cdot\mathbf{n}\ dA-\int_{\Omega_{F}^{\eta}\left(t\right)}\left|\nabla\mathbf{u}\right|^{2}\ dx .
\end{multline}
Putting together \eqref{Rey}--\eqref{RHS}, the relation \eqref{eqn:test-u-fluid-formal} becomes 
\begin{equation*}
    \begin{aligned}\frac{d}{dt}\int_{\Omega_{F}^{\eta}\left(t\right)}\frac{\left|\mathbf{u}\right|^{2}}{2}\ dx+\int_{\Omega_{F}^{\eta}\left(t\right)}\left|\nabla\mathbf{u}\right|^{2}\ dx=
\int_{\Gamma\left(t\right)}\sigma\mathbf{u}\cdot\mathbf{n}\ dA+\int_{\Gamma_{in/out}}\sigma\mathbf{u}\cdot\mathbf{n}\ dA-\int_{\Gamma_{in/out}}\frac{\left|\mathbf{u}\right|^{2}}{2}\mathbf{u}\cdot\mathbf{n}\ dA.
\end{aligned}
\end{equation*}
The first term containing $
\Gamma(t)$ will cancel with a future-one from the elastic equation. On $\Gamma_{in}$, since $\mathbf{n}=(-1,0,0)$, we have 
\begin{equation*}
\begin{aligned}\int_{\Gamma_{in}}\sigma\mathbf{u}\cdot\mathbf{n}\ dA-\int_{\Gamma_{in}}\frac{\left|\mathbf{u}\right|^{2}}{2}\mathbf{u}\cdot\mathbf{n}\ dA  =
\int_{\Gamma_{in}}\left(\partial_{1}\mathbf{u}^{1}-p-\frac{\left|\mathbf{u}\right|^{2}}{2}\right)\cdot\mathbf{u}^{1}\ dA  =
-P_{in}\left(t\right)\int_{\Gamma_{in}}\mathbf{u}\cdot\mathbf{n}\ dA
\end{aligned}
\end{equation*}
 where we used the fact that,  since $\mathbf{u}^{2}=\mathbf{u}^{3}=0$ around $\Gamma_{in}$ and $\text{div} \ \mathbf{u}=0$ it follows that $\partial_{1}\mathbf{u}^{1}=0$. Thus, we obtain
\begin{multline}\label{eqn:mult-u}
\frac{d}{dt}\int_{\Omega_{F}^{\eta}\left(t\right)}\frac{\left|\mathbf{u}\right|^{2}}{2}\ dx+\int_{\Omega_{F}^{\eta}\left(t\right)}\left|\nabla\mathbf{u}\right|^{2}\ dx+P_{in/out}\left(t\right)\int_{\Gamma_{in/out}}\mathbf{u}\cdot\mathbf{n}\ dA=\int_{\omega}J\left(t,z,\theta\right)\left(\sigma\mathbf{n}\circ\phi_{\eta}\right)\cdot\mathbf{e}_{r}\partial_{t}\eta\ dA,
\end{multline}
where we performed a change of variables to obtain the last equality. Now we multiply \eqref{eqn:wave} by $\partial_{t} \eta$ and integrate on $\omega$ to get that 
\begin{equation}\label{eqn:mult-partialteta}
   \begin{aligned}\frac{d}{dt}\int_{\omega}\frac{1}{2}\left(\partial_{t}\eta\right)^{2}dA+\int_{\omega}\left|\nabla^{2}\eta\right|^{2}dA =
-\int_{\omega}J_{\eta}\left(t,\theta,z\right)\left(\sigma\mathbf{n}\circ\phi_{\eta}\right)\cdot\mathbf{e}_{r}\partial_{t}\eta\ dA  +\int_{\omega}\mathbb{S}\left(t,R,\theta,z\right)\mathbf{e}_{r}\cdot\mathbf{e}_{r}\partial_{t}\eta\ dA.
\end{aligned}
\end{equation}
Let us also multiply \eqref{eqn:thick} by $\partial_{t}\mathbf{d}$ and integrate to obtain that 
\begin{equation}\label{eqn:mult-partialtd}
   \frac{d}{dt}\int_{\Omega_{S}}\frac{1}{2}\left(\partial_{t}\mathbf{d}\right)^{2}\ dx+\int_{\Omega_{S}}\left(\left|\nabla\mathbf{d}\right|^{2}+\left|\nabla\partial_{t}\mathbf{d}\right|^{2}\right)\ dx=\int_{\partial\Omega_{S}}\mathbb{S}\mathbf{e}_{r}\cdot\partial_{t}\mathbf{d}\ dA .
\end{equation}
We sum  \eqref{eqn:mult-u}, \eqref{eqn:mult-partialteta}, \eqref{eqn:mult-partialtd} 
and make use of the coupling condition \eqref{eqn:dynamic-coupling}
to obtain the following energy balance:
\begin{equation}\label{eqn:energy-balance}
\frac{d}{dt}E\left(t\right)+D\left(t\right)=\pm P_{in/out}\left(t\right)\int_{\Gamma_{in/out}}\mathbf{u}\cdot\mathbf{n}\ dA
\end{equation}
where we denote the energy by
\begin{equation}\label{eqn:energy-generic}
E\left(t\right):=E_{kin}\left(t\right)+E_{el}\left(t\right).    
\end{equation}
The kinetic energy $E_{kin}$ and the elastic energy $E_{el}$ are given by
\begin{equation}\label{eqn:en-kinetic}
E_{kin}\left(t\right):=\frac{1}{2}\left(\int_{\Omega_{F}^{\eta}\left(t\right)}\left|\mathbf{u}\right|^{2}dx+\int_{\omega}\left(\partial_{t}\eta\right)^{2}dA+\int_{\Omega_{S}}\left|\partial_{t}\mathbf{d}\right|^{2}dx\right),
\end{equation}
\begin{equation}\label{eqn:en-elastic}
E_{el}\left(t\right):=\frac{1}{2}\left(\int_{\omega}\left|\nabla^{2}\eta\right|^{2}dA+\int_{\Omega_{S}}\left(\left|\nabla\mathbf{d}\right|^{2}+\left|\text{div}\mathbf{d}\right|^{2}\right)dx\right).
\end{equation}
The diffusive energy $D$ is given by
\begin{equation}\label{eqn:def-diff}
D\left(t\right):=\int_{\Omega_{F}^{\eta}\left(t\right)}\left|\nabla\mathbf{u}\right|^{2}dx+\int_{\Omega_{S}}\left|\nabla\partial_{t}\mathbf{d}\right|^{2}dx .
\end{equation}
\subsection{Uniform energy estimates}\label{ssec:unif-energ-est}
Let us assume the time-periodicity of the unknowns and hence $E\left(0\right)=E\left(T\right)$. Let us  integrate \eqref{eqn:energy-balance} in time to obtain 
\begin{equation}\label{eqn:int-diff}
\int_{I}D\left(t\right)dt=\int_{I}P_{in/out}\left(t\right)\int_{\Gamma_{in/out}}\left(\mathbf{u}\cdot\mathbf{n}\right) \ dA \ dt.
\end{equation}

\paragraph{Trace estimates.}
{
Since $\mathbf d$ vanishes on $\Gamma^s_{\mathrm{in/out}}$, so does
$\partial_t\mathbf d$. Hence, for almost every $t\in I$, Poincar\'e's
inequality gives
\begin{equation}\label{eqn:est-partial-t-d-D-t}
\int_{\Omega_S}|\partial_t\mathbf d|^2\,dx
\lesssim_{\Omega_S}
\int_{\Omega_S}|\nabla\partial_t\mathbf d|^2\,dx
\leq D(t).
\end{equation}
Moreover,
\[
\mathbf d(t,R\cos\theta,R\sin\theta,z)
=
\eta(t,\theta,z)\mathbf e_r,
\]
and therefore
\[
\partial_t\mathbf d(t,R\cos\theta,R\sin\theta,z)
=
\partial_t\eta(t,\theta,z)\mathbf e_r.
\]
The trace theorem on the fixed domain $\Omega_S$, together with
\eqref{eqn:est-partial-t-d-D-t}, yields
\begin{equation}\label{est-partial-t-eta-D-t}
\begin{aligned}
\int_\omega|\partial_t\eta|^2\,dA
&\lesssim
\int_{\Omega_S}
\left(
|\partial_t\mathbf d|^2+
|\nabla\partial_t\mathbf d|^2
\right)\,dx  \lesssim D(t).
\end{aligned}
\end{equation}
Next, we write 
\begin{equation}\label{eqn:int-parts-u-trace}
\begin{aligned}\mathbf{u}\left(r\cos\theta,r\sin\theta,z\right)= & \mathbf{u}\left(\left(R+\eta\right)\cos\theta,\left(R+\eta\right)\sin\theta,z\right)-\int_{r}^{R+\eta}\frac{d}{d\tilde{r}}\mathbf{u}\left(\tilde{r}\cos\theta,\tilde{r}\sin\theta,z\right)d\tilde{r}\\
= & \partial_{t}\eta\ \mathbf{e}_{r}-\int_{r}^{R+\eta}\nabla\mathbf{u}\left(\tilde{r}\cos\theta,\tilde{r}\sin\theta,z\right)\mathbf{e}_{r}d\tilde{r}.
\end{aligned}
\end{equation}
Multiplying \eqref{eqn:int-parts-u-trace} by $r$, integrating over $\left(0, R+\eta \right)$ and applying the Cauchy-Schwarz inequality, we get
\begin{multline}\label{eqn:radial-fluid-estimate1}
\int_{0}^{R+\eta}r\left|\mathbf{u}\left(r\cos\theta,r\sin\theta,z\right)\right|^{2}dr\le  \left|\partial_{t}\eta\right|^{2}\left(R+\eta\right)^{2}+\int_{0}^{R+\eta}2r\left(\int_{r}^{R+\eta}\left|\nabla\mathbf{u}\right|d\tilde{r}\right)^{2}dr\\
\le  \left|\partial_{t}\eta\right|^{2}\left(R+\eta\right)^{2}+\int_{0}^{R+\eta}\left(\int_{r}^{R+\eta}\sqrt{2\tilde{r}}\left|\nabla\mathbf{u}\right|d\tilde{r}\right)^{2}dr\\
\le  \left|\partial_{t}\eta\right|^{2}\left(R+\eta\right)^{2}+ \int_{0}^{R+\eta}\left(\int_{r}^{R+\eta}2\tilde{r}\left|\nabla\mathbf{u}\right|^{2}d\tilde{r}\right)\int_{r}^{R+\eta}1d\tilde{r}dr\\
\le  \left|\partial_{t}\eta\right|^{2}\left(R+\eta\right)^{2}+\left(\int_{0}^{R+\eta}2\tilde{r}\left|\nabla\mathbf{u}\right|^{2}d\tilde{r}\right)\frac{\left(R+\eta\right)^{2}}{2}
\lesssim_{R}  \left|\partial_{t}\eta\right|^{2}+\left(\int_{0}^{R+\eta}2\tilde{r}\left|\nabla\mathbf{u}\right|^{2}d\tilde{r}\right).
\end{multline}
Integrating \eqref{eqn:radial-fluid-estimate1} over $\omega$, using
$\|\eta\|_{L^\infty(I\times\omega)}<R$, and recalling
\eqref{est-partial-t-eta-D-t}, we obtain, for almost every $t\in I$,
\begin{equation}\label{eqn:est-u-D-t}
\begin{aligned}\int_{\Omega_{F}^{\eta}\left(t\right)}\left|\mathbf{u}\right|^{2}dx\lesssim_{R}  \int_{\omega}\left|\partial_{t}\eta\right|^{2}dA+\int_{\Omega_{F}^{\eta}\left(t\right)}\left|\nabla\mathbf{u}\right|^{2}dx
\lesssim_{R}  D\left(t\right)
\end{aligned}
\end{equation}

\paragraph{Diffusion estimate.}
Let us recall \eqref{eqn:int-diff}. Observe that on the inlet and outlet,
$\mathbf{n}=\mp\mathbf{e}_{z}$ on $\Gamma_{in/out}$ implies $|\mathbf{u}\cdot\mathbf{n}|=|u_{z}|$ on $\Gamma_{in/out}$. Writing $u_{z}\left(t,x,y,0\right)=-\int_{0}^{z}\partial_{z}u_{z}+u_{z}\left(t,x,y,z\right)$ for any $z\in (0,L)$  we may estimate 
\begin{equation}
\begin{aligned}\int_{\Gamma_{in}}\left|u_{z}\right|^{2}dxdy & \lesssim\int_{\Gamma_{in}}\left(\int_{0}^{z}\partial_{z}u_{z}\right)^{2}dxdy+\int_{\Gamma_{in}}\left|u_{z}\left(t,x,y,z\right)\right|^{2}dxdy\\
 & \lesssim\int_{\Gamma_{in}}z\left(\int_{0}^{z}\left|\partial_{z}u_{z}\right|^{2}\right)dxdy+\int_{\Gamma_{in}}\left|u_{z}\left(t,x,y,z\right)\right|^{2}dxdy\\
 & \lesssim_{L}\int_{\Omega_{F}^{\eta}\left(t\right)}\left|\partial_{z}u_{z}\right|^{2}+\int_{\Gamma_{in}}\left|u_{z}\left(t,x,y,z\right)\right|^{2}dxdy
\end{aligned}
\end{equation}
and integrating once more from $z=0$ to $z=L$ we obtain that 
\[\left\Vert \mathbf{u}\right\Vert _{L^{2}\left(\Gamma_{in/out}\right)}\lesssim\left\Vert \mathbb{D}\mathbf{u}\right\Vert _{L^{2}\left(\Omega_{F}^{\eta}\right)}+\left\Vert \mathbf{u}\right\Vert _{L^{2}\left(\Omega_{F}^{\eta}\right)}.
\]

 Consequently, using \eqref{eqn:est-u-D-t}, we find
\begin{equation}
\begin{aligned}\int_{I}D(t)\,dt & \lesssim\|P_{in/out}\|_{L^{2}(I)}\|\mathbf{u}\cdot\mathbf{n}\|_{L^{2}(I\times\Gamma_{\mathrm{in/out}})}\\
 & \lesssim\|P_{in/out}\|_{L^{2}(I)}\left(\|\mathbf{u}\|_{L^{2}(I\times\Omega_{F}^{\eta})}^{2}+\|\mathbb{D}\mathbf{u}\|_{L^{2}(I\times\Omega_{F}^{\eta})}^{2}\right)^{1/2}\\
 & \lesssim\|P_{\mathrm{in/out}}\|_{L^{2}(I)}\left(\int_{I}D(t)\,dt\right)^{1/2}.
\end{aligned}
\end{equation}

By Young's inequality, for every $\kappa>0$,
\[
\int_I D(t)\,dt
\leq
\kappa\int_I D(t)\,dt
+
C_\kappa\|P_{\mathrm{in/out}}\|_{L^2(I)}^2.
\]
Choosing $\kappa>0$ sufficiently small and absorbing the first
term on the right-hand side, we obtain the diffusion estimate
\begin{equation}\label{eqn:diffusion-estimate}
\int_I D(t)\,dt
\lesssim
\int_I|P_{\mathrm{in/out}}(t)|^2\,dt.
\end{equation}
}
We may now update the previous trace estimates
\eqref{eqn:est-partial-t-d-D-t}-\eqref{eqn:est-u-D-t} to 
 
\begin{equation}\label{eqn:diff1}
\int_{I}\int_{\Omega_{F}^{\eta}\left(t\right)}\left|\mathbf{u}\right|^{2}dxdt\lesssim\int_{I}\left|P_{in/out}\left(t\right)\right|^{2}dt.
\end{equation}
and 
\begin{equation}\label{eqn:diff2}
\begin{aligned}\int_{I}\int_{\Omega_{S}}\left|\partial_{t}\mathbf{d}\right|^{2}dxdt\lesssim\int_{I}\int_{\Omega_{S}}\left|\nabla\partial_{t}\mathbf{d}\right|^{2}dxdt\lesssim\int_{I}\left|P_{in/out}\left(t\right)\right|^{2}dt,\end{aligned}
\end{equation}
and

\begin{equation}\label{eqn:diff3}
\int_{I}\int_{\omega}\left|\partial_{t}\eta\right|^{2}dAdt\lesssim\int_{I}D\left(t\right)dt\lesssim\int_{I}\left|P_{in/out}\left(t\right)\right|^{2}dt.
\end{equation}

Adding \eqref{eqn:diff1}, \eqref{eqn:diff2}, \eqref{eqn:diff3}
we get that 
\begin{equation}\label{eqn:kin-diff}
\int_{I}E_{kin}\left(t\right)dt\lesssim\int_{I}\left|P_{in/out}\left(t\right)\right|^{2}dt .
\end{equation}

\paragraph{Uniform in time estimates.}
By the  mean-value theorem there exists $t_0 \in I$ for which \[E\left(t_{0}\right)=\frac{1}{T}\int_{I}E\left(t\right)\ dt=\fint_{I} E(t)\ dt, \] and integrating \eqref{eqn:energy-balance} from $t_0$ to an arbitrary $t\in I$ and using \eqref{eqn:kin-diff} making use of Young's inequality and the trace estimate we find that 
\begin{equation}\label{eqn:en-ineq-1}
  \begin{aligned}\sup_{t\in I}E\left(t\right) & \le\fint_{I}E\left(t\right)dt+\int_{I}\int_{\Omega_{F}^{\eta}\left(t\right)}\left|\nabla\mathbf{u}\right|^{2}dxdt+\int_{I}\left|P_{in/out}\left(t\right)\right|\left|\int_{\Gamma_{in/out}}\mathbf{u}\cdot\mathbf{n}\ dA\ dt\right|\\
 & \lesssim\fint_{I}E_{el}\left(t\right)dt+\int_{I}\left|P_{in/out}\left(t\right)\right|^{2}dt.
\end{aligned}
\end{equation}

Now, in order to estimate $E_{el}$, it would be convenient to test \eqref{eqn:wave} by $ 
\eta$ and \eqref{eqn:thick} by $\mathbf{d}$. The test function for the fluid equation is then $\mathcal{F}_{\eta} \left(\eta \right)$ (see also Remark~\ref{rmk:test-functions-extension} and with the computations and notations of Subsection~\ref{ssec:deriving-weak-form}). Thus, we obtain
\begin{multline}\label{eqn:est-E-el-long}
E_{el}\left(t\right)\le  \int_{I}\int_{\Omega_{F}^{\eta}\left(t\right)}\left|\mathbf{u}\cdot\partial_{t}\mathcal{F}_{\eta}\left(\eta\right)\right|+\left|\nabla\mathbf{u}\cdot\nabla\mathcal{F}_{\eta}\left(\eta\right)\right|\ dx\ dt+\int_{I}\left|b\left(t,\mathbf{u,}\mathbf{u},\mathcal{F}_{\eta}\left(\eta\right)\right)\right|\ dt \\ +\int_{I}\int_{\Omega_{S}}\left(\partial_{t}\mathbf{d}\right)^{2}\left|\nabla\partial_{t}\mathbf{d}\cdot\nabla\mathbf{d}\right|\ dA\ dt +
  \int_{I}\int_{\omega}\frac{1}{2}\left|\left(\partial_{t}\eta\right)^{2}\left(\eta+R\right)\right|+\left(\partial_{t}\eta\right)^{2}\ dA\ dt+
  \int_{I}\left\langle F\left(t\right),\mathcal{F}_{\eta}\left(\eta\right)\right\rangle _{\Gamma_{in/out}}dt
=:  \sum_{k=1}^{8}T_{k}.
\end{multline}
We will now estimate only the most tedious terms, based on the operator $\mathcal{F}_{\eta}$ introduced in Proposition \ref{prop:estimates-extension-operator}. The remaining terms can be estimated similarly and simpler. Using Proposition~\ref{prop:estimates-extension-operator}, the estimate \eqref{eqn:kin-diff} and  the embedding $H_{x}^{2}\hookrightarrow W^{1,q}_{x} \hookrightarrow L_{x}^{\infty}$  valid for any $1\le q <\infty$ in 2D, we obtain
\begin{equation}\label{eqn:est-T1}
   \begin{aligned}\left|T_{1}\right|\le & \left\Vert \mathbf{u}\right\Vert _{L_{t}^{2}L_{x}^{2}}\left\Vert \partial_{t}\mathcal{F}_{\eta}\left(\eta\right)\right\Vert _{L_{t}^{2}L_{x}^{2}}
\lesssim \left\Vert P_{in/out}\right\Vert _{L_{t}^{2}}\left\Vert \partial_{t}\left(\eta^{2}\right)\right\Vert _{L_{t}^{2}L_{x}^{2}}\\
\lesssim & \left\Vert P_{in/out}\right\Vert _{L_{t}^{2}}\left\Vert \partial_{t}\eta\right\Vert _{L_{t}^{2} L^2_{x}}\left\Vert \eta\right\Vert _{L_{t}^{\infty}L_{x}^{\infty}}
\lesssim \left\Vert P_{in/out}\right\Vert _{L_{t}^{2}}^{2}\left(\sup_{t\in I}E_{el}\left(t\right)\right)^{1/2}.
\end{aligned}
\end{equation}
Next, we estimate the convective term $T_3$ for which we have 
\begin{equation}\label{eqn:est-T3}
    \begin{aligned}T_{3} & \le\int_{I}\int_{\Omega_{F}^{\eta}\left(t\right)}\left|\mathbf{u}\right|\left|\nabla\mathbf{u}\right|\left|\mathcal{F}_{\eta}\left(\eta\right)\right|+\left|\mathbf{u}\right|^{2}\left|\nabla\mathcal{F}_{\eta}\left(\eta\right)\right|dxdt\\
 & =:T_{3,1}+T_{3,2},
\end{aligned}
\end{equation}
and then 
\begin{equation}\label{eqn:est-T31-T32}
\begin{aligned}T_{3,1}+T_{3,2}\le & \left\Vert \mathbf{u}\right\Vert _{L_{t}^{2}L_{x}^{6}}\left\Vert \nabla\mathbf{u}\right\Vert _{L_{t}^{2}L_{x}^{2}}\left\Vert \mathcal{F}_{\eta}\left(\eta\right)\right\Vert _{L_{t}^{\infty}L_{x}^{3}}+\left\Vert \mathbf{u}\right\Vert _{L_{t}^{2}L_{x}^{6}}^{2}\left\Vert \mathcal{F}_{\eta}\left(\eta\right)\right\Vert _{L_{t}^{\infty}W_{x}^{1,3/2}}\\
\lesssim & \left\Vert P_{in/out}\right\Vert _{L_{t}^{2}}^{2}\left(\left\Vert \eta^{2}\right\Vert _{L_{t}^{\infty}L_{x}^{3}}+\left\Vert \left|\nabla\eta\right|\eta\right\Vert _{L_{t}^{\infty}L_{x}^{3/2}}\right)\\
\lesssim & \left\Vert P_{in/out}\right\Vert _{L_{t}^{2}}^{2}\left(\left(\sup_{t\in I}E_{el}\left(t\right)\right)+\left(\sup_{t\in I}E_{el}\left(t\right)\right)^{1/2+1/2}\right)\\
\lesssim & \left\Vert P_{in/out}\right\Vert _{L_{t}^{2}}^{2}\left(\sup_{t\in I}E_{el}\left(t\right)\right)
\end{aligned}
\end{equation}
Thus, for \emph{sufficiently small} $\left\Vert P_{in/out}\right\Vert _{L_{t}^{2}}$ from
\eqref{eqn:est-E-el-long}-\eqref{eqn:est-T31-T32} it follows that 
\begin{equation}\label{eqn:E-El-new}
    \sup_{t\in I}E_{el}\left(t\right)\lesssim\left\Vert P_{in/out}\right\Vert _{L_{t}^{2}L_{x}^{2}}^{2}.
\end{equation}
From \eqref{eqn:E-El-new} and \eqref{eqn:en-ineq-1}, we have proved the following
\begin{proposition}   
There exists a positive constant $C_0=C_0 (\texttt{data})$ for which $\left\Vert P_{in/out}\right\Vert _{L_{t}^{2}}\le C_{0}$ implies that 
\begin{equation}\label{eqn:formal-en-bound}
    \sup_{t\in I}E\left(t\right) + \int_{I}D(t)dt
    \lesssim\left\Vert P_{in/out}\right\Vert _{L_{t}^{2}}^{2}.
\end{equation}
\end{proposition}
\section{Compactness}\label{sec:compactness}
In this section we aim to prove the $L^2$ compactness for a sequence of fluid velocities. 
We will proceed in the  same spirit as in   \cite[Proposition 3.8]{LR14},  with some non-trivial adjustments to the current setting.

Let us assume that we have a sequence of solutions $\left(\mathbf{u}_{n},\eta_{n},\mathbf{d}_{n}\right)\in V_{soln}^{\eta_n}$ such that the corresponding energies of the system $E_n$ are uniformly bounded, that is $
\sup_{n\ge1}\sup_{t\in I}E_{n}\left(t\right)\lesssim c<\infty.$
Of course, there exists a weak limit $\left(\mathbf{u},\eta,\mathbf{d}\right)\in V_{soln}^{\eta}$ and a subsequence for which 
$
\left(\mathbf{u}_{n},\eta_{n},\mathbf{d}_{n}\right)\rightharpoonup\left(\mathbf{u},\eta,\mathbf{d}\right)$
in the corresponding spaces. As usual when dealing with the Navier-Stokes equation, we need to prove that $\mathbf{u}_{n}\to\mathbf{u}$ in $L^{2}_{t,x}$. Due to the coupled nature of the solutions we shall in fact  prove  that $\left(\mathbf{u}_{n},\partial_{t}\eta_{n},\partial_{t}\mathbf{d}_{n}\right)\to\left(\mathbf{u},\partial_{t}\eta,\partial_{t}\mathbf{d}\right)$ in $L^{2}_{t,x}$ for a subsequence. 
More precisely we have 
\begin{proposition}\label{prop:compactness}
    Let $\left(\mathbf{u}_{n},\eta_{n},\mathbf{d}_{n}\right)\in V_{soln}^{\eta_n}$ be a sequence of solutions for \eqref{eqn:weak-form} with finite energy $E_n$ and diffusion $D_n$
    \begin{equation}\label{eqn:hyp-bound-En}
    \sup_{n\ge1}\left(\sup_{t\in I}E_{n}\left(t\right)+\int_{I}D_{n}\left(t\right)dt\right)\le c_{0}<\infty
    \end{equation}
    where $E_{n}\left(t\right)$
    and $D_{n}\left(t\right)$
    are the analogue of  $E$ and $D$ from \eqref{eqn:energy-generic} and \eqref{eqn:def-diff}. The constant $c_0$ is fixed sufficiently small to ensure that $\left\Vert \eta_{n}\right\Vert _{L_{t,x}^{\infty}}<R$.
Then the sequence $\left(\mathbf{u}_{n},\partial_{t}\eta_{n},\partial_{t}\mathbf{d}_{n}\right)_{n\ge1}$ has a strongly convergent subsequence in $L^{2}_{t,x}$.
\end{proposition}
\begin{proof}

    The idea is to prove the following convergences 
    \begin{equation}
    \label{eqn:compact-1}
\begin{aligned}\int_{I}\int_{\Omega_{F}^{\eta_{n}}\left(t\right)}\mathbf{u}_{n}\cdot\mathcal{F}_{\eta_{n}}\left(\partial_{t}\eta_{n}\right)dxdt+\int_{I}\int_{\omega}\left(\partial_{t}\eta_{n}\right)^{2}dAdt & +\int_{I}\int_{\Omega_{S}}\left(\partial_{t}\mathbf{d}_{n}\right)^{2}dxdt\\ \to
\int_{I}\int_{\Omega_{F}^{\eta}\left(t\right)}\mathbf{u}_{n}\cdot\mathcal{F}_{\eta}\left(\partial_{t}\eta\right)dxdt+\int_{I}\int_{\omega}\left(\partial_{t}\eta\right)^{2}dAdt & +\int_{I}\int_{\Omega_{S}}\left(\partial_{t}\mathbf{d}\right)^{2}dxdt
\end{aligned}
    \end{equation}

and
\begin{equation}\label{eqn:compact-2}
\begin{aligned}\int_{I}\int_{\Omega_{F}^{\eta_{n}}\left(t\right)}\mathbf{u}_{n}\cdot\left(\mathbf{u}_{n}-\mathcal{F}_{\eta_{n}}\left(\partial_{t}\eta_{n}\right)dxdt\right)\to\int_{I}\int_{\Omega_{F}^{\eta}\left(t\right)}\mathbf{u}\cdot\left(\mathbf{u}-\mathcal{F}_{\eta}\left(\partial_{t}\eta\right)\right)dxdt\end{aligned}
\end{equation}
as $n\to \infty$.
By adding \eqref{eqn:compact-1} and \eqref{eqn:compact-2} our claim follows immediately.

The proof of the two  convergences is presented in Subsection~\ref{ssec:1st-compactness} and Subsection~\ref{ssec:2nd-compactness}.
\end{proof}

Let us note that by an eventual diagonal argument we may prove that there exists a set $I_0 \subseteq I$ which is dense in $I$ and for which
\begin{equation}
\left(\mathbf{u}_{n}\left(t\right),\eta_{n}\left(t\right),\mathbf{d}_{n}\left(t\right)\right)\rightharpoonup\left(\mathbf{u}\left(t\right),\eta\left(t\right),\mathbf{d}\left(t\right)\right)
\end{equation}
in the appropriate function spaces of spatial variables for all $t\in I_0$.
\subsection{Proof of  \eqref{eqn:compact-1}}\label{ssec:1st-compactness}

Let us consider arbitrary but fixed $\xi\in C_{0}^{\infty}\left(\omega\right)$ and $\boldsymbol{\xi}\in C^{\infty}\left(\Omega_{S}\right)$ such that $\text{tr}_{\Gamma}\boldsymbol{\xi}=\xi\mathbf{e}_{r}$ in $\omega$.

We consider the functionals
\begin{equation}
c_{\xi, \boldsymbol{\xi},n}\left(t\right):=\int_{\Omega_{F}^{\eta_{n}}\left(t\right)}\mathbf{u}_{n}\cdot\mathcal{F}_{\eta_{n}}\left(\xi\right)dx+\int_{\omega}\partial_{t}\eta_{n} \xi\  dA+\int_{\Omega_{S}}\partial_{t}\mathbf{d}_{n}\cdot\boldsymbol{\xi}\ dA ,
\end{equation}
and 
\begin{equation}
c_{\xi, \boldsymbol{\xi}}\left(t\right):=\int_{\Omega_{F}^{\eta}\left(t\right)}\mathbf{u}\cdot\mathcal{F}_{\eta}\left(\xi\right)dx+\int_{\omega}\partial_{t}\eta \xi \ dA+\int_{\Omega_{S}}\partial_{t}\mathbf{d}\cdot\boldsymbol{\xi}\ dA.
\end{equation}
Our aim is to prove that 
\begin{equation}
    \sup_{t\in I}\sup_{\left\Vert \xi, \boldsymbol{\xi}\right\Vert _{L_{x}^{2}}\le1}\left|c_{\xi, \boldsymbol{\xi},n}\left(t\right)-c_{\xi, \boldsymbol{\xi}}\left(t\right)\right|\xrightarrow{n\to\infty}0.
\end{equation}
We follow the following steps.
\begin{enumerate}
    \item We prove that for each fixed $\xi, \boldsymbol{\xi}\in H^{2},\left\Vert (\xi, \boldsymbol{\xi})\right\Vert _{H^{2}_{x}}\le1$ it follows that 
    \begin{equation}
        c_{\xi, \boldsymbol{\xi},n}\left(t\right)\to c_{\xi, \boldsymbol{\xi}}\left(t\right)\quad\text{\text{in}}\ C\left(\overline{I}\right).
    \end{equation}
We claim that there exists $\alpha \in (0,1)$ for which
\begin{equation}\label{eqn:holder-cont-cbn}
 \sup_{\left\Vert \left(\xi, \boldsymbol{\xi}\right)\right\Vert _{H_{x}^{2}}\le1}\left|c_{\xi, \boldsymbol{\xi},n}\left(t\right)-c_{\xi, \boldsymbol{\xi},n}\left(s\right)\right|\lesssim\left|t-s\right|^{\alpha}\quad\forall \ t,s\in I.
\end{equation}
Indeed, using the weak formulation \eqref{eqn:weak-form} we find  that\footnote{Recall the definition of the convective term from \eqref{eqn:functionals}.}
\begin{equation}
\begin{aligned}\left|c_{\xi, \boldsymbol{\xi},n}\left(t\right)-c_{\xi, \boldsymbol{\xi},n}\left(s\right)\right|\le & \left|\int_{s}^{t}\int_{\Omega_{F}^{\eta_{n}}}-\mathbf{\mathbf{u}}_{n}\cdot\mathcal{F}_{\eta_{n}}\left(\xi\right)+\nabla\mathbf{\mathbf{u}}_{n}:\nabla\mathcal{F}_{\eta_{n}}\left(\xi\right)dxdt\right|+\\
 & \left|\int_{s}^{t}b\left(\tau,\mathbf{\mathbf{u}}_{n},\mathbf{\mathbf{u}}_{n},\mathcal{F}_{\eta_{n}}\left(\xi\right)\right)d\tau\right|+ \text{other-terms}
\end{aligned}
\end{equation}
for all $t,s \in I$.
Let us estimate the most tedious term, the convective term, the other  following in a similar  way. 
As usual, by interpolation we get that $\mathbf{u}_{n}\in L_{t}^{\infty}L_{x}^{2}\cap L_{t}^{2}W_{x}^{1,2}\hookrightarrow L_{t,x}^{10/3}$.

We have 
\begin{equation} 
\begin{aligned}&\left|\int_{s}^{t}b\left(\tau,\mathbf{\mathbf{u}}_{n},\mathbf{\mathbf{u}}_{n},\mathcal{F}_{\eta_{n}}\left(\xi\right)\right)d\tau\right| \\ &\le \int_{s}^{t}\int_{\Omega_{F}^{n}\left(t\right)}\left|\mathbf{\mathbf{u}}_{n}\right|\left|\nabla\mathbf{\mathbf{u}}_{n}\right|\left|\mathcal{F}_{\eta_{n}}\left(\xi\right)\right|dxd\tau+
\int_{s}^{t}\int_{\Omega_{F}^{n}\left(t\right)}\left|\mathbf{\mathbf{u}}_{n}\right|^{2}\left|\nabla\mathcal{F}_{\eta_{n}}\left(\xi\right)\right|dxd\tau
\le  C_{1}+C_{2}.
\end{aligned}
\end{equation}
 We use the energy-estimates to get that
\begin{equation}
 \begin{aligned}\left|C_{1}\right|\le & \left\Vert \mathbf{u}_{n}\right\Vert _{L_{t}^{\infty}L_{x}^{2}}\left\Vert \nabla\mathbf{u}_{n}\right\Vert _{L_{t,x}^{2}}\left\Vert \mathcal{F}_{\eta_{n}}\left(\xi\right)\right\Vert _{L_{t}^{\infty}L_{x}^{\infty}}\left|t-s\right|^{1/2}\\
\lesssim & \sup_{t\in I}E_{n}\left(t\right)\cdot\left\Vert P_{in/out}\right\Vert _{L_{t}^{2}}\left(\left\Vert \eta_{n}\xi\right\Vert _{L_{t}^{\infty}L_{x}^{\infty}}\right)\left|t-s\right|^{1/2}
\lesssim  \left|t-s\right|^{1/2},
\end{aligned}
\end{equation}
and with a similar argument, we obtain
\begin{equation}
 \begin{aligned}\left|C_{2}\right| & \le\left\Vert \mathbf{u}_{n}\right\Vert _{L_{t}^{\infty}L_{x}^{2}}^{2}\left\Vert \nabla\mathcal{F}_{\eta_{n}}\left(\xi\right)\right\Vert _{L_{t,x}^{\infty}}\left|t-s\right|\lesssim\left|t-s\right|.
 \end{aligned}
\end{equation}
Hence, it is easy to see by standard estimates  that $\sup_{n\ge1}\sup_{t\in I}\left|c_{\xi, \boldsymbol{\xi},n}\left(t\right)\right|\lesssim1$ 
and also that  
$c_{\xi, \boldsymbol{\xi},n}\left(t\right)\to c_{\xi, \boldsymbol{\xi}}\left(t\right)$ in a distributional sense.
Consequently, by the Arzela-Ascoli Theorem we get that
\[c_{\xi, \boldsymbol{\xi},n}\left(t\right)\to c_{\xi, \boldsymbol{\xi}}\left(t\right)\quad\text{in}\ C\left(\overline{I}\right)\] for each fixed $\xi, \boldsymbol{\xi}$ with $\left\Vert (\xi, \boldsymbol{\xi})\right\Vert _{H^{2}\times H^{1}}\le1$.
\item 
Then, we claim the uniform convergence  w.r.t. $\left(\xi, \boldsymbol{\xi}\right)\in H^{2}\times H^{1}$, namely     \begin{equation}\label{eqn:claim-hn}
h_{n}\left(t\right):=\sup_{\left\Vert \left(\xi, \boldsymbol{\xi}\right)\right\Vert _{H^{2}\times H^{1}}\le1}\left|c_{\xi, \boldsymbol{\xi},n}\left(t\right)-c_{\xi, \boldsymbol{\xi}}\left(t\right)\right|\to0\quad\text{in}\ C\left(\overline{I}\right).
     \end{equation}
Indeed, for $t\in I_0$ we have in fact that $\mathbf{u}_{n}\left(t\right)\to\mathbf{u}\left(t\right)\ \text{in}\ H_{x}^{-1}$ due to the embedding $L^{2}\hookrightarrow\hookrightarrow H^{-1}$. Regarding the term $\int_{\Omega_{F}^{\eta_{n}}\left(t\right)}\mathbf{u}_{n}\cdot\mathcal{F}_{\eta_{n}}bdx-\int_{\Omega_{F}^{\eta}\left(t\right)}\mathbf{u}\cdot\mathcal{F}_{\eta}bdx$. Due to the fact that the extension operators $\mathcal{F}_{\eta_n, \eta}$ are defined on a steady cylinder $C$ (cf. Proposition~\ref{prop:estimates-extension-operator}) we can extend $\mathbf{u}_n, \mathbf{u}$ by zero and write 

\begin{equation}
\begin{aligned}\sup_{\left\Vert b\right\Vert _{H_{x}^{2}}\le1}\left|\int_{\Omega_{F}^{\eta_{n}}\left(t\right)}\mathbf{u}_{n}\cdot\mathcal{F}_{\eta_{n}}\left(\xi\right)dx-\int_{\Omega_{F}^{\eta}\left(t\right)}\mathbf{u}\cdot\mathcal{F}_{\eta}\left(\xi\right)dx\right| & \\ \lesssim
\sup_{\left\Vert b\right\Vert _{H_{x}^{2}}\le1}\left\Vert \mathbf{u}_{n}\left(t\right)-u\left(t\right)\right\Vert _{H_{x}^{-1}}\left\Vert \mathcal{F}_{\eta_{n}}\left(\xi\right)\right\Vert _{H_{x}^{1}}+\left\Vert \mathbf{u}\left(t\right)\right\Vert _{L_{x}^{2}}\left\Vert \mathcal{F}_{\eta_{n}}\left(\xi\right)-\mathcal{F}_{\eta}\left(\xi\right)\right\Vert _{L_{x}^{2}} & \\ \lesssim
\left\Vert \mathbf{u}_{n}-u\right\Vert _{L_{x}^{\infty}H_{x}^{-1}}+\left\Vert \left(\eta_{n}-\eta\right)\xi\right\Vert _{L_{t}^{\infty}L_{x}^{2}}+\left\Vert \left(\eta_{n}-\eta\right)\xi\right\Vert _{L_{t}^{\infty}L_{x}^{2}}^{1/2} & \\ \lesssim
\left\Vert \mathbf{u}_{n}-u\right\Vert _{L_{x}^{\infty}H_{x}^{-1}}+\left\Vert \left(\eta_{n}-\eta\right)\right\Vert _{L_{t}^{\infty}L_{x}^{2}}+\left\Vert \left(\eta_{n}-\eta\right)\right\Vert _{L_{t}^{\infty}L_{x}^{2}}^{1/2} & \to 0 ,
\end{aligned}
\end{equation}
where in the last estimate we used \eqref{eqn:cons-ext-3}.
This means  there exists a sequence of positive numbers $a_n\to 0$ such that 
\[
0\le h_{n}\left(t\right)\le a_{n}\quad\forall\ t\in I_{0}.
\]
Using now the density of $I_0$ in $I$, we can easily ensure that 

\[
h_{n}\left(t\right)\to0\quad\text{in}\ C\left(\overline{I}\right),
\]
which proves the claim \eqref{eqn:claim-hn}.
\item In this step, we aim to allow for $\left(\xi, \boldsymbol{\xi}\right)\in L_{x}^{2}\times H_{x}^{1}$. 
To this end, let us observe that the reasoning from the previous step yields the following estimate: for all $\varepsilon>0$ there exists $n(\varepsilon)\ge 1$ such that for all $n>n\left(\varepsilon\right)$ it holds that
\begin{equation}\label{eqn:ineq-bb}
\left|c_{\xi, \boldsymbol{\xi},n}\left(t\right)-c_{\xi, \boldsymbol{\xi}}\left(t\right)\right|\le\varepsilon\left(\left\Vert \xi\right\Vert _{H_{x}^{2}}+\left\Vert \boldsymbol{\xi}\right\Vert _{H_{x}^{1}}\right) \quad \forall\left(\xi, \boldsymbol{\xi}\right)\in H_{x}^{2}\times H_{x}^{1}.
\end{equation}

Now for each $\xi\in L^{2}\left(\omega\right)$ extended by zero to $\mathbb{R}^{2}$ its standard mollification in the variables $\theta,z$, denoted
$\xi_\rho$ for  $\rho>0$. We can consider its correspondent $\boldsymbol{\xi}_{\rho}\in H^{1}\left(\Omega_{S}\right)$ with $\boldsymbol{\xi}_{\rho}\left(R,\cdot\right)=\xi_{\rho}\mathbf{e}_{r}$.
Plugging them in \eqref{eqn:ineq-bb} yields 
\begin{equation}
    \begin{aligned}\left|c_{\xi_{\rho},\boldsymbol{\xi}_{\rho},n}\left(t\right)-c_{\xi_{\rho},\boldsymbol{\xi}_{\rho}}\left(t\right)\right|  \le\varepsilon\left(\left\Vert \xi_{\rho}\right\Vert _{H_{x}^{2}}+\left\Vert \boldsymbol{\xi}_{\rho}\right\Vert _{H_{x}^{1}}\right)
  \lesssim\frac{\varepsilon}{\rho^{N}}\left(\left\Vert \xi\right\Vert _{L_{x}^{2}}+\left\Vert \boldsymbol{\xi}\right\Vert _{H_{x}^{1}}\right)
\lesssim\sqrt{\varepsilon}\left(\left\Vert \xi\right\Vert _{L_{x}^{2}}+\left\Vert \boldsymbol{\xi}\right\Vert _{H_{x}^{1}}\right),
\end{aligned}
\end{equation}
for a certain positive integer $N$ and $\rho=\varepsilon^{1/(2N)}$. Letting now $\varepsilon\to 0$, we conclude that 
\begin{equation}
g_{n}\left(t\right):=\sup_{\left\Vert \left(\xi, \boldsymbol{\xi}\right)\right\Vert _{L_{x}^{2}\times H_{x}^{1}}}\left|c_{\xi, \boldsymbol{\xi},n}\left(t\right)-c_{\xi, \boldsymbol{\xi}}\left(t\right)\right|\to0\quad\text{in}\ C\left(\overline{I}\right).
\end{equation}
\item Finally, we take $\left(\xi,\boldsymbol{\xi}\right)=\left(\partial_{t}\eta_{n}\left(t,\cdot\right),\partial_{t}\mathbf{d}_{n}\left(t,\cdot\right)\right)$  and we obtain \eqref{eqn:compact-1}.
\end{enumerate}
\subsection{Proof of \eqref{eqn:compact-2}}\label{ssec:2nd-compactness}
We aim to prove \eqref{eqn:compact-2}. 
We proceed as follows:
\begin{enumerate}
\item Let $\sigma>0$ and $\delta_{\sigma}$ be such that 
\[
\delta_{\sigma}\in C_{per}^{\infty}\left(I;C_{0}^{\infty}\left(\omega\right)\right),\quad\eta-\sigma<\delta_{\sigma}<\eta.
\]
We recall the Piola transform $\mathcal{J}_{\delta_{\sigma}}$ from Lemma~\ref{lm:Piola} and we claim that 
\begin{equation}\label{eqn:claim-piola}
\int_{I}\int_{\Omega_{F}^{\eta_{n}}\left(t\right)}\mathbf{u}_{n}\cdot\mathcal{J}_{\delta_{\sigma}}\phi dx\xrightarrow{n\to\infty}\int_{I}\int_{\Omega_{F}^{\eta}\left(t\right)}\mathbf{u}\cdot\mathcal{J}_{\delta_{\sigma}}\phi dx
\end{equation}
uniformly for all $\phi \in H^{1}_{\text{div}}(\Omega)$.
A few remarks are needed. 
First, since $\eta_n \to \eta$ in $C_{t,x}^{0}$, 

for sufficiently large $n>n(\sigma)$ we have that $\Omega_{F}^{\delta_{\sigma}}\subset\Omega_{F}^{\eta_{n}}$ and so the integrands in \eqref{eqn:claim-piola} are well-defined after extending $\mathcal{J}_{\delta_{\sigma}}\phi$ by zero.

Secondly, that the Piola mapping $\mathcal{J}_{\delta_{\sigma}}$  maps

\[
H:=\left\{ \mathbf{u}\in H^{1}\left(\Omega;\mathbb{R}^{3}\right):\text{div}\mathbf{u}=0,\ \mathbf{u}\times\mathbf{e}_{r}=\mathbf{0}\ \text{on}\ \Gamma,\ \mathbf{u}\times\mathbf{e}_{z}=\mathbf{0}\ \text{on}\ \Gamma_{in/out}\right\} 
\]
into
\[
\begin{aligned}H_{fl}^{\delta_{\sigma}}:= & \left\{ \mathbf{u}\in H^{1}\left(\Omega_{F}^{\delta_{\sigma}};\mathbb{R}^{3}\right):\text{div}\mathbf{u}=0,\ \mathbf{u}\times\mathbf{e}_{r}=0\ \text{on}\ \Gamma^{\delta_{\sigma}},\ \mathbf{u}\times\mathbf{e}_{z}=0\ \text{on}\ \Gamma_{in/out}\right\} \end{aligned}
.
\]

We recall that the Piola mapping $\mathcal{J}_{\delta_{\sigma}}$ maps $W_{x}^{1,2}$ into its correspondent with a continuity constant depending on $\left\Vert \nabla\delta_{\sigma}\right\Vert _{L^{\infty}_{x}}$, cf. Lemma~\ref{lm:Piola}.

The proof of \eqref{eqn:claim-piola} follows closely the steps already detailed in Subsection~\ref{ssec:1st-compactness} after defining the analogous functionals $c_{\phi,n}\left(t\right):=\int_{\Omega_{F}^{\eta_{n}}\left(t\right)}\mathbf{u}_{n}\cdot\mathcal{J}_{\delta_{\sigma}}\phi dx$ and its corresponding limit $c_{\phi}(t)$ for all $\phi \in H$ .

\item  Our next objective is to write $\mathbf{u}_{n}-\mathcal{F}_{\eta_{n}}\left(\partial_{t}\eta_{n}\right)\in L_{\text{div}}^{2}\left(\Omega_{\eta_{n}}\right)$ as $\mathbf{u}_{n}-\mathcal{F}_{\eta_{n}}\left(\partial_{t}\eta_{n}\right)=\mathcal{J}_{\delta_{\sigma}}\phi$ and use the convergence \eqref{eqn:claim-piola}. To this end we use the approximation Lemma~\ref{lemma:approx-psi}, which for any $\varepsilon>0$ there exists $\sigma_{\varepsilon}>0$ and a function 
$
\psi_{t,n}\in C_{\text{0,div}}^{\infty}\left(\Omega_{\eta_{n}}\right)$
such that 

\begin{equation}\label{eqn:prop-approx}
\text{supp} \ \psi_{t,n}\subset\Omega_{\eta_{n}-3\sigma_{\varepsilon}} \ \text{and} \ 
\left\Vert \mathbf{u}_{n}-\mathcal{F}_{\eta_{n}}\left(\partial_{t}\eta_{n}\right)-\psi_{t,n}\right\Vert _{\left(H^{1/4}\left(\mathbb{R}^{3}\right)\right)^{\prime}}<\varepsilon.
\end{equation}
Now, since $\eta_n \to \eta$ uniformly it follows that for  all $n>n(\varepsilon)$ holds the following
\[
\Omega_{\eta_{n}-3\sigma_{\varepsilon}}\subset\Omega_{\eta-2\sigma_{\varepsilon}}\subset\Omega_{\delta_{\sigma_{\varepsilon}}}
\]
so by extending $\psi_{t,n}$ by zero to $\Omega_{\delta_{\sigma_{\varepsilon}}}$ it follows that we can write 
\[
\psi_{t,n}=\mathcal{J}_{\delta_{\sigma_{\varepsilon}}}\phi,\ \phi\in H,
\]
and then from \eqref{eqn:claim-piola}, it follows that 
\begin{equation}
\lim_{n\to\infty}\int_{I}\int_{\Omega_{F}^{\eta_{n}}\left(t\right)}\mathbf{u}_{n}\cdot\psi_{t,n}dxdt=\lim_{n\to\infty}\int_{I}\int_{\Omega_{F}^{\eta}\left(t\right)}\mathbf{u}\cdot\psi_{t,n}dxdt.
\end{equation}
Using \eqref{eqn:prop-approx}, we obtain that
\begin{equation}
    \begin{aligned}\lim_{n\to\infty}\int_{I}\int_{\Omega_{F}^{\eta_{n}}\left(t\right)}\mathbf{u}_{n}\cdot\left(\mathbf{u}_{n}-\mathcal{F}_{\eta_{n}}\left(\partial_{t}\eta_{n}\right)\right)dxdt & =\lim_{n\to\infty}\int_{I}\int_{\Omega_{F}^{\eta}\left(t\right)}\mathbf{u}\cdot\left(\mathbf{u}_{n}-\mathcal{F}_{\eta_{n}}\left(\partial_{t}\eta_{n}\right)\right)dxdt\\
 & =\int_{I}\int_{\Omega_{F}^{\eta}\left(t\right)}\mathbf{u}\cdot\left(\mathbf{u}-\mathcal{F}_{\eta}\left(\partial_{t}\eta\right)\right)dxdt,
\end{aligned}
\end{equation}
where the last convergence only makes use of the weak-convergence. It follows that 
\eqref{eqn:compact-2} holds.

\end{enumerate}

\section{The main results}\label{sec:main-construction}
In this section we prove the existence of a weak solution, that is we prove Theorem~\ref{thm:main}.

The proof of Theorem~\ref{thm:main} will be presented in the rest of the section and it is divided into several steps, each one presented in a separate Subsection. It consists of a decoupling procedure, presented in Subsection~\ref{ssec:decoupled} and a fixed-point argument presented in Subsection~\ref{ssec:set-v-fix-pt}.
\subsection{Construction of the decoupled solution}\label{ssec:decoupled}
Suppose that for given 
\[\delta\in C_{per}^{1}\left(I;C_{0}^{\infty}\left(\omega;\mathbb{R}\right)\right),\quad \mathbf{v}\in L_{per}^{2}\left(I;L^{2}\left(\mathbb{R}^{3}\right)\right)
\]
we prescribe a fluid domain 
\[
\Omega_{F}^{\delta}\left(t\right):=\left\{ \left(r\cos\theta,r\sin\theta,z\right)\in\mathbb{R}^{3}:r<R+\delta,\ \left(\theta,z\right)\in\omega\right\} 
\]
 We aim to solve the FSI problem \eqref{eqn:weak-form}  in the \emph{given} fluid domain $\Omega_{F}^{\delta}\left(t\right)$. For this we introduce the following adapted \emph{decoupled} formulation. For any $\delta$ we will obtain a triplet $\left(\mathbf{u},\eta,\mathbf{d}\right)\in V_{soln}^{\delta}$ with $\text{tr}_{\Gamma^{\delta}}\mathbf{u}=\eta\mathbf{e}_{r}$. This triplet is not yet a solution to the original problem, since that would require that $\text{tr}_{\Gamma^{\eta}}\mathbf{u}=\eta\mathbf{e}_{r}=\text{tr}_{\Gamma}\mathbf{d}$. It is this coupling that is missing, and that will
 be achieved after a fixed-point procedure that will ensure $\delta=\eta$;
 see Subsection~\ref{ssec:set-v-fix-pt}. The approach follows the lines of the one from \cite[Section 3]{LR14}.

\begin{definition}
We call $\left(\mathbf{u},\eta,\mathbf{d}\right)\in V_{soln}^{\delta}$ a weak-solution for the \emph{decoupled} and \emph{linearized} around $(\delta,\mathbf{v})$ problem if 
\begin{equation}
\begin{aligned}&\int_{I}\int_{\Omega_{F}^{\delta}\left(t\right)}-\mathbf{u}\cdot\partial_{t}\mathbf{q}+\nabla\mathbf{u}:\nabla\mathbf{q}\ dx \ dt+\int_{I}b\left(t,\mathbf{u},\mathbf{v},\mathbf{q}\right) \ dt  +
\int_{I}\int_{\omega}-\frac{1}{2}\partial_{t}\eta\partial_{t}\delta\left(R+\delta\right)\xi-\partial_{t}\eta\partial_{t}\xi \ dA\ dt + \\ &+\int_{I}K\left(\eta,\xi\right)\ dt +
\int_{I}\int_{\Omega_{S}}-\partial_{t}\mathbf{d}\cdot\partial_{t}\bm{\xi}\ dx\ dt+\int_{I}a_{S}\left(\mathbf{d},\bm{\xi}\right)\ dt =
\int_{I}\left\langle F\left(t\right),\mathbf{q}\right\rangle\ dt ,
\end{aligned}
\end{equation}
holds for all $\left(\mathbf{q},\xi,\boldsymbol{\xi}\right)\in V_{test}^{\delta}$.

\end{definition}

\subsubsection{The Galerkin approximation}
Let $\left(Y_{k}\right)_{k\ge1}$ be a smooth basis of $H_{0}^{2}\left(\omega\right)$.
In order to extend it to a system of linear-independent vectors in $\Omega_S$ we first fix a cut-off function
\[
s\in C^{\infty}\left(R,R+H\right),\ s\left(R\right)=1,\ s\left(R+H\right)=0,
\]
and then the operator 
\[
\mathcal{F}_{S}:H_{0}^{2}\left(\omega;\mathbb{R}\right)\mapsto H^{2}\left(\Omega_{S}\right),\ \mathcal{F}_{S}\left(\xi\right)=\xi\left(s\left(r\right)\cos\theta,s\left(r\right)\sin\theta,z\right)\mathbf{e}_{r}\left(\theta\right).
\]
Obviously, $\mathcal{F}_{S}\left(\xi\right)\left(R,\cdot\right)=\xi\mathbf{e}_{r}$. We set
\[
\mathbf{Y}_{k}^{S}:=\mathcal{F}_{S}\left(Y_{k}\right),\ k\ge 1.
\]
For the fluid domain, employing the extension operator $\mathcal{F}_{\delta}$ (from Proposition~\ref{prop:estimates-extension-operator}), we define
\[
Y_{k}^{F}:=\mathcal{F}_{\delta}\left(Y_{k}\right),\ k\ge1.
\]

Next we consider $\left\{ \mathbf{\tilde{Z}}_{k}^{F}\right\} _{k\ge1}$, an $H^{1}$ basis of the space 
\[
\mathcal{H}:=\left\{ \mathbf{v}\in H_{\text{div}}^{1}\left(\Omega\right):\text{tr}_{\Gamma}\mathbf{v}=\mathbf{0},\ \text{tr}_{\Gamma_{in/out}}\mathbf{v}\times\mathbf{e}_{z}=0\right\} .
\]
This basis can be constructed by considering the eigenvalues of the  operator 
\[
\Lambda_{\xi}:L_{\text{div}}^{2}\left(\Omega\right)\mapsto L_{\text{div}}^{2}\left(\Omega\right),\quad\Lambda_{\xi}\left(\mathbf{f}\right):=\mathbf{v},
\]
where for each $\xi\in L^{2}\left(\Gamma_{in/out}\right)$ with $\int_{\Gamma_{in/out}}\xi \ dA=0$, we solve the Stokes problem
\[
\begin{cases}
-\Delta\mathbf{v}+\nabla\pi=\mathbf{f} & \text{in}\ \Omega,\\
\text{div}\ \mathbf{v}=0 & \text{in}\ \Omega,\\
\mathbf{v}=\pm\xi\mathbf{e}_{z} & \text{on}\ \Gamma_{in/out},\\
\mathbf{v}=\mathbf{0} & \text{on}\ \Gamma.
\end{cases}
\]
One can easily note that $\Lambda_{\xi}$ is symmetric and compact and thus its eigenvectors form a countable basis of $L_{\text{div}}^{2}\left(\Omega\right)$. By ranging $\xi$ in a countable dense $L^2$ basis of the space $\left\{ \xi:\xi\in L^{2}\left(\Gamma_{in/out}\right),\ \int\xi=0\right\}$, we obtain the basis $\left\{ \mathbf{\tilde{Z}}_{k}^{F}\right\} _{k\ge1}$  of the space $\left\{ \mathbf{v}\in H^{1}\left(\Omega\right):\text{tr}_{\Gamma_{in/out}}\mathbf{v}=\xi \mathbf{n}=\pm\xi\mathbf{e}_{z}\right\} $.\footnote{The argument is standard in the theory of Navier-Stokes equations, see for example \cite[Section 2.5]{Seregin-book}. }

 Using the Piola mapping, we can define
$\mathbf{Z}_{k}^{F}:=\mathcal{J}_{\delta}\mathbf{\tilde{Z}}_{k}^{F}$ 
and thus we obtain an $H^{1}$ basis of the space 
\[
\mathcal{H}_{\delta}:=\left\{ \mathbf{v}\in H_{\text{div}}^{1}\left(\Omega_{F}^{\delta}\left(t\right)\right):\text{tr}_{\Gamma^{\delta}}\mathbf{v}=\mathbf{0},\ \text{tr}_{\Gamma_{in/out}}\mathbf{v}\times\mathbf{e}_{z}=0\right\} . 
\]

Finally, let us consider the family $\left\{ \mathbf{Z}_{k}^{S}\right\} _{k\ge1}$ which is a basis for the space 
\[
\left\{ \mathbf{d}\in H^{1}\left(\Omega_{S}\right):\mathbf{d}=\mathbf{0}\ \text{\text{on}\ \ensuremath{\Gamma\cup\Gamma_{in/out}^{s}}}\right\}. 
\]

We may join all the above families into the formula
\begin{equation}\label{eqn:XK-defn}
\left(\mathbf{X}_{k}^{F},X_{k},\mathbf{X}_{k}^{S}\right):=\begin{cases}
\left(\mathbf{Y}_{k}^{F},Y_{k},\mathbf{Y}_{k}^{S}\right), & k\ \mbox{is }\text{odd}\\
\left(\mathbf{Z}_{k}^{F},0,\mathbf{Z}_{k}^{S}\right), & k\ \mbox{is } \text{even}.
\end{cases}
\end{equation}

\begin{remark}
The set $\mathcal{S}:=\text{span}\left\{ \left(\varphi\left(t\right)\mathbf{X}_{k}^{F},\varphi\left(t\right)X_{k},\varphi\left(t\right)\mathbf{X}_{k}^{S}\right):\varphi\in C_{per}^{1}\left(I\right),k\in\mathbb{N}_{\ge1}\right\} $ is dense in $V_{test}^{\delta}$.
Let $\left(\mathbf{q},\xi,\boldsymbol{\xi}\right)\in V_{test}^{\delta}$ and $\varepsilon>0$.
We can find $n_{\varepsilon}\in \mathbb{N}_{\ge1}$ and
$n_{\varepsilon}$ functions $e_{n_{\varepsilon}}^{k}\in C_{per}^{1}\left(I\right)$ with $1\le k \le n_{\varepsilon}$
for which 

$\left\Vert \xi-\sum_{k=1}^{n_{\varepsilon}}e_{n}^{k}X_{k}\right\Vert _{C^{1}\left(I;H^{2}\left(\omega\right)\right)}\le\varepsilon$. 
Thus we can consider 
\[
\mathbf{q}_{n}:=\sum_{k=1}^{n_{\varepsilon}}e_{n}^{k}\mathbf{X}_{k}^{F},\quad\boldsymbol{\xi}_{n}:=\sum_{k=1}^{n_{\varepsilon}}e_{n}^{k}\mathbf{X}_{k}^{S}
\]
and we note that $\left\Vert \boldsymbol{\xi}-\boldsymbol{\xi}_{n}\right\Vert _{C^{1}\left(I;H^{1}\left(\Omega_{S}\right)\right)}\lesssim_{\Omega_{S}}\varepsilon $. 
Regarding $\mathbf{q}-\mathbf{q}_{n}$ we note that $\mathbf{q}-\mathbf{q}_{n}-\mathcal{F_{\delta}}\left(\xi-\xi_{n}\right)\in\mathcal{H}_{\delta}$ and thus $\mathcal{J}_{\delta}^{-1}\left(\mathbf{q}-\mathbf{q}_{n}-\mathcal{F_{\delta}}\left(\xi-\xi_{n}\right)\right)\in\mathcal{H}$. But here we can find $m_{\varepsilon}$ functions $f_{m_{\varepsilon}}^{k}\in C_{per}^{1}\left(I\right)$ such that 
\[
\left\Vert \mathcal{J}_{\delta}^{-1}\left(\mathbf{q}-\mathbf{q}_{n}-\mathcal{F_{\delta}}\left(\xi-\xi_{n}\right)\right)-\sum_{k=1}^{m_{\varepsilon}}f_{m_{\varepsilon}}^{k}\tilde{\mathbf{Z}}_{k}^{F}\right\Vert _{C^{1}\left(I;H^{1}\left(\Omega\right)\right)}\le\varepsilon .
\]
and thus
\[
\left\Vert \mathbf{q}-\mathbf{q}_{n}-\mathcal{F_{\delta}}\left(\xi-\xi_{n}\right)-\sum_{k=1}^{m_{\varepsilon}}f_{m_{\varepsilon}}^{k}\mathbf{Z}_{k}^{F}\right\Vert _{C^{1}\left(I;H^{1}\left(\Omega_{F}^{\delta}\right)\right)}\lesssim\varepsilon .
\]

In the end we have that 
\begin{equation}
\begin{aligned}\left(\mathbf{q},\xi,\boldsymbol{\xi}\right)= & \underbrace{\left(\mathbf{q}_{n},\xi_{n},\boldsymbol{\xi}_{n}\right)}_{\in\mathcal{S}}+\underbrace{\left(\mathbf{q}-\mathbf{q}_{n}-\mathcal{F}_{\delta}\left(\xi-\xi_{n}\right),0,\boldsymbol{0}\right)}_{\in\mathcal{S}}+\\
 & \left(\mathcal{F}_{\delta}\left(\xi-\xi_{n}\right),\xi-\xi_{n},\boldsymbol{\xi}-\boldsymbol{\xi}_{n}\right)
\end{aligned}
\end{equation}
and the error $\left(\mathcal{F}_{\delta}\left(\xi-\xi_{n}\right),\xi-\xi_{n},\boldsymbol{\xi}-\boldsymbol{\xi}_{n}\right)$ will be small. 
\end{remark}

\paragraph{Discrete solutions.}
We can now make the ansatz  
\begin{equation}\label{eqn:ansatz-un}
\begin{aligned}\mathbf{u}_{n}\left(t,x\right):= & \sum_{k=1}^{n}\left(\mathbf{a}_{n}^{k}\right)^{\prime}\left(t\right)\mathbf{X}_{k}^{F},\quad\text{in}\ I\times\Omega_{F}^{\delta},\\
\eta_{n}\left(t,x\right):= & \sum_{k=1}^{n}\mathbf{a}_{n}^{k}\left(t\right)X_{k}\left(x\right),\quad\text{in}\ I\times\omega,\\
\mathbf{d}_{n}\left(t,x\right):= & \sum_{k=1}^{n}\mathbf{a}_{n}^{k}\left(t\right)\mathbf{X}_{k}^{S}\left(x\right),\quad\text{in}\ I\times\Omega_{S},
\end{aligned}
\end{equation}
where the unknown $\mathbf{a}_{n}\in C^{2}\left(I;\mathbb{R}^{n}\right)$ is obtained by solving the  system 
\begin{equation}
    \begin{aligned}\int_{\Omega_{F}^{\delta}\left(t\right)}-\mathbf{u}_{n}\cdot\partial_{t}\mathbf{X}_{k}^{F}+\nabla\mathbf{u}_{n}:\nabla\mathbf{X}_{k}^{F}\ dx+b\left(t,\mathbf{u}_{n},\mathbf{v},\mathbf{X}_{k}^{F}\right) +
\int_{\omega}-\frac{1}{2}\left(\partial_{t}\eta_{n}\right)^{2}X_{k}\left(R+\delta\right)+\partial_{tt}\eta_{n}X_{k}\ dA &\\ +K\left(\eta_{n},X_{k}\right)  +
\int_{\Omega_{S}}\partial_{tt}\mathbf{d}_{n}\cdot\mathbf{X}_{k}^{S}dx+a_{S}\left(\mathbf{d}_{n},\mathbf{X}_{k}^{S}\right) =
\left\langle F\left(t\right),\mathbf{X}_{k}^{F}\right\rangle \quad t\in I,k\in\overline{1,n} .
\end{aligned}
\end{equation}
In order to obtain meaningful energy estimates when testing with $\left(\mathbf{u}_{n},\partial_{t}\eta_{n},\partial_{t}\mathbf{d}_{n}\right)$ we need to rewrite the system to an equivalent form.
Thus, we obtain the equivalent system
\begin{equation}\label{eqn:Galerkin}
   \begin{aligned}
   &\frac{d}{dt}\int_{\Omega_{F}^{\delta}\left(t\right)}\frac{\mathbf{u}_{n}\cdot\mathbf{X}_{k}^{F}}{2}\ dx +\int_{\Omega_{F}^{\delta}\left(t\right)}\frac{\partial_{t}\mathbf{u}_{n}\cdot\mathbf{X}_{k}^{F}-\mathbf{u}_{n}\cdot\partial_{t}\mathbf{X}_{k}^{F}}{2}\ dx +
\int_{\Omega_{F}^{\delta}\left(t\right)}\nabla\mathbf{u}_{n}:\nabla\mathbf{X}_{k}^{F}\ dx+b\left(t,\mathbf{u}_{n},\mathbf{v},\mathbf{X}_{k}^{F}\right)\\&+
\int_{\omega}\partial_{tt}\eta_{n}X_{k}\ dA+K\left(\eta_{n},X_{k}\right) +
\int_{\Omega_{S}}\partial_{tt}\mathbf{d}_{n}\cdot\mathbf{X}_{k}^{S}\ dx+a_{S}\left(\mathbf{d}_{n},\mathbf{X}_{k}^{S}\right) =
\left\langle F\left(t\right),\mathbf{X}_{k}^{F}\right\rangle \quad\forall\ t\in I,\ k\in\overline{1,n} .
\end{aligned}
\end{equation}
Let us now discuss the solvability of \eqref{eqn:Galerkin}. The mass matrix corresponding to the 2nd order terms is given by 
\[
M_{ij}\left(t\right):=\int_{\Omega_{F}^{\delta}\left(t\right)}\mathbf{X}_{i}^{F}\cdot\mathbf{X}_{j}^{F}\ dx+\int_{\omega}X_{i}X_{j}\ dA+\int_{\Omega_{S}}\mathbf{X}_{i}^{S}\cdot\mathbf{X}_{j}^{S}\ dx\quad i,j=\overline{1,n},
\]
which is obviously positive definite. 
A standard Picard-Lindel\"{o}f argument can thus be employed, see \cite[Proposition 3.27]{LR14}. 
Therefore, if we set $\left(\mathbf{a}_{n}\left(0\right),\mathbf{a}_{n}^{\prime}\left(0\right)\right)=:\left(\mathbf{a}_{n0},\mathbf{a}_{n1}\right)\in\mathbb{R}^{2\times n}$
then $\mathbf{a}_n$ exists locally on some interval $(0,T*)$ which can be extended to $I=(0,T)$
due to usual Gronwall arguments.

Indeed, from \eqref{eqn:ansatz-un} we multiply \eqref{eqn:Galerkin} by $\left(\mathbf{a}_{n}^{k}\right)^{\prime}$, which corresponds to testing with  $\left(\mathbf{u}_{n},\eta_{n},\mathbf{d}_{n}\right)$, and we obtain the analogue of \eqref{eqn:energy-balance}, namely the discrete energy balance
\begin{equation}\label{eqn:En-energy-balance}
    \frac{d}{dt}E_{n}\left(t\right)+D_{n}\left(t\right)=P_{in/out}\left(t\right)\int_{\Gamma_{in/out}}\mathbf{u}\cdot\mathbf{n}\  dA.
\end{equation}

\subsubsection{Establishing the periodicity}\label{ssec:periodicity}
As a consequence of the above discussion, the following Poincar\'{e} mapping  is thus well defined:
\begin{equation}\label{eqn:poincare-map}
    P_{n}:\mathbb{R}^{2\times n}\mapsto\mathbb{R}^{2\times n},\quad P_{n}:\left(\mathbf{a}_{n}\left(0\right),\mathbf{a}_{n}^{\prime}\left(0\right)\right)\mapsto\left(\mathbf{a}_{n}\left(T\right),\mathbf{a}_{n}^{\prime}\left(T\right)\right).
\end{equation}
\begin{remark}\label{rmk:an-C2}
    It is useful to notice that for a fixed $n$ we have that $\mathbf{a}_{n}$ from \eqref{eqn:Galerkin} is bounded in $C^{2}(I;\mathbb{R}^{n})$, although not uniformly w.r.t. $n$.
\end{remark}
\begin{proposition}\label{prop:Pn-fixed-point}
    The mapping $P_n$ from \eqref{eqn:poincare-map} has a fixed point.
\end{proposition}
\begin{proof}
    We shall apply the Leray-Schauder fixed point theorem.

    To this end we check the
\begin{itemize}
    \item Compactness: Let $\left|\left(\mathbf{a}_{n_{k}}\left(0\right),\mathbf{a}_{n_{k}}^{\prime}\left(0\right)\right)\right|\le1$ for all $k \ge 1$. From Remark~\ref{rmk:an-C2} we have that there exists a subsequence of $\mathbf{a}_{n_{k}}$, which we don't relabel, and a function $\mathbf{a}_{n}\in C^{2}\left(I\right)$ for which $\mathbf{a}_{n_{k}}\to\mathbf{a}_{n}$ in $C^{1}(I)$. It is then clear than $\left(\mathbf{a}_{n_{k}}\left(T\right),\mathbf{a}_{n_{k}}^{\prime}\left(T\right)\right)_{k}$ is convergent.

    \item  Continuity: Let $\left(\mathbf{a}_{n_{k}}\left(0\right),\mathbf{a}_{n_{k}}^{\prime}\left(0\right)\right)\to\left(\mathbf{a}_{n0},\mathbf{a}_{n1}\right)$ as $k\to \infty$. With $\mathbf{a}_n$ from above it is clear that $\mathbf{a}_n$ solves \eqref{eqn:Galerkin} with initial data $\left(\mathbf{a}_{n0},\mathbf{a}_{n1}\right)$. Then $\left(\mathbf{a}_{n_{k}}\left(T\right),\mathbf{a}_{n_{k}}^{\prime}\left(T\right)\right)\to\left(\mathbf{a}_{n}\left(T\right),\mathbf{a}_{n}^{\prime}\left(T\right)\right)$ which proves the continuity.
    \item Boundedness: we need to check that the set 
\[
LS:=\left\{ \left(\mathbf{a}_{n0},\mathbf{a}_{n1}\right):\left(\mathbf{a}_{n0},\mathbf{a}_{n1}\right)=\lambda P_{n}\left(\mathbf{a}_{n0},\mathbf{a}_{n1}\right)\ \text{for some }\lambda\in\left[0,1\right]\right\} 
\]
is uniformly bounded. The crucial observation is that any initial data $\left(\mathbf{a}_{n0},\mathbf{a}_{n1}\right)\in LS$ produces a solution $\mathbf{a}_n$ of $\eqref{eqn:Galerkin}$ which enjoys the property 
\[
\left(\mathbf{a}_{n}\left(0\right),\mathbf{a}_{n}^{\prime}\left(0\right)\right)=\lambda\left(\mathbf{a}_{n}\left(T\right),\mathbf{a}_{n}^{\prime}\left(T\right)\right),\quad\lambda\in\left[0,1\right]
\]
and thus
\begin{equation}\label{eqn:lambda-energy}
E_{n}\left(0\right)=\lambda^{2}E_{n}\left(T\right)\le E_{n}\left(T\right).
\end{equation}
Integrating \eqref{eqn:En-energy-balance} on $I$ and using \eqref{eqn:lambda-energy} we find out that 
\begin{equation}\label{eqn:diff-Dn}
\int_{I}D_{n}\left(t\right)dt:=\int_{I}\int_{\Omega_{F}^{\delta}\left(t\right)}\left|\nabla\mathbf{u}_{n}\right|^{2}dxdt+\int_{I}\int_{\Omega_{S}}\left|\partial_{t}\nabla\mathbf{d}\right|^{2}dxdt\lesssim\left\Vert P_{in/out}\right\Vert _{L_{t}^{2}}^{2}
\end{equation}
Now noticing that \eqref{eqn:diff-Dn} is the analogue of \eqref{eqn:diffusion-estimate} the reasoning follows in exactly the same manner as in Subsection~\ref{ssec:unif-energ-est}. Important is  that from \eqref{eqn:ansatz-un} it is possible to test \eqref{eqn:Galerkin} with $\left(\mathcal{F}_{\delta}\left(\eta_{n}\right),\eta_{n},\mathcal{F}_{S}\left(\eta_{n}\right)\right)$ at the discrete level. Thus we obtain the same energy bound as \eqref{eqn:formal-en-bound}, namely: for $\left\Vert P_{in/out}\right\Vert _{L_{t}^{2}}\le C_{0}=C_{0}\left(\texttt{data}\right)$ it follows that 
\begin{equation}\label{eqn:uniform-En-estimate}
    \sup_{t\in I}E_{n}\left(t\right)+\int_{I}D_{n}\left(t\right)dt\lesssim\left\Vert P_{in/out}\right\Vert _{L_{t}^{2}}^{2}\lesssim C_{0}^{2}.
\end{equation}
\end{itemize}
This proves Proposition~\ref{prop:Pn-fixed-point}.
    \end{proof}
From \eqref{eqn:uniform-En-estimate},
we may now obtain a weak limit as well as convergence up to a subsequence (not relabeled) 
\[
\left(\mathbf{u}_{n},\eta_{n},\mathbf{d}_{n}\right)\rightharpoonup\left(\mathbf{u},\eta,\mathbf{d}\right)\quad\text{in}\ V_{soln}^{\delta}.
\]
that is 
\begin{equation}
\begin{aligned}\eta_{n} & \rightharpoonup\eta & \qquad & \text{weakly in }L_{}^{\infty}\left(I;H_{0}^{2}\left(\omega\right)\right)\cap W^{1,\infty}\left(I;L^{2}\left(\omega\right)\right),\\[0.3em]
\mathbf{d}_{n} & \rightharpoonup\mathbf{d} & \qquad & \text{weakly in }W^{1,\infty}\left(I;L^{2}\left(\Omega_{S}\right)\right)\cap W^{1,2}\left(I;H^{1}\left(\Omega_{S}\right)\right),\\[0.3em]
\mathbf{u}_{n} & \rightharpoonup\mathbf{u} & \qquad & \text{weakly in }L^{\infty}\left(I;L^{2}\left(\Omega_{F}^{\delta}\right)\right)\cap L^{2}\left(I;H^{1}\left(\Omega_{F}^{\delta}\right)\right).
\end{aligned}
\end{equation}
Since the problem \eqref{eqn:Galerkin} is linearized we obtain by passing to the limit that for any sufficiently regular $\delta, \mathbf{v}$ there exists a time-periodic weak solution. More precisely, by employing a regularizing operator $\mathcal{R}_{\varepsilon}$ we obtain the following result
\begin{proposition}\label{prop:existence-dec-reg-per}
For any $\varepsilon>0$, there exists at least one time-periodic weak solution 
$\left(\mathbf{u},\eta,\mathbf{d}\right)=\left(\mathbf{u}_{\varepsilon},\eta_{\varepsilon},\mathbf{d}_{\varepsilon}\right)\in V_{soln}^{\mathcal{R}_{\varepsilon}\delta}$ of the  following decoupled and regularized problem:
\begin{equation}\label{eqn:dec-reg-eps}
\begin{aligned}&\int_{I}\int_{\Omega_{F}^{\mathcal{R}_{\varepsilon}\delta}\left(t\right)}-\mathbf{u}\cdot\partial_{t}\mathbf{q}+\nabla\mathbf{u}:\nabla\mathbf{q}dxdt+\int_{I}b\left(t,\mathbf{u},\mathcal{R}_{\varepsilon}\mathbf{v},\mathbf{q}\right)dt  +
\int_{I}\int_{\omega}-\frac{1}{2}\partial_{t}\eta\partial_{t}\mathcal{R}_{\varepsilon}\delta\left(R+\mathcal{R}_{\varepsilon}\delta\right)\xi-\partial_{t}\eta\partial_{t}\xi \ dA\ dt \\& +\int_{I}K\left(\eta,\xi\right)dt  +
\int_{I}\int_{\Omega_{S}}-\partial_{t}\mathbf{d}\cdot\partial_{t}\bm{\xi}dxdt+\int_{I}a_{S}\left(\mathbf{d},\bm{\xi}\right)dt  =
\int_{I}\left\langle F\left(t\right),\mathbf{q}\right\rangle dt\quad\forall\left(\mathbf{q},\xi,\boldsymbol{\xi}\right)\in V_{test}^{\mathcal{R}_{\varepsilon}\delta}.
\end{aligned}
\end{equation}
Furthermore, there exists a constant $C_0$ for which
\begin{equation}\label{eqn:dec-reg-estimate-eps}
\left\Vert P_{in/out}\right\Vert _{L_{t}^{2}}\le C_{0}\implies\sup_{t\in I}E_{\varepsilon}\left(t\right)+\int_{I}D_{\varepsilon}\left(t\right)\lesssim\left\Vert P_{in/out}\right\Vert _{L_{t}^{2}}^{2},
\end{equation}
where $E_{\varepsilon}$ is the analogue of $E$ from \eqref{eqn:energy-generic}.
    
\end{proposition}
\subsection{The set-valued fixed-point argument}\label{ssec:set-v-fix-pt}
Let $\varepsilon>0$ be fixed and $\left(\delta,\mathbf{v}\right):=\left(\mathcal{R}_{\varepsilon}\delta,\mathcal{R}_{\varepsilon}\mathbf{v}\right)$.
In this section we aim to find a fixed-point of the following set-valued mapping 
\[
F:\left(\delta,\mathbf{v}\right)\mapsto\left\{ \left(\eta,\mathbf{u}\right):\left(\eta,\mathbf{u}, \mathbf{d}\right)\ \text{is a solution of } \eqref{eqn:dec-reg-eps} \ \text{and fulfills} \ \eqref{eqn:dec-reg-estimate-eps}\right\} .
\]
To this end, we shall employ the Kakutani-Glicksberg-Fan fixed-point Theorem~\ref{thm:kakutani}. We consider the space 
\[
Z:=C_{per}\left(\overline{I};C\left(\overline{\omega}\right)\right)\times L_{per}^{2}\left(I;L^{2}\left(\mathbb{R}^{3}\right)\right),
\]
and its convex subset 
\[
D:=\left\{ \left(\delta,\mathbf{v}\right)\in Z:\left\Vert \delta\right\Vert _{L_{t,x}^{\infty}}\le C_1 <R,\ \left\Vert \mathbf{v}\right\Vert _{L_{t,x}^{2}}\le C_{2}\right\} , 
\]
where $C_1$ can be chosen arbitrary in $(0,R)$ while $C_2$
can be chosen as the one which ensures $\left\Vert \mathbf{u}\right\Vert _{L_{t,x}^{2}}\le C_{2}$ from \eqref{eqn:dec-reg-estimate-eps}.
From \eqref{eqn:dec-reg-estimate-eps} we see that $F:D\mapsto \mathcal{P}(D)$ \emph{as long as $\left\Vert P_{in/out}\right\Vert _{L_{t}^{2}}$ is sufficiently small.} 
Let us now check the assumptions of Theorem~\ref{thm:kakutani}.
\begin{itemize}
    \item the upper-semicontinuity: is equivalent to the closed graph property. This means we need to prove that if $\left(\delta_{n},\mathbf{v}_{n}\right)\to\left(\delta,\mathbf{v}\right)$ in $Z$ and $\left(\eta_{n},\mathbf{u}_{n}\right)\in F\left(\delta_{n},\mathbf{v}_{n}\right)$ with $\left(\eta_{n},\mathbf{u}_{n}\right)\to\left(\eta,\mathbf{u}\right)$ then $\left(\eta,\mathbf{u}\right)\in F\left(\delta,\mathbf{v}\right)$. Indeed, it is not difficult to observe that $\left(\eta,\mathbf{u}, \mathbf{d}\right)$ fulfills the variational formulation \eqref{eqn:dec-reg-eps} corresponding to $\left(\delta,\mathbf{v}\right)$ and also \eqref{eqn:dec-reg-estimate-eps}.
    The coupling condition is also preserved by a limit passage and using the continuity of the trace operator. 
Let us now explain how to pass to the limit in the weak formulation 
\begin{equation}\label{eqn:n-weak-form-dec}
\begin{aligned}&\int_{I}\int_{\Omega_{F}^{\mathcal{R}_{\varepsilon}\delta_{n}}\left(t\right)}\mathbf{u}_{n}\cdot\partial_{t}\mathbf{q}_{n}+\nabla\mathbf{u}_{n}:\nabla\mathbf{q}_{n}\ dx\ dt+\int_{I}b\left(t,\mathbf{u}_{n},\mathcal{R}_{\varepsilon}\mathbf{v}_{n},\mathbf{q}_{n}\right)\ dt+ \\&
\int_{I}\int_{\omega}-\frac{1}{2}\partial_{t}\eta_{n}\partial_{t}\mathcal{R}_{\varepsilon}\delta_{n}\left(R+\mathcal{R}_{\varepsilon}\delta_{n}\right)\xi-\partial_{t}\eta_{n}\partial_{t}\xi\ dA\ dt+\int_{I}K\left(\eta_{n},\xi\right)\ dt +\\&
\int_{I}\int_{\Omega_{S}}-\partial_{t}\mathbf{d}_{n}\cdot\partial_{t}\bm{\xi}dxdt+\int_{I}a_{S}\left(\mathbf{d}_{n},\bm{\xi}\right)\ dt  =
\int_{I}\left\langle F\left(t\right),\mathbf{q}\right\rangle dt\quad\forall\left(\mathbf{q},\xi,\boldsymbol{\xi}\right)\in V_{test}^{\mathcal{R}_{\varepsilon}\delta_{n}}.
\end{aligned}
\end{equation}
The only problem is to obtain a weak-formulation valid against all $\left(\mathbf{q},\xi,\boldsymbol{\xi}\right)\in V_{test}^{\mathcal{R}_{\varepsilon}\delta}$. To this end, for $\left(\mathbf{q},\xi,\boldsymbol{\xi}\right)\in V_{test}^{\mathcal{R}_{\varepsilon}\delta}$ we have that

\[
\left(\mathcal{F}_{\mathcal{R}_{\varepsilon}\delta_{n}}\xi,\xi,\boldsymbol{\xi}\right),\left(\mathbf{\mathbf{q}-}\mathcal{F}_{\mathcal{R}_{\varepsilon}\delta_{n}}\xi,0,\boldsymbol{0}\right)\in V_{test}^{\mathcal{R}_{\varepsilon}\delta_{n}}
\]
where the second inclusion is justified by the fact that $\text{tr}_{\Gamma^{\mathcal{R}_{\varepsilon}\delta}\left(t\right)}\mathbf{\left(\mathbf{\mathbf{q}}-\mathcal{F}_{\mathcal{R}_{\varepsilon}\delta}\xi\right)=\mathbf{0}}$ and hence can be extended by zero to   $\Omega_{F}^{\mathcal{R}_{\varepsilon}\delta+\sigma}\left(t\right)$ for a small $\sigma>0$ and then can be restricted to $\Omega_{F}^{\mathcal{R}_{\varepsilon}\delta_{n}}\left(t\right)$ 
since $\mathcal{R}_{\varepsilon}\delta_{n}\to\mathcal{R}_{\varepsilon}\delta$. Thus $\left(\eta,\mathbf{u}\right)\in F\left(\delta,\mathbf{v}\right)$.
\item  For any $\left(\delta,\mathbf{v}\right)\in D$, we can easily see that $F\left(\delta,\mathbf{v}\right)$ is non-empty, cf. Proposition~\ref{prop:existence-dec-reg-per}, and  convex, since the problem is  linearized. To prove that $F\left(\delta,\mathbf{v}\right)$ and  $F(D)$ are relatively compact in $D$, we employ the same argument as in Section~\ref{sec:compactness}. 

Therefore, we can apply the fixed point Theorem~\ref{thm:kakutani}. With this, we update the Proposition~\ref{prop:existence-dec-reg-per} to the following
\begin{proposition}\label{prop:existence-eps-reg}
For any $\varepsilon>0$, there exists at least one time-periodic weak solution 
$\left(\mathbf{u},\eta,\mathbf{d}\right)=\left(\mathbf{u}_{\varepsilon},\eta_{\varepsilon},\mathbf{d}_{\varepsilon}\right)\in V_{soln}^{\mathcal{R}_{\varepsilon}\eta_{\varepsilon}}$ of the  following decoupled and regularized problem:
\begin{equation}\label{eqn:reg-eps}
\begin{aligned}
&\int_{I}\int_{\Omega_{F}^{\mathcal{R}_{\varepsilon}\eta}\left(t\right)}-\mathbf{u}\cdot\partial_{t}\mathbf{q}+\nabla\mathbf{u}:\nabla\mathbf{q}\ dx\ dt+\int_{I}b\left(t,\mathbf{u},\mathcal{R}_{\varepsilon}\mathbf{u},\mathbf{q}\right)\ dt +
\int_{I}\int_{\Omega_{S}}-\partial_{t}\mathbf{d}\cdot\partial_{t}\bm{\xi}\ dx\ dt+\int_{I}a_{S}\left(\mathbf{d},\bm{\xi}\right)\ dt \\ & +
\int_{I}\int_{\omega}-\frac{1}{2}\partial_{t}\eta\partial_{t}\mathcal{R}_{\varepsilon}\eta\left(R+\mathcal{R}_{\varepsilon}\eta\right)\xi-\partial_{t}\eta\partial_{t}\xi\ dA\ dt +\int_{I}K\left(\eta,\xi\right)dt   =
\int_{I}\left\langle F\left(t\right),\mathbf{q}\right\rangle dt\quad\forall\left(\mathbf{q},\xi,\boldsymbol{\xi}\right)\in V_{test}^{\mathcal{R}_{\varepsilon}\eta}.
\end{aligned}
\end{equation}

Furthermore, there exists a constant $C_1$ for which
\begin{equation}\label{eqn:eps-reg-soln}
\left\Vert P_{in/out}\right\Vert _{L_{t}^{2}}\le C_{1}\implies\sup_{t\in I}E_{\varepsilon}\left(t\right)+\int_{I}D_{\varepsilon}\left(t\right)\lesssim\left\Vert P_{in/out}\right\Vert _{L_{t}^{2}}^{2},
\end{equation}
where $E_{\varepsilon}$ is the analogue of $E$ from \eqref{eqn:energy-generic}.
    
\end{proposition}
 
\end{itemize}

\subsection{Proof of the main results}
We aim to let $\varepsilon \to 0$ in \eqref{eqn:reg-eps}.
From \eqref{eqn:eps-reg-soln} there exists a weak limit $\left(\mathbf{u},\eta,\mathbf{d}\right)\in V_{soln}$ for which up to a subsequence we have that $\left(\mathbf{u}_{\varepsilon},\eta_{\varepsilon},\mathbf{d}_{\varepsilon}\right)\rightharpoonup\left(\mathbf{u},\eta,\mathbf{d}\right)$. It can be easily seen that the main difficulty in the limit passage in \eqref{eqn:reg-eps}
is the strong convergence (up to a subsequence) of $\mathbf{u}_{\varepsilon}\to\mathbf{u}$ in $L^{2}_{t,x}$. In the sequel we present the details.
 
\begin{proof}[Proof of Theorem~\ref{thm:main}] 
We recall Proposition~\ref{prop:existence-eps-reg}.

For every $\varepsilon>0$  there exists at least one solution ${\left(\mathbf{u}_{\varepsilon},\eta_{\varepsilon},\mathbf{d}_{\varepsilon}\right)\in V_{soln}^{\mathcal{R}_{\varepsilon}\eta_{\varepsilon}}}$.
Taking $\varepsilon_{n}=\frac{1}{n}$ for every $n\ge 1$ sufficiently large, 
from \eqref{eqn:eps-reg-soln} we can extract a weakly convergent subsequence (not relabeled) $\left(\mathbf{u}_{n},\eta_{n},\mathbf{d}_{n}\right)\in V_{soln}^{\mathcal{R}_{\varepsilon_{n}}\eta_{n}}$ 
for which 
$\left(\mathbf{u}_{n},\eta_{n},\mathbf{d}_{n}\right)\rightharpoonup\left(\mathbf{u},\eta,\mathbf{d}\right)$ in the following sense:

\begin{equation}
\begin{aligned}\eta_{n} & \rightharpoonup\eta & \qquad & \text{weakly in }L_{t}^{\infty}H_{x}^{2}\cap W_{t}^{1,\infty}L_{x}^{2},\\[0.3em]
\mathbf{d}_{n} & \rightharpoonup\mathbf{d} & \qquad & \text{weakly in }W_{t}^{1,\infty}L_{x}^{2}\cap W_{t}^{1,2}H_{x}^{1},\\[0.3em]
\mathbf{u}_{n} & \rightharpoonup\mathbf{u} & \qquad & \text{weakly in }L_{t}^{\infty}L_{x}^{2}\cap L_{t}^{2}H_{x}^{1}.
\end{aligned}
\end{equation}

Note that we have the convergence of the domains $\Omega_{F}^{\mathcal{R}_{\varepsilon}\left(\eta_{\varepsilon}\right)}\to\Omega_{F}^{\eta}$ since $\mathcal{R}_{\varepsilon}\left(\eta_{\varepsilon}\right)\to\eta $ in $L^{\infty}\left(I\times\omega\right)$. This is due to the basic inequality
\[
\left\Vert \mathcal{R}_{\varepsilon}\left(\eta_{\varepsilon}\right)-\eta\right\Vert _{L^{\infty}\left(I\times\omega\right)}\le\left\Vert \mathcal{R}_{\varepsilon}\left(\eta_{\varepsilon}\right)-\mathcal{R}_{\varepsilon}\eta\right\Vert _{L^{\infty}\left(I\times\omega\right)}+\left\Vert \mathcal{R}_{\varepsilon}\eta-\eta\right\Vert _{L^{\infty}\left(I\times\omega\right)}
\]
and the fact that the operator $\mathcal{R}_{\varepsilon}:L^{\infty}\left(I;L^{\infty}\left(\omega\right)\right)\mapsto L^{\infty}\left(I;L^{\infty}\left(\omega\right)\right)$ is continuous, see Lemma~\ref{lm:regularization}.
Regarding the convergence 
\[
\partial_{t}\mathcal{R}_{\varepsilon_{n}}\eta_{n}\to\partial_{t}\eta\quad\text{in }L^{2}\left(I;L^{2}\left(\omega\right)\right)
\]
we use the convergence \eqref{eqn:strong-conv-partial-t-eta-n} combined with the properties of $\mathcal{R}_{\varepsilon_n}$ from Lemma~\ref{lm:regularization}.

Next, we can repeat the arguments from Proposition~\ref{prop:compactness} to ensure that 
\begin{equation}\label{eqn:reg-comp-1}
\begin{aligned}\int_{I}\int_{\Omega_{F}^{\mathcal{R}_{\varepsilon_{n}}\eta_{n}}\left(t\right)}\mathbf{u}_{n}\cdot\mathcal{F}_{\mathcal{R}_{\varepsilon_{n}}\eta_{n}}\left(\partial_{t}\eta_{n}\right)dxdt+\int_{I}\int_{\omega}\left(\partial_{t}\eta_{n}\right)^{2}dAdt & +\int_{I}\int_{\Omega_{S}}\left(\partial_{t}\mathbf{d}_{n}\right)^{2}dxdt\to\\
\int_{I}\int_{\Omega_{F}^{\eta}\left(t\right)}\mathbf{u}_{n}\cdot\mathcal{F}_{\eta}\left(\partial_{t}\eta\right)dxdt+\int_{I}\int_{\omega}\left(\partial_{t}\eta\right)^{2}dAdt & +\int_{I}\int_{\Omega_{S}}\left(\partial_{t}\mathbf{d}\right)^{2}dxdt
\end{aligned}
\end{equation}
and 
\begin{equation}\label{eqn:reg-comp-2}
\begin{aligned}\int_{I}\int_{\Omega_{F}^{\mathcal{R}_{\varepsilon_{n}}\eta_{n}}\left(t\right)}\mathbf{u}_{n}\cdot\left(\mathbf{u}_{n}-\mathcal{F}_{\eta_{n}}\left(\partial_{t}\eta_{n}\right)dxdt\right)\to\int_{I}\int_{\Omega_{F}^{\eta}\left(t\right)}\mathbf{u}\cdot\left(\mathbf{u}-\mathcal{F}_{\eta}\left(\partial_{t}\eta\right)\right)dxdt\end{aligned}.
\end{equation}

Adding \eqref{eqn:reg-comp-1} and \eqref{eqn:reg-comp-2} we obtain the norm convergence in $L^{2}_{t,x}$, which correlated with the weak convergence already obtained, ensures strong convergence. That is 
\begin{equation}
\mathbf{u}_{n}\chi_{\Omega_{F}^{\mathcal{R}_{\varepsilon_{n}}\eta_{n}}}\to\mathbf{u}\chi_{\Omega_{F}^{\eta}}\quad\text{strongly in }L^{2}\left(I;L^{2}\left(\mathbb{R}^{3}\right)\right)
\end{equation}
and 
\begin{equation}\label{eqn:strong-conv-partial-t-eta-n}
    \left(\partial_{t}\eta_{n},\ \partial_{t}\mathbf{d}_{n}\right)\to\left(\partial_{t}\eta,\ \partial_{t}\mathbf{d}\right)\quad\text{strongly in }L^{2}\left(I;L^{2}\left(\omega\right)\right)\times L^{2}\left(I;L^{2}\left(\Omega_{S}\right)\right)
\end{equation}

We are now ready to pass to the limit in \eqref{eqn:dec-reg-eps}. 
The argument is very similar to the one from \eqref{eqn:n-weak-form-dec}. See also \cite[Section 3.4]{LR14} for further details.

The weak limit is thus a solution for the FSI-problem \eqref{eqn:FSI}.
\end{proof}

\begin{proof}[Proof of Corollary~\ref{thm-corollary}]
    It is a simplified proof of Theorem~\ref{thm:main}, since the initial data $\left(\mathbf{u}_{0},\eta_{0},\eta_{1},\mathbf{d}_{0},\mathbf{d}_{1}\right)$ is given.  Thus, the arguments from Subsection~\ref{ssec:periodicity} can be skipped.
    
    In order to obtain a set-valued fixed-point result as in Subsection~\ref{ssec:set-v-fix-pt} we need to restrict the time-interval to $\left(0,T_{\max}\right)$ with $T_{\max}$ depending on the initial data. The rest of the proof follows then exactly as  above.
\end{proof}

 \section{Appendix}\label{sec:appendix}

\subsection{Trace operator}
\begin{lemma}\label{lm:trace}
   Let $\eta\in H^{2}\left(\omega\right)$. For all $1<p\le\infty $ and all $s\in (1,p)$
    the trace operator defined as 
    \[
\text{tr}_{\Gamma^{\eta}}:W^{1,p}\left(\Omega_{F}^{\eta}\right)\mapsto W^{1-\frac{1}{s},s}\left(\omega\right)\] by the formula \[ \text{tr}_{\Gamma^{\eta}}\left(\mathbf{v}\right):=\mathbf{v}\left(\left(R+\eta\right)\cos\theta,\left(R+\eta\right)\sin\theta,z\right),\ \left(\theta,z\right)\in\omega 
    \]
    is linear, continuous, with a continuity constant depending on $\omega$, $s$ and a bound for $\left\Vert \eta\right\Vert _{H^{2}\left(\omega\right)}$. In particular, the continuity constant is independent of time since $\left\Vert \eta\right\Vert _{H^{2}\left(\omega\right)}$ is always controlled by the energy.
\end{lemma}
\begin{proof}
    See \cite[Lemma 2.6]{LR14}.
\end{proof}

\subsection{Korn's identity}
The following Korn identity is proved in \cite[Lemma A.5.]{LR14}
\begin{lemma}\label{lm:korn} The equality 
\[
\int_{\Omega_{F}^{\eta}}\nabla\mathbf{u}:\nabla\mathbf{q}\ dx=2\int_{\Omega_{F}^{\eta}}\mathbb{D}\mathbf{u}:\mathbb{D}\mathbf{q}\ dx\]
holds for all smooth functions $\mathbf{u}$ and $\mathbf{q}$ such that $\text{tr}_{\Gamma^{\eta}}\mathbf{q}=b\mathbf{e}_{r}$ for some smooth function $b:\omega\mapsto\mathbb{R}$.
\end{lemma}
\begin{remark}
    The validity of Lemma~\ref{lm:korn} is strongly based on the \emph{no-slip} condition $\text{tr}_{\Gamma^{\eta}}\mathbf{u}=\partial_{t}\eta\ \mathbf{e}_{r}$.
\end{remark}
\subsection{The Piola mapping}
\begin{lemma}\label{lm:Piola}
Let $\eta\in C^{2}\left(\omega\right)$ with $\left\Vert \eta\right\Vert _{H_{x}^{2}}<M$ and $M$ chosen small enough so that $\left\Vert \eta\right\Vert _{L_{x}^{\infty}}<R$. For each function $\phi:\Omega\mapsto\mathbb{R}^{3}$ we define the \emph{Piola transform} of $\phi$ by the formula.
\[
\mathcal{J}_{\eta}\phi:\Omega_{F}^{\eta}\mapsto\mathbb{R}^{3},\quad\mathcal{J}_{\eta}\phi:=\left[\frac{\nabla\psi_{\eta}}{\det\nabla\psi_{\eta}}\phi\right]\circ\psi_{\eta}^{-1}.
\]
The Piola transform $\mathcal{J}_{\eta}$ is a one-to-one mapping from the set of functions defined on $\Omega$ with values in $\mathbb{R}^{3}$  to the set of functions defined on $\Omega^{\eta}_{F}$ with values in $\mathbb{R}^{3}$. It preserves the divergence-free property and the zero boundary values. Furthermore, for each $1\le p \le \infty$ and $ 1\le r<p$  it follows that 
\[
\mathcal{J}_{\eta}\phi:W^{1,p}\left(\Omega\right)\mapsto W^{1,r}\left(\Omega_{F}^{\eta}\right)
\]
is continuous with a continuity constant depending on $\Omega$, $p$, $r$, $M$. In case $r=p$ the continuity constant depends additionally on $\left\Vert \nabla \eta\right\Vert _{L_{x}^{\infty}}$. Its inverse $\mathcal{J}_{\eta}^{-1}$ enjoys the same properties.
\end{lemma}
\begin{proof}
    It follows by elementary computations. See also \cite[p. 211-212]{LR14}.
\end{proof}

\subsection{The Bogovskii operator}
The following result is well-known, see e.g. \cite{Bog1}, \cite{Bog2}, \cite{galdi-book} 
\begin{theorem}\label{thm:bogovskii}
Let $\Omega \subset \mathbb{R}^{n}$ be a bounded and Lipschitz domain, let $m\ge -1$ and $q\in (1,\infty)$. Then there exists an operator denoted 

    \[\text{Bog}:\left\{ f\in W^{m,q}\left(\Omega\right):\int_{\Omega}fdx=0\right\} \mapsto W_{0}^{m+1,q}\left(\Omega;\mathbb{R}^{n}\right).
\]
which is linear, continuous and has the property 
\[ \text{div}\ \text{Bog}\left(f\right)=f
\]
for all $f\in W^{m,q}\left(\Omega\right)$ with $\int_{\Omega}f=0$.

The continuity constant depends only on $\Omega,\ n,\ m,q$.
In particular
\[\text{Bog}:\left\{ f\in C^{\infty}\left(\Omega\right):\int_{\Omega}fdx=0\right\} \mapsto C_{0}^{\infty}\left(\Omega;\mathbb{R}^{n}\right).
\]

\end{theorem}
\subsection{An approximation lemma}
Given an open set $O$, we define the space $L^{2}_{\text{div}}(O)$ as follows:
\[
L_{\text{div}}^{2}\left(O\right):=\left\{ \phi\in C^{\infty}\left(O\right):\ \phi\ \text{has compact support},\ \ \phi\ \text{div}\phi=0\right\} ^{\left\Vert \cdot\right\Vert _{2}}.
\]
Then we recall  \cite[Lemma A.13]{LR14}, which can be adapted to our case
as follows
\begin{lemma}\label{lemma:approx-psi}
For all $N>0$, $s>0$, $\varepsilon>0$ there exists $\sigma>0$ such that: for all 
\[
\eta\in H_{0}^{2}\left(\omega\right),\quad\left\Vert \eta\right\Vert _{H_{x}^{2}}<N
\]
and all $\varphi\in L_{\text{div}}^{2}\left(\Omega_{F}^{\eta}\right)$ with $\left\Vert \varphi\right\Vert _{L_{x}^{2}}\le1$, there exists 
$\psi\in L_{\text{div}}^{2}\left(\Omega_{F}^{\eta}\right)$ with $\left\Vert \psi\right\Vert _{L_{x}^{2}}\le 2$ such that, after extension by zero to $\mathbb{R}^{3}$ it holds that
\[
\left\Vert \varphi-\psi\right\Vert _{\left(H^{s}\left(\mathbb{R}^{3}\right)\right)^{\prime}}<\varepsilon,\quad\text{supp}\ \psi\subset\Omega_{F}^{\eta-\sigma}
\]
\end{lemma}

\subsection{A regularizing operator}
    {
\begin{lemma}\label{lm:regularization}
For $k\in\mathbb N_0$, let
\[
C_0^k(\omega)
:=
\left\{
\delta\in C^k(\overline\omega):
\begin{array}{l}
\delta(\theta+2\pi,z)=\delta(\theta,z),\ 
\left.\partial_z^\ell\delta\right|_{z=0,L}=0,
\quad \ell=0,\ldots,k
\end{array}
\right\},
\]
and set
\[
C_0^\infty(\omega):=\bigcap_{k=0}^\infty C_0^k(\omega).
\]
For every $0<\varepsilon<L/4$, there exists a linear operator
\[
\mathcal R_\varepsilon:
C_{\mathrm{per}}^0\bigl(I;C_0^0(\omega)\bigr)
\longrightarrow
C_{\mathrm{per}}^\infty
\bigl(I;C_0^\infty(\omega)\bigr)
\]
with the following properties:

\begin{enumerate}[(i)]
\item
The function $\mathcal R_\varepsilon\delta$ is $T$-periodic in time
and is identically zero in a neighbourhood of $z=0,L$. In particular,
\[
\left.\mathcal R_\varepsilon\delta\right|_{z=0,L}
=
\left.\partial_z\mathcal R_\varepsilon\delta\right|_{z=0,L}
=0.
\]

\item
If
\[
\delta\in C^{0,\gamma}
\bigl(\overline I\times\overline\omega\bigr),
\qquad \gamma\in(0,1),
\]
then
\[
\mathcal R_\varepsilon\delta
\longrightarrow\delta
\quad\text{in }L^\infty(I\times\omega).
\]

\item
Let $1\leq p<\infty$ and $1<q<\infty$. If
\[
\delta\in L^p\bigl(I;W^{1,q}(\omega)\bigr)
\]
is periodic in $\theta$ and has zero trace at $z=0,L$, then
\[
\mathcal R_\varepsilon\delta
\longrightarrow\delta
\quad\text{in }L^p\bigl(I;W^{1,q}(\omega)\bigr).
\]

\item
If
\[
\partial_t\delta\in L^p(I\times\omega),
\qquad 1\leq p<\infty,
\]
then
\[
\partial_t\mathcal R_\varepsilon\delta
=
\mathcal R_\varepsilon(\partial_t\delta)
\longrightarrow
\partial_t\delta
\quad\text{in }L^p(I\times\omega).
\]

\item
For every $\delta$ as above,
\[
\|\mathcal R_\varepsilon\delta\|_{L^\infty(I\times\omega)}
\leq
\|\delta\|_{L^\infty(I\times\omega)}.
\]
\end{enumerate}
\end{lemma}

\begin{proof}
We extend $\delta$ periodically in $t$ and $\theta$ and by zero in
the $z$-direction outside $[0,L]$. We denote the resulting extension
by $E\delta$. Since $\delta|_{z=0,L}=0$, if
$\delta\in C^{0,\gamma}(\overline I\times\overline\omega)$, then
$E\delta$ remains H\"older continuous with the same exponent.

Let
\[
\tau\in C_c^\infty((-1,1)),
\qquad
\varphi\in C_c^\infty(B_1(0))
\]
be nonnegative functions satisfying
\[
\int_{\mathbb R}\tau(s)\,ds=1,
\qquad
\int_{\mathbb R^2}\varphi(x)\,dx=1.
\]
Set
\[
\tau_\varepsilon(s)
:=
\frac1\varepsilon\tau\left(\frac{s}{\varepsilon}\right),
\qquad
\varphi_\varepsilon(x)
:=
\frac1{\varepsilon^2}
\varphi\left(\frac{x}{\varepsilon}\right).
\]

Choose $\alpha_\varepsilon\in C_c^\infty((0,L))$ such that
\[
0\leq\alpha_\varepsilon\leq1,
\]
and
\[
\alpha_\varepsilon(z)=
\begin{cases}
0,
&\operatorname{dist}(z,\{0,L\})\leq\varepsilon,\\
1,
&\operatorname{dist}(z,\{0,L\})\geq2\varepsilon.
\end{cases}
\]
We may additionally choose $\alpha_\varepsilon$ so that
\[
\|\partial_z^j\alpha_\varepsilon\|_{L^\infty(0,L)}
\leq C_j\varepsilon^{-j},
\qquad j\geq1.
\]
We first define the time regularization
\[
\mathcal T_\varepsilon\delta
:=
\tau_\varepsilon*_t E\delta.
\]
We then regularize in the spatial variables by setting
\[
S_\varepsilon\delta
:=
\varphi_\varepsilon*_{\theta,z}
\left(\mathcal T_\varepsilon\delta\right).
\]
Finally, we define
\[
\mathcal R_\varepsilon\delta(t,\theta,z)
:=
\alpha_\varepsilon(z)S_\varepsilon\delta(t,\theta,z).
\]
Since the extension is periodic in $t$ and $\theta$, so is
$\mathcal R_\varepsilon\delta$. Moreover,
$\mathcal R_\varepsilon\delta$ is identically zero whenever
$\operatorname{dist}(z,\{0,L\})\leq\varepsilon$. Thus,
\[
\left.\mathcal R_\varepsilon\delta\right|_{z=0,L}
=
\left.\partial_z\mathcal R_\varepsilon\delta\right|_{z=0,L}
=0,
\]
which proves (i).

Suppose now that
$\delta\in C^{0,\gamma}
(\overline I\times\overline\omega)$. Then
\begin{align*}
\|\mathcal R_\varepsilon\delta-\delta\|_{L^\infty}
&\leq
\|S_\varepsilon\delta-\delta\|_{L^\infty}
+
\|(1-\alpha_\varepsilon)\delta\|_{L^\infty}.
\end{align*}
The standard estimate for mollification gives
\[
\|S_\varepsilon\delta-\delta\|_{L^\infty}
\leq C\varepsilon^\gamma.
\]
Furthermore, since $\delta=0$ at $z=0,L$,
\[
\begin{aligned}
\|(1-\alpha_\varepsilon)\delta\|_{L^\infty}
&\leq
\sup_{\substack{(t,\theta)\in I\times(0,2\pi)\\
\operatorname{dist}(z,\{0,L\})<2\varepsilon}}
|\delta(t,\theta,z)|\\
&\leq C\varepsilon^\gamma.
\end{aligned}
\]
This proves (ii).

For (iii), since $\delta$ has zero trace at $z=0,L$, its zero
extension belongs to $W^{1,q}$, and the standard mollification
properties give
\[
S_\varepsilon\delta
\longrightarrow\delta
\quad\text{in }L^p\bigl(I;W^{1,q}(\omega)\bigr).
\]
The terms arising from $\partial_z\alpha_\varepsilon$ are controlled
by the Hardy inequality in the boundary strips. More precisely,
\[
\left\|
(\partial_z\alpha_\varepsilon)S_\varepsilon\delta
\right\|_{L^p(I;L^q(\omega))}
\leq
C
\left\|
\partial_z\delta
\right\|_
{L^p(I;L^q(\{\operatorname{dist}(z,\{0,L\})<3\varepsilon\}))},
\]
and the right-hand side converges to zero as
$\varepsilon\to0$. This proves (iii).

Since $\alpha_\varepsilon$ is independent of time and convolution
commutes with temporal differentiation,
\[
\partial_t\mathcal R_\varepsilon\delta
=
\mathcal R_\varepsilon(\partial_t\delta).
\]
The approximation-of-the-identity property in $L^p$, together with
$\alpha_\varepsilon\to1$ almost everywhere and
$0\leq\alpha_\varepsilon\leq1$, gives
\[
\mathcal R_\varepsilon(\partial_t\delta)
\longrightarrow\partial_t\delta
\quad\text{in }L^p(I\times\omega),
\]
which proves (iv).

Finally, the kernels are nonnegative and normalized, and
$0\leq\alpha_\varepsilon\leq1$. Therefore,
\[
|\mathcal R_\varepsilon\delta(t,\theta,z)|
\leq
\|\delta\|_{L^\infty(I\times\omega)},
\]
which proves (v).
\end{proof}
}

For a given $\mathbf{v}\in L_{per}^{2}\left(I;L^{2}\left(\mathbb{R}^{3}\right)\right)$ we may proceed analogously to define $\mathcal{R}_{\varepsilon}\mathbf{v}$.

    
\begin{remark}
In case $\delta$ as above is not time periodic and the initial value $\delta(0,\cdot)=\delta_{0}(\cdot)$ is given we can extend it in time from $[0,T]$ to a larger interval $[-\eps, T+\epsilon]$ by $\delta\left(t,\cdot\right)=\begin{cases}
\delta\left(-t,\cdot\right) & t\in\left[-\epsilon,0\right]\\
\delta\left(t,\cdot\right) & t\in\left[0,T\right]\\
\delta\left(2T-t,\cdot\right) & t\in\left[T,T+\epsilon\right]
\end{cases}$ and then proceed analogously.
    
\end{remark}
\subsection{The Kakutani-Glicksberg-Fan set-valued fixed point theorem}
The following result can be found in \cite{Granas-Dug}:
\begin{theorem}\label{thm:kakutani}
    Let $C$ be a convex subset of a normed vector space $Z$ and $F:C \mapsto \mathcal{P}(C)$ a set-valued mapping, where $\mathcal{P}(C)$ denotes the subsets of $C$. We assume the following:
    \begin{enumerate}[(i)]
        \item $F$ is upper-semicontinuous; or, equivalently, it has the \emph{closed-graph} property: if $c_n \to c$ and $z_n \in F(c_n)$ with $z_n \to z$ then $z \in F(c)$
        \item $F(C)$ is contained in a compact subset of $C$
        \item For all $z\in C$, $F(z)$ is non-empty, convex and compact.
    \end{enumerate}
 Then $F$ has a fixed point, that is there exists $c_0 \in C$ with $c_0 \in F(c_0)$  
\end{theorem}
\section*{Statements and Declarations}

\subsection*{Data Availibility} Data sharing is not applicable to this article as no datasets were generated or analyzed during the current study.

\subsection*{Conflict of interest}
 On behalf of all authors, the corresponding author states that there is no conflict of interest.

\subsection*{Acknowledgements}
A.R is supported by the Grant RYC2022-036183-I funded by MICIU/AEI/10.13039/501100011033 and by ESF+. A.R, C.M have been partially supported by the Basque Government through the BERC 2022-2025 program and by the Spanish State Research Agency through BCAM Severo Ochoa CEX2021-001142-S and through project PID2023-146764NB-I00 funded by MICIU/AEI/10.13039/501100011033 and cofunded by the European Union.

We would like to thank the referees for their careful reading and comments which have significantly improved the presentation of this paper. 

\bibliographystyle{plain}
\bibliography{bibliography} 

@article{brandt2026strong,
  title={Strong time-periodic solutions for a multilayered fluid--structure interaction system with nonlinear coupling},
  author={Brandt, Felix and M{\^\i}ndril{\u{a}}, Claudiu and Roy, Arnab},
  journal={Nonlinearity},
  volume={39},
  number={3},
  pages={035018},
  year={2026},
  publisher={IOP Publishing}
}

@article{bonheure2024global,
  title={Global weak solutions to a time-periodic body-liquid interaction problem},
  author={Bonheure, Denis and Galdi, Giovanni P},
  journal={Annales de l'Institut Henri Poincar{\'e} C},
  volume={42},
  number={6},
  pages={1509--1541},
  year={2024}
}

@article{bociu2021multilayered,
  title={Multilayered poroelasticity interacting with Stokes flow},
  author={Bociu, Lorena and Canic, Suncica and Muha, Boris and Webster, Justin T},
  journal={Siam journal on mathematical analysis},
  volume={53},
  number={6},
  pages={6243--6279},
  year={2021},
  publisher={SIAM}
}

@article{mindrilua2026existence,
  title={Existence of weak solutions for incompressible fluid-Koiter shell interactions with Navier slip boundary condition},
  author={M{\^\i}ndril{\u{a}}, Claudiu and Roy, Arnab},
  journal={arXiv preprint arXiv:2602.20016},
  year={2026}
}

@article{macha2022existence,
  title={Existence of a weak solution to a nonlinear fluid-structure interaction problem with heat exchange},
  author={M{\'a}cha, V{\'a}clav and Muha, Boris and Ne{\v{c}}asov{\'a}, {\v{S}}{\'a}rka and Roy, Arnab and Trifunovi{\'c}, Sr{\dj}an},
  journal={Communications in partial differential equations},
  volume={47},
  number={8},
  pages={1591--1635},
  year={2022},
  publisher={Taylor \& Francis}
}

@article{maity2021existence,
  title={Existence of strong solutions for a system of interaction between a compressible viscous fluid and a wave equation},
  author={Maity, Debayan and Roy, Arnab and Takahashi, Tak{\'e}o},
  journal={Nonlinearity},
  volume={34},
  number={4},
  pages={2659},
  year={2021},
  publisher={IOP Publishing}
}

@article{maity2020maximal,
  title={Maximal-in-time existence and uniqueness of strong solution of a 3D fluid-structure interaction model},
  author={Maity, Debayan and Raymond, Jean-Pierre and Roy, Arnab},
  journal={SIAM Journal on Mathematical Analysis},
  volume={52},
  number={6},
  pages={6338--6378},
  year={2020},
  publisher={SIAM}
}

@article{cheng2010interaction,
  title={The interaction of the 3D Navier--Stokes equations with a moving nonlinear Koiter elastic shell},
  author={Cheng, C.H. Arthur and Shkoller, Steve},
  journal={SIAM journal on mathematical analysis},
  volume={42},
  number={3},
  pages={1094--1155},
  year={2010},
  publisher={SIAM}
}

@article{grandmont2016existence,
  title={Existence of global strong solutions to a beam--fluid interaction system},
  author={Grandmont, C{\'e}line and Hillairet, Matthieu},
  journal={Archive for Rational Mechanics and Analysis},
  volume={220},
  number={3},
  pages={1283--1333},
  year={2016},
  publisher={Springer}
}

@article{beirao2004existence,
  title={On the existence of strong solutions to a coupled fluid-structure evolution problem},
  author={Beir{\~a}o da Veiga, Hugo},
  journal={Journal of Mathematical Fluid Mechanics},
  volume={6},
  pages={21--52},
  year={2004},
  publisher={Springer}
}

@article{lengeler2014weak,
  title={Weak solutions for an incompressible, generalized Newtonian fluid interacting with a linearly elastic Koiter type shell},
  author={Lengeler, Daniel},
  journal={SIAM Journal on Mathematical Analysis},
  volume={46},
  number={4},
  pages={2614--2649},
  year={2014},
  publisher={SIAM}
}

@article{grandmont2008existence,
  title={Existence of weak solutions for the unsteady interaction of a viscous fluid with an elastic plate},
  author={Grandmont, C{\'e}line},
  journal={SIAM journal on mathematical analysis},
  volume={40},
  number={2},
  pages={716--737},
  year={2008},
  publisher={SIAM}
}

@article{chambolle2005existence,
  title={Existence of weak solutions for the unsteady interaction of a viscous fluid with an elastic plate},
  author={Chambolle, Antonin and Desjardins, Beno{\=\i}t and Esteban, Maria J and Grandmont, C{\'e}line},
  journal={Journal of Mathematical Fluid Mechanics},
  volume={7},
  pages={368--404},
  year={2005},
  publisher={Springer}
}

@article{vcanic2021moving,
  title={Moving boundary problems},
  author={{\v{C}}ani{\'c}, Sun{\v{c}}ica},
  journal={Bulletin of the American Mathematical Society},
  volume={58},
  number={1},
  pages={79--106},
  year={2021}
}

@Article{Boris,
 Author = {Muha, Boris and {\v{C}}ani{\'c}, Sun{\v{c}}ica},
 Title = {Existence of a solution to a fluid-multi-layered-structure interaction problem},
 FJournal = {Journal of Differential Equations},
 Journal = {J. Differ. Equations},
 ISSN = {0022-0396},
 Volume = {256},
 Number = {2},
 Pages = {658--706},
 Year = {2014},
 Language = {English},
 DOI = {10.1016/j.jde.2013.09.016},
 zbMATH = {6447075},
 Zbl = {1326.35282}
}

@Article{LR14,
 Author = {Lengeler, Daniel and  R\r{u}{\v{z}}i{\v{c}}ka, Michael},
 Title = {Weak solutions for an incompressible {Newtonian} fluid interacting with a {Koiter} type shell},
 FJournal = {Archive for Rational Mechanics and Analysis},
 Journal = {Arch. Ration. Mech. Anal.},
 ISSN = {0003-9527},
 Volume = {211},
 Number = {1},
 Pages = {205--255},
 Year = {2014},
 Language = {English},
 DOI = {10.1007/s00205-013-0686-9},
 zbMATH = {6260961},
 Zbl = {1293.35211}
}

@Article{MC15,
 Author = {Muha, Boris and {\v{C}}ani{\'c}, Sun{\v{c}}ica},
 Title = {Fluid-structure interaction between an incompressible, viscous 3D fluid and an elastic shell with nonlinear {Koiter} membrane energy},
 FJournal = {Interfaces and Free Boundaries},
 Journal = {Interfaces Free Bound.},
 ISSN = {1463-9963},
 Volume = {17},
 Number = {4},
 Pages = {465--495},
 Year = {2015},
 Language = {English},
 DOI = {10.4171/IFB/350},
 zbMATH = {6545818},
 Zbl = {1333.35207}
}

@Book{Seregin-book,
 Author = {Seregin, Gregory},
 Title = {Lecture notes on regularity theory for the {Navier}-{Stokes} equations},
 ISBN = {978-981-4623-40-7; 978-981-4623-42-1},
 Year = {2015},
 Publisher = {Hackensack, NJ: World Scientific},
 Language = {English},
 DOI = {10.1142/9314},
 zbMATH = {6373784},
 Zbl = {1314.35001}
}

@Book{Granas-Dug,
 Author = {Granas, Andrzej and Dugundji, James},
 Title = {Fixed point theory},
 FSeries = {Springer Monographs in Mathematics},
 Series = {Springer Monogr. Math.},
 ISSN = {1439-7382},
 ISBN = {0-387-00173-5},
 Year = {2003},
 Publisher = {New York, NY: Springer},
 Language = {English},
 zbMATH = {1973375},
 Zbl = {1025.47002}
}

@Book{galdi-book,
 Author = {Galdi, Giovanni P.},
 Title = {An introduction to the mathematical theory of the {Navier}-{Stokes} equations. {Steady}-state problems},
 Edition = {2nd ed.},
 ISBN = {978-0-387-09619-3; 978-0-387-09620-9},
 Year = {2011},
 Publisher = {New York, NY: Springer},
 Language = {English},
 DOI = {10.1007/978-0-387-09620-9},
 zbMATH = {5931547},
 Zbl = {1245.35002}
}

@Misc{Bog1,
 Author = {Bogovskij, M. E.},
 Title = {The solution of some problems of vector analysis related to the operators div and grad},
 Year = {1980},
 Language = {Russian},
 HowPublished = {Tr. {Semin}. {S}.{L}. {Soboleva} 1, 5-40 (1980).},
 zbMATH = {3753729},
 Zbl = {0479.58018}
}

@Article{Bog2,
 Author = {Bogovskij, M. E.},
 Title = {Solution of the first boundary value problem for the equation of continuity of an incompressible medium},
 FJournal = {Soviet Mathematics. Doklady},
 Journal = {Sov. Math., Dokl.},
 ISSN = {0197-6788},
 Volume = {20},
 Pages = {1094--1098},
 Year = {1979},
 Language = {English},
 zbMATH = {3787290},
 Zbl = {0499.35022}
}

@Article{Ciarlet05,
 Author = {Ciarlet, Philippe G.},
 Title = {An introduction to differential geometry with applications to elasticity},
 FJournal = {Journal of Elasticity},
 Journal = {J. Elasticity},
 ISSN = {0374-3535},
 Volume = {78-79},
 Number = {1-3},
 Pages = {3--201},
 Year = {2005},
 Language = {English},
 DOI = {10.1007/s10659-005-4738-8},
 zbMATH = {5014085},
 Zbl = {1086.74001}
}

@Article{BS21,
 Author = {Breit, Dominic and Schwarzacher, Sebastian},
 Title = {Navier-{Stokes}-{Fourier} fluids interacting with elastic shells},
 FJournal = {Annali della Scuola Normale Superiore di Pisa. Classe di Scienze. Serie V},
 Journal = {Ann. Sc. Norm. Super. Pisa, Cl. Sci. (5)},
 ISSN = {0391-173X},
 Volume = {24},
 Number = {2},
 Pages = {619--690},
 Year = {2023},
 Language = {English},
 DOI = {10.2422/2036-2145.202105_090},
 zbMATH = {7741031},
 Zbl = {1525.76081}
}

@Article{BS18,
 Author = {Breit, Dominic and Schwarzacher, Sebastian},
 Title = {Compressible fluids interacting with a linear-elastic shell},
 FJournal = {Archive for Rational Mechanics and Analysis},
 Journal = {Arch. Ration. Mech. Anal.},
 ISSN = {0003-9527},
 Volume = {228},
 Number = {2},
 Pages = {495--562},
 Year = {2018},
 Language = {English},
 DOI = {10.1007/s00205-017-1199-8},
 zbMATH = {6860758},
 Zbl = {1454.76065}
}

@Article{MC13,
 Author = {Muha, Boris and Cani{\'c}, Suncica},
 Title = {Existence of a weak solution to a nonlinear fluid-structure interaction problem modeling the flow of an incompressible, viscous fluid in a cylinder with deformable walls},
 FJournal = {Archive for Rational Mechanics and Analysis},
 Journal = {Arch. Ration. Mech. Anal.},
 ISSN = {0003-9527},
 Volume = {207},
 Number = {3},
 Pages = {919--968},
 Year = {2013},
 Language = {English},
 DOI = {10.1007/s00205-012-0585-5},
 zbMATH = {6148986},
 Zbl = {1260.35148}
}

@Article{MS22,
 Author = {Muha, Boris and Schwarzacher, Sebastian},
 Title = {Existence and regularity of weak solutions for a fluid interacting with a non-linear shell in three dimensions},
 FJournal = {Annales de l'Institut Henri Poincar{\'e}. Analyse Non Lin{\'e}aire},
 Journal = {Ann. Inst. Henri Poincar{\'e}, Anal. Non Lin{\'e}aire},
 ISSN = {0294-1449},
 Volume = {39},
 Number = {6},
 Pages = {1369--1412},
 Year = {2022},
 Language = {English},
 DOI = {10.4171/AIHPC/33},
 zbMATH = {7665922},
 Zbl = {1512.74027}
}

@Misc{Claudiu23,
 Author = {M{\^{\i}}ndril{\u{a}}, Claudiu and Schwarzacher, Sebastian},
 Title = {Time-periodic weak solutions for the interaction of an incompressible fluid with a linear {Koiter} type shell under dynamic pressure boundary conditions},
 Year = {2023},
 HowPublished = {Preprint, {arXiv}:2303.13625 [math.{AP}] (2023)},
 URL = {https://arxiv.org/abs/2303.13625},
 arXiv = {arXiv:2303.13625}
}

@Article{Claudiu22,
 Author = {M{\^{\i}}ndril{\u{a}}, Claudiu and Schwarzacher, Sebastian},
 Title = {Time-periodic weak solutions for an incompressible {Newtonian} fluid interacting with an elastic plate},
 FJournal = {SIAM Journal on Mathematical Analysis},
 Journal = {SIAM J. Math. Anal.},
 ISSN = {0036-1410},
 Volume = {54},
 Number = {4},
 Pages = {4139--4162},
 Year = {2022},
 Language = {English},
 DOI = {10.1137/21M1458946},
 zbMATH = {7562642},
 Zbl = {1494.35014}
}

@article{kreml,
doi = {10.1088/1361-6544/ad92f0},
url = {https://dx.doi.org/10.1088/1361-6544/ad92f0},
year = {2024},
month = {nov},
publisher = {IOP Publishing},
volume = {38},
number = {1},
pages = {015005},
author = {Ondřej Kreml and Václav Mácha and Šárka Nečasová and Srđan Trifunović},
title = {On time-periodic solutions to an interaction problem between compressible viscous fluids and viscoelastic beams},
journal = {Nonlinearity}
}

@Article{Casanova,
 Author = {Casanova, Jean-J{\'e}r{\^o}me},
 Title = {Existence of time-periodic strong solutions to a fluid-structure system},
 FJournal = {Discrete and Continuous Dynamical Systems},
 Journal = {Discrete Contin. Dyn. Syst.},
 ISSN = {1078-0947},
 Volume = {39},
 Number = {6},
 Pages = {3291--3313},
 Year = {2019},
 Language = {English},
 DOI = {10.3934/dcds.2019136},
 zbMATH = {7068999},
 Zbl = {1415.74022}
}

@Article{BaSt22,
 Author = {Bathory, Michal and Stefanelli, Ulisse},
 Title = {Variational resolution of outflow boundary conditions for incompressible {Navier}-{Stokes}},
 FJournal = {Nonlinearity},
 Journal = {Nonlinearity},
 ISSN = {0951-7715},
 Volume = {35},
 Number = {11},
 Pages = {5553--5592},
 Year = {2022},
 Language = {English},
 DOI = {10.1088/1361-6544/ac8fd8},
 zbMATH = {7604050},
 Zbl = {1504.35288}
}

@Article{CMP94,
 Author = {Conca, C. and Murat, F. and Pironneau, O.},
 Title = {The {Stokes} and {Navier}-{Stokes} equations with boundary conditions involving the pressure},
 FJournal = {Japanese Journal of Mathematics. New Series},
 Journal = {Jpn. J. Math., New Ser.},
 ISSN = {0289-2316},
 Volume = {20},
 Number = {2},
 Pages = {279--318},
 Year = {1994},
 Language = {English},
 zbMATH = {716248},
 Zbl = {0826.35093}
}

@Article{Lengeler16,
 Author = {Lengeler, Daniel},
 Title = {Weak solutions for an incompressible, generalized {Newtonian} fluid interacting with a linearly elastic {Koiter} type shell},
 FJournal = {SIAM Journal on Mathematical Analysis},
 Journal = {SIAM J. Math. Anal.},
 ISSN = {0036-1410},
 Volume = {46},
 Number = {4},
 Pages = {2614--2649},
 Year = {2014},
 Language = {English},
 DOI = {10.1137/130911299},
 zbMATH = {6365811},
 Zbl = {1307.35197}
}

@article{juodagalvyte2020time,
  title={Time periodic {N}avier--{S}tokes equations in a thin tube structure},
  author={Juodagalvyt{\.e}, Rita and Panasenko, Grigory and Pileckas, Konstantinas},
  journal={Boundary Value Problems},
  volume={2020},
  pages={1--35},
  year={2020},
  publisher={Springer}
}

@book{panasenko2024multiscale,
  title={Multiscale Analysis of Viscous Flows in Thin Tube Structures},
  author={Panasenko, Grigory and Pileckas, Konstantin},
  year={2024},
  publisher={Springer}
}

@article{galdi2014hyperbolic,
  title={Hyperbolic--Parabolic Coupling and the Occurrence of Resonance in Partially Dissipative Systems},
  author={Galdi, Giovanni Paolo and Mohebbi, Mahdi and Zakerzadeh, Rana and Zunino, Paolo},
  journal={Fluid-structure interaction and biomedical applications},
  pages={197--256},
  year={2014},
  publisher={Springer}
}

@article{mosny2024time,
  title={Time-Periodic Solutions for Hyperbolic-Parabolic Systems},
  author={Mosny, Stanislav and Muha, Boris and Schwarzacher, Sebastian and Webster, Justin T},
  journal={Accepted in J. Eur. Math. Soc. (JEMS)},
  year={2026}
}

@article{benson2024resonance,
  title={Resonance and Periodic Solutions for Harmonic Oscillators with General Forcing},
  author={Benson, Isaac and Webster, Justin T},
  journal={arXiv preprint arXiv:2407.17144},
  year={2024}
}

@article{Pironneau1986,
  author  = {Pironneau, O.},
  title   = {Conditions aux limites sur la pression pour les {\'e}quations de {S}tokes et de {N}avier-{S}tokes},
  journal = {C. R. Acad. Sci. Paris, S{\'e}r. I},
  volume  = {303},
  pages   = {403--406},
  year    = {1986}
}

@article{Pironneau1987,
  author  = {Pironneau, O.},
  title   = {{\`A} nouveau sur les {\'e}quations de {S}tokes et de {N}avier-{S}tokes avec des conditions aux limites sur la pression},
  journal = {C. R. Acad. Sci. Paris, S{\'e}r. I},
  volume  = {304},
  pages   = {23--28},
  year    = {1987}
}

@article{ConcaMuratPironneau1994,
  author  = {Conca, C. and Murat, F. and Pironneau, O.},
  title   = {The {S}tokes and {N}avier-{S}tokes equations with boundary conditions involving the pressure},
  journal = {Japan. J. Math. (N.S.)},
  volume  = {20},
  number  = {2},
  pages   = {279--318},
  year    = {1994}
}

@article{ConcaParesPironneauThiriet1995,
  author  = {Conca, C. and Par{\'e}s, C. and Pironneau, O. and Thiriet, M.},
  title   = {Navier-{S}tokes equations with imposed pressure and velocity fluxes},
  journal = {Internat. J. Numer. Methods Fluids},
  volume  = {20},
  number  = {4},
  pages   = {267--287},
  year    = {1995}
}

@article{HeywoodRannacherTurek1996,
  author  = {Heywood, J. G. and Rannacher, R. and Turek, S.},
  title   = {Artificial boundaries and flux and pressure conditions for the incompressible {N}avier-{S}tokes equations},
  journal = {Internat. J. Numer. Methods Fluids},
  volume  = {22},
  number  = {5},
  pages   = {325--352},
  year    = {1996}
}

@article{FormaggiaGerbeauNobileQuarteroni2001,
  author  = {Formaggia, L. and Gerbeau, J.-F. and Nobile, F. and Quarteroni, A.},
  title   = {On the coupling of {3D} and {1D} {N}avier-{S}tokes equations for flow problems in compliant vessels},
  journal = {Comput. Methods Appl. Mech. Engrg.},
  volume  = {191},
  number  = {6--7},
  pages   = {561--582},
  year    = {2001}
}
\end{document}